%% file: alg-posets.tex
\documentclass{amsart}
\pdfoutput=1

\usepackage{xparse}
\usepackage{mathtools}
\usepackage{amsthm}
\usepackage{amssymb}
\usepackage{tikz-cd}
\usepackage{graphicx}
\usepackage{mathabx}
\usepackage{todonotes}
\usepackage{mathscinet}
\usepackage[utf8]{inputenc}

\usepackage{booktabs}
\usepackage{tabularx}
\newcolumntype{B}{>{\centering\arraybackslash}X} 
\newcolumntype{S}{>{\centering\arraybackslash\hsize=.5\hsize}X} 
\usepackage{caption}

\usepackage[colorlinks]{hyperref}
\usepackage[capitalize]{cleveref}

\definecolor{dark-red}{rgb}{0.6,0,0}
\definecolor{dark-green}{rgb}{0,0.4,0}
\definecolor{medium-blue}{rgb}{0,0,0.5}
\hypersetup{
    colorlinks, linkcolor={dark-red},
    citecolor={dark-green}, urlcolor={medium-blue}
}

\DeclareUnicodeCharacter{012D}{\u{\i}}

\crefformat{section}{#2\S#1#3}
\crefformat{subsection}{#2\S#1#3}
\crefformat{subsubsection}{#2 \S#1#3}
\crefname{equation}{}{}

\newcommand{\Kap}{\mathrm{Kap}}
\newcommand{\ul}[1]{\underline{#1}}

\newcommand{\mc}[1]{\mathcal{#1}}
\newcommand{\mr}[1]{\mathrm{#1}}
\newcommand{\mbb}[1]{\mathbb{#1}}
\newcommand{\mbf}[1]{\mathbf{#1}}

\newcommand{\HP}{\mbf{H}}
\newcommand{\Conf}{\mbf{C}}
\newcommand{\Sym}{\mbf{S}}
\newcommand{\Hilb}{\mr{Hilb}}
\newcommand{\red}{\mathrm{red}}

\newcommand{\Coh}{\mr{Coh}}
\newcommand{\sgn}{\mr{sgn}}
\newcommand{\rk}{\mr{rk}}
\newcommand{\op}{\mr{op}}
\newcommand{\et}{\text{ét}}
\newcommand{\proet}{\text{proét}}

\DeclareMathOperator{\Cons}{Cons}
\DeclareMathOperator{\Gr}{Gr}
\DeclareMathOperator{\Sheaves}{Sh}
\DeclareMathOperator{\Id}{Id}
\DeclareMathOperator{\Fil}{Fil}
\DeclareMathOperator{\Forget}{Forget}

\newcommand{\scref}[2]{\texorpdfstring{\cref{#1}}{#2~\ref{#1}}}

\newcommand{\slt}{\mathord{<}}
\newcommand{\sle}{\mathord{\le}}

\input{theoremnumbering}

\input{macros}

\usepackage[style=ieee-alphabetic,dashed=false,safeinputenc]{biblatex}
\renewbibmacro*{doi+eprint+url}{%
  \iffieldundef{doi}{%
  \iffieldundef{eprint}{%
    \iftoggle{bbx:url}
    {\usebibmacro{url+urldate}}
    {}}{%
    \iftoggle{bbx:eprint}
      {\usebibmacro{eprint}}
      {\iftoggle{bbx:url}
      {\usebibmacro{url+urldate}}
      {}}}
  }{\printfield{doi}}
}

\begin{document}

\title{Cohomological and motivic inclusion-exclusion}
\author{Ronno Das}
\author{Sean Howe}

\begin{abstract}
We categorify the inclusion-exclusion principle for partially ordered topological spaces and schemes to  a filtration on the derived category of sheaves. As a consequence, we obtain functorial spectral sequences  that generalize the two spectral sequences of a stratified space and certain Vassiliev-type spectral sequences; we also obtain Euler characteristic analogs in the Grothendieck ring of varieties. As an application, we give an algebro-geometric proof of Vakil and Wood's homological stability conjecture for the space of smooth hypersurface sections of a smooth projective variety. In characteristic zero this conjecture was previously established by Aumonier via topological methods.

\end{abstract}

\maketitle

\tableofcontents

\section{Introduction}

In this work, we explore consequences of topological poset theory when applied to partially ordered topological spaces (pospaces) and partially ordered schemes (poschemes). In particular, we investigate some ramifications of a simplicial proof of the inclusion-exclusion formula (described in \cref{ss.intro-classical-inclusion-exclusion}) in these enriched contexts.

We reinterpret this simplicial proof as a categorified inclusion-exclusion principle in topology and algebraic geometry via rank filtrations on the derived category of sheaves (see \cref{ss.intro-sheaves}). We give a simple criterion for the nerve of a pospace or poscheme to satisfy cohomological descent (\cref{main.cd-criterion}), and obtain in these cases functorial cohomological spectral sequences attached to a rank function (\cref{main.spectral-sequences}). These spectral sequences give a common generalization of the two spectral sequences of a stratified space (cf. \cite{petersen:spectral-sequence-stratified-space}), Vassiliev-type\footnote{Indeed, we have essentially adopted Vassiliev's method of topological order complexes and run it directly in the simplicial category instead of passing to a geometric realization.} spectral sequences (as in, e.g., \cite{vassiliev:icm-discriminants, tommasi:stability, vassiliev:homology-hypersurfaces}), and other related spectral sequences that have appeared in the literature (e.g.\ \cite{banerjee}). We also give a combinatorial analog in the Grothendieck ring of varieties, \cref{maintheorem.mot-inc-exc}, which generalizes the motivic inclusion-exclusion principle of \cite{bilu-howe:mot-eul-mot-stat, vakil-wood:discriminants}.

For our main application, consider a smooth projective variety $X$ over an algebraically closed field $\kappa$. Let $\mc{L}$ be an ample line bundle on $X$ and let $V_d$ be the vector space of global sections of $\mc{L}^d$. Let $U_d \subset V_d$ be the open subscheme of sections with non-singular vanishing locus. We refine and prove Vakil--Wood's \cite{vakil-wood:discriminants} homological stabilization conjecture for $U_d$ as $d \rightarrow \infty$. To state the result, let $H_\bullet(-)$ denote either rational $\ell$-adic étale homology for $\ell \nmid \mr{char} \kappa$ or rational singular homology if $\kappa=\mbb{C}$. Note that $H_\bullet(-)$ has a canonical weight filtration \cite{bbd, deligne:weil2}, which we denote by $W$ and normalize by treating $H_i(Y)$ for $Y$ smooth projective as having degree $-i$ and weight zero. For $\kappa=\mbb{C}$, weights can be detected by mixed Hodge theory; for arbitrary $\kappa$ (including also $\kappa=\mbb{C}$) they are detected by spreading out and using Frobenius eigenvalues over finite fields.

For $\kappa$ of characteristic zero, the following is a consequence of stronger topological results of Aumonier \cite{aumonier}; we give an algebro-geometric proof valid for arbitrary~$\kappa$.

\begin{maintheorem}\label{intro.stabilization}
For $d \gg_{i} 0$, there are natural inclusion-exclusion isomorphisms
\[ \Gr_W^{-k} H_{i}(U_d) \xrightarrow{\sim} H_{i-k}(X^k)[\sgn](k),  \]
where $H_{i-k}(X^k)[\sgn](k)$ is the isotypic part for the sign character of the symmetric group permuting the coordinates, shifted and twisted into degree $-i$ and weight $-k$. \end{maintheorem}

Vakil and Wood  \cite{vakil-wood:discriminants, vakil-wood-errata} proved an analogous stabilization in the Grothendieck ring of varieties, where the limit is a special value of the inverse Kapranov zeta function. The latter is given by a convergent infinite sum in a completed Grothendieck ring, and in the introduction to \cref{s.smooth-hypersurface} we explain how the weight $-k$ part in \cref{intro.stabilization} naturally corresponds to the $k$th term in this sum. More generally, we  define motivic and sheaf-theoretic incidence algebras for the poscheme of effective zero-cycles on $X$, then explain the inversion formula and ``cohomological special values" of the inverse Kapranov zeta function as natural outputs of M\"obius inversion. To prove \cref{intro.stabilization}, we use a cohomological approximate-inclusion-exclusion formula derived from the relative poscheme of effective zero cycles of a natural resolution of the discriminant locus $V_d\backslash U_d$. This is a cohomological analog of the motivic inclusion-exclusion that arises in Vakil and Wood's proof of stabilization in the Grothendieck ring, and here it yields an $E_1$ spectral sequence. We show the sequence degenerates on $E_1$ by first comparing with a simpler complex introduced by O.~Banerjee to obtain vanishing of the $E_1$ differential and then obtain vanishing of the higher differentials by a weight argument.

\begin{remark}\label{remark.aumonier-banerjee} Our proof of homological stabilization is distinct from Aumonier's \cite{aumonier}, however, the description of the stable cohomology in loc.\ cit.\ and the definitive knowledge that certain spectral sequences must degenerate at $E_1$ played an important role in the development of the proof presented here. We are indebted also to O.~Banerjee for sharing her insights on the sign cohomology of configuration spaces as well as the existence of the complex referred to above --- the former plays a key role in our proof, while the latter led to a considerable simplification through the replacement of the rank spectral sequence with the skeletal spectral sequence. We note that Tommasi has announced closely related results (see \cite[\S 1.1]{aumonier} for a statement), and O.~Banerjee has announced a closely related spectral sequence.
\end{remark}

We now outline the remainder of the introduction:  in \cref{ss.intro-classical-inclusion-exclusion} we take a detour  to explain a categorification of the inclusion-exclusion formula for finite sets as a toy model for our computations with pospaces and poschemes. In \cref{ss.intro-sheaves} we explain the spectral sequences obtained from a suitable pospace or poscheme. In \cref{ss.intro-poschemes} we explain the decategorification to the Grothendieck ring of varieties, and in \cref{ss.outline} we give an outline of the body of the article. Further discussion of \cref{intro.stabilization} is deferred to the introduction to \cref{s.smooth-hypersurface}; the reader interested primarily in this application may wish to skip immediately to that section and return to the rest of the paper as needed.

\subsection{Categorifying the inclusion-exclusion formula}\label{ss.intro-classical-inclusion-exclusion}
The inclusion-exclusion formula for finite sets states:
\begin{equation}\label{eq.inc-exc-fin-set-simplest}
\text{ If } X= \bigcup_{i=1}^nX_i \text{ is a finite set, then } |X| = \sum_{J \subset \{1,...,n\}} (-1)^{|J|-1} \left|\bigcap_{j \in J} X_j\right|.
\end{equation}
This statement admits a topological proof: consider the poset
\[ \mc{P}:=\bigsqcup_i X_i = \{ (x, i)\, |\, x \in X_i \} \subset X \times \{1, \dots ,n\}, \; (x,i) \leq (y,j) \iff x=y \textrm{ and } i \leq j.\]
We form the \emph{nerve} $N\mc{P}$, a simplicial set whose $m$-simplices are the ordered chains $a_0 \leq a_1 \leq \dotsm \leq a_m$ of length $m+1$ in $\mc{P}$. The sum on the right of \cref{eq.inc-exc-fin-set-simplest} computes the Euler characteristic of $N\mc{P}$ by counting the non-degenerate simplices (i.e.\ the strict chains $a_0 < a_1 < \dotsm < a_m$).  However, $N\mc{P}$ is homotopic to the constant simplicial set $X$ --- because the fiber $\mc{P}_x$ is totally ordered for any point $x\in X$, we can contract all chains above $x$ to the constant chain on any choice of element in  $\mc{P}_x$. Because homotopy preserves Euler characteristic, we recover \cref{eq.inc-exc-fin-set-simplest}:
\[ |X| = \chi(X)= \chi(N\mc{P})= \sum_{J \subset \{1,...,n\}} (-1)^{|J|-1} \left|\bigcap_{j \in J} X_j \right|. \]
The argument is point by point, so this lifts to an identity in $K_0(\mr{FinSet}/X)\cong \bbZ^X$:
\begin{equation}\label{eq.indicator-inc-exc-product}
\mbf{1}_X = \sum_{J\subset \{1,..,n\}} (-1)^{|J|-1} \mbf{1}_{\bigcap_{j \in J} X_j}
\end{equation}
Note that we can rewrite the right-hand side as $\left.\left( 1 - \prod_{j=1}^n (1- \mbf{1}_{X_j} t) \right)\right|_{t=1}$ and then \cref{eq.indicator-inc-exc-product} can be established immediately by evaluating at elements $x \in X= \bigcup_{i=1}^n X_i$. Summing over $X$ then gives a proof of  \cref{eq.inc-exc-fin-set-simplest} that avoids any topology.

The advantage of the topological approach is that it leads to an interesting \emph{categorification} of the inclusion-exclusion formula: we can think of the indicator function $\mbf{1}_X$ as being the function on $X$ assigning to a point $x \in X$ the rank at $x$ of the constant local system $\bbQ$ on $X$. More generally,  writing $D_c(X,\bbQ)$ for the bounded derived category of complexes of $\bbQ$-sheaves on $X$ (a $\bbQ$-sheaf on $X$ is just the choice of a $\mbb{Q}$-vector space $\mc{V}_x$ for each $x \in X$) with finite dimensional cohomology sheaves, we have
\[ K_0(D_c(X,\bbQ)) \xrightarrow{\sim} \bbZ^{X}, \qquad K \mapsto  \Big( x \mapsto \sum_i (-1)^i \dim_{\bbQ} H^i(K)_x \Big) \]
identifying $\bbQ = \bbQ[0]$ with $\mbf{1}_X$. The simplicial homotopy described above gives
\[ \bbQ[0] \cong C^\bullet(N\mc{P}, \bbQ) \textrm{ in } D_c(X,\bbQ) \]
where $C^\bullet(N\mc{P}, \bbQ)$ denotes the complex whose fiber at $x$ is the simplicial cochain complex $C^\bullet(N\mc{P}_x, \bbQ)$. This is a categorification of the inclusion-exclusion formula because $C^\bullet(N\mc{P}, \bbQ)$ can be equipped with a filtration $F^\bullet C^\bullet(N\mc{P}, \bbQ)$ such that
\[ \mbf{1}_X = \chi_X(\bbQ[0])=\chi_X( C^\bullet(N\mc{P}, \bbQ) ) = \sum_i \chi_X( \Gr^iC^\bullet(N\mc{P}, \bbQ))\]
realizes \cref{eq.indicator-inc-exc-product}. Indeed, any increasing filtration of the simplicial set $N\{1,\ldots,n\}\cong\Delta^{n-1}$ induces a filtration $F_\bullet N\mc{P}$, and we can then define
\[ F^i C^\bullet(N\mc{P}, \bbQ) = \ker C^\bullet(N\mc{P}, \bbQ) \xrightarrow{\mr{restriction}} C^\bullet(F_iN\mc{P}, \bbQ). \]
Noting that the non-empty subsets $J \subset \{1,\ldots,n\}$ correspond exactly to the non-degenerate simplices of $N\{1,\ldots,n\}$, we find that
\[ \chi_X(\Gr^i C^\bullet(N(\mc{P}), \bbQ)) = \sum_{J \in  F_{i+1}N\mc{P} \backslash F_{i}N\mc{P}} (-1)^{|J|-1} \mbf{1}_{\bigcap_{j \in J} X_J}. \]
In particular, if we filter $N\{1,\ldots,n\}\cong\Delta^{n-1}$ by first adding in all of the zero-simplices one at a time, then all of the one-simplices, etc., each term in \cref{eq.indicator-inc-exc-product} will correspond to exactly one graded piece of the complex. There are also other interesting ways of grouping the terms corresponding to filtrations: for example, the skeletal filtration will group the terms by $|J|$, while the filtration $\Delta^0 \subset \Delta^1 \subset \Delta^2 \subset \dotsb \subset \Delta^{n-1}$ (corresponding via nerves to the poset filtration $\{1\} \subset \{1,2\} \subset \dotsb \subset \{1,\ldots,n\}$) will group terms according to the largest element contained in $J$.

None of these sheaf-theoretic constructions depend on starting with the complex $\bbQ[0]$, and in fact what we have really categorified is the decomposition of \emph{multiplication} by $\mbf{1}_X$ --- in other words, we have factored the identity functor on $D_c(X, \bbQ)$ through the filtered derived category $DF^+(X,\bbQ)$ in a way compatible with multiplication by the two sides of \cref{eq.indicator-inc-exc-product} after taking Euler characteristics.

\subsection{Inclusion-exclusion for sheaves}\label{ss.intro-sheaves}
In the previous section, we obtained the poset $\mc{P}$ fibered over the finite set $X$ from a representation of $X$ as a union of subsets. However, these origins were  immaterial to the reasoning, and only arose in the interpretation of the final formulas. We now generalize:  A \emph{pospace} (resp.\ \emph{poscheme}) $\mc{P}$ over a topological space (resp.\ scheme) $X$ is a  topological space (resp.\ scheme) over $X$ equipped with a closed partial ordering relation $\mathord{\leq}_{\mc{P}} \subset \mc{P} \times_X \mc{P}$. In the introduction, we assume all pospaces/poschemes $\mc{P}/X$ are \emph{split}, i.e.\ that the diagonal $\Delta_{\mc{P}}$ is clopen in $\sle_{\mc{P}}$. We say $\mc{P}/X$ is \emph{proper} if $\mc{P}\rightarrow X$ is proper. We say $\mc{P}/X$ is \emph{ranked} if it is equipped with a strictly increasing map $\mr{rk}:\mc{P} \rightarrow \mbb{Z}$.

Given a pospace (resp.\ poscheme) $\mc{P}/X$, we form its nerve $N\mc{P}$ (or order complex), the simplicial topological space (resp.\ simplicial scheme) whose $m$-simplices are the space (resp.\ scheme) of ordered chains of length $m+1$ in $\mc{P}$. We wish to compute the cohomology of sheaves on $X$ by pullback to $N\mc{P}$, so that we can exploit filtrations on the nerve in order to obtain filtrations on this cohomology. This is possible when $N\mc{P}/X$ is of cohomological descent, and we now state a    criterion for this in terms of the partial ordering. A center (or weak center in the scheme-theoretic case) is a section $p$ such that $\mc{P} = \{\leq p\} \cup \{ \geq p \}$ (see \cref{defn:center}).

\begin{maintheorem}[Fiberwise center descent criterion; see also \cref{body.cd-criterion}]\label{main.cd-criterion}\hfill
\begin{enumerate}
\item If $\mc{P}/X$ is a proper {\bf pospace} and $\mc{P}_x$ admits a center for all $x \in X$, then $N\mc{P}/X$ is of cohomological descent for abelian sheaves on $X$.
\item If $\mc{P}/X$ is a proper {\bf poscheme} and $\mc{P}_x$ admits a weak center for every geometric point $x: \Spec \kappa \rightarrow X$ ($\kappa$ algebraically closed), then
\begin{enumerate}
\item $N\mc{P}/X$ is of cohomological descent for étale torsion sheaves on $X$.
\item If $X$ is furthermore Noetherian and $L$ is an algebraic extension of $\mbb{Q}_\ell$ for some $\ell$ invertible on $X$, then $\mc{P}/X$ is of cohomological descent for constructible $\mc{O}_L$ or $L$-sheaves on $X$.
\end{enumerate}
\end{enumerate}
\end{maintheorem}

\begin{remark}\label{remark.cd-complexes}
More generally, one has cohomological descent on suitable complexes --- see \cref{body.cd-criterion} for a precise statement. We note that in (2)-(b) we use the formalism of \cite{bhatt-scholze:pro-etale} for  $\ell$-adic sheaves; this is useful in order to construct filtrations by working in the derived category of an abelian category of sheaves.
\end{remark}

\begin{remark}
\cref{main.cd-criterion} is a version of Quillen's Theorem A and of the Vietoris mapping theorem in that it converts fiberwise contractibility to a global equivalence. The topological case  can be deduced from the Vietoris mapping theorem.
\end{remark}

\begin{remark}In the toy model (\cref{ss.intro-classical-inclusion-exclusion}),  we used the existence of a maximum on each fiber to deduce that the nerve was contractible --- maxima and minima are examples of centers. In the topological/scheme theoretic setting, one can expect to obtain a useful fiberwise criterion only when $\mc{P}/X$ is proper, so that we can invoke proper base change. With properness imposed, any condition on the fibers that gives (weak) contractibilty of the nerve would suffice, and in \cref{ss.contractibility-criterion} we develop some more general tools for verifying contractibility using the partial ordering.\end{remark}

A finite filtration of $N\mc{P}$ by closed simplicial subspaces/subschemes induces a support filtration on sheaves on $N\mc{P}$. When $N\mc{P}/X$ is of cohomological descent, this in turns induces a derived filtration on sheaves on $X$ (just as in the toy model of \cref{ss.intro-classical-inclusion-exclusion}, the filtration can only be seen after replacing $\mc{F}$ with a resolution obtained by working on $N\mc{P}$). We are most interested in the case where the filtration on $N\mc{P}$ is inherited from some structure on $\mc{P}$ itself: in particular, if $\mc{P}/X$ is finitely ranked, then we can consider the rank filtration of $N\mc{P}$ by $N(\rk \leq i)$ and the induced derived filtration on sheaves on $X$. In this case, the cohomology of the graded parts can be interpreted as the reduced cohomology complexes of nerves of certain auxiliary pospaces/poschemes constructed from $\mc{P}$; in practice, these reduced cohomology complexes are computable. Concretely, they are defined as follows: let $\mc{P}^+$ be the pospace/scheme obtained by adjoining a disjoint minimum section $-\infty$ over $X$. We choose an extension of the rank function to $\mc{P}^+$, and let $\mc{P}^+_r:=\rk^{-1}(r)$. In this setting, we define (up to quasi-isomorphism) a complex of sheaves $\tilde{C}(-\infty, \mc{P}^+_r, \mc{F})$ over $\mc{P}^+_r$ that at each geometric point $p \in \mc{P}^+_r$ computes the reduced cohomology of $\mc{F}$ on $N(-\infty,p)$. By convention as in \cite{petersen:spectral-sequence-stratified-space}, if $r=\rk(-\infty)$ so $\mc{P}^+_r\cong X$ then  $\tilde{C}(-\infty, \mc{P}^+_r, \mc{F}):=\mc{F}[2]$. We write $\mc{P}_{r}=\mc{P}^+_r \cap \mc{P}$, which is equal to $\mc{P}^+_r$ unless $r=\rk(-\infty)$ in which case $\mc{P}_r=\emptyset$.

\begin{maintheorem}[See also \cref{body.filtration-functor,body.spectral-sequences}] \label{main.spectral-sequences}
Suppose $X$ is a topological space and $\mc{F}$ is as sheaf on $X$, or $X$ is a separated finite type scheme over an algebraically closed field and $\mc{F}$ is a pro-étale sheaf on $X$. Suppose $\pi:\mc{P} \rightarrow X$ is a proper finitely ranked pospace/poscheme and let $Z=\pi(\mc{P})$ with complement $U:=X\backslash Z$. Suppose also that  $\mc{P}|_Z$ is of cohomological descent for $\mc{F}|_Z$ (e.g.\ it meets the criterion of \cref{main.cd-criterion}).  Then there are spectral sequences, functorial in $\mc{F}$,
\begin{enumerate}
\item For compactly supported cohomology on $Z$:
\[ E_1^{p,q}=H_c^{p+q-1}(\mc{P}_{p+1}, \tilde{C}(-\infty, \mc{P}_{p+1}^+, \mc{F})) \Rightarrow H^{p+q}_c(Z,\mc{F}). \]
\item For compactly supported cohomology on $U$:
\[ E_1^{p,q} = H^{p+q-2}_c(\mc{P}_p^+, \tilde{C}(-\infty, \mc{P}^+_p, \mc{F}))   \Rightarrow H^{p+q}_c(U, \mc{F}).	 \]
\end{enumerate}
\end{maintheorem}

The two versions (1) and (2) are related by the long exact sequence in compactly supported cohomology and are essentially equivalent. We state them individually for convenience in applications.

\begin{remark}\label{remark.filtration-functor-spectral-sequences}
As in the toy model, these spectral sequences arise from a functorial filtration on a suitable category of complexes and this is made precise in \cref{body.filtration-functor}. Thus we can replace $\mc{F}$ with a complex of sheaves and $H^\bullet_c$ with $Rf_*$ or $Rf_!$ for a general morphism $f: X \rightarrow S$ --- see  \cref{body.spectral-sequences}.
\end{remark}

\begin{remark} In the topological case, if we take geometric realizations, then the terms are computing the compactly supported cohomology of the complement of one stratum of $|N\mc{P}|$ in the next, and the spectral sequence above arises from interpreting these complements using the geometric realizations of the nerves of auxiliary pospaces. In particular, for constant coefficients $A$, we can rewrite as:
\begin{enumerate}
\item $E_1^{p, q} = H_c^{p + q}(|N(-\infty, \mc{P}_{p+1}]| - |N(-\infty, \mc{P}_{p+1})|; A) \Rightarrow H_c^{p+q}(Z; A)$.
\item $E_1^{p, q} \Rightarrow H_c^{p+q}(U; A)$, with
\[E_1^{p,q} =\begin{cases*} 0 & if $p < 0$,\\
H_c^q(X; A) & if $p=0$, \\
H_c^{p + q - 1}(|N(-\infty, \mc{P}_p]| - |N(-\infty, \mc{P}_p)|; A) & if $p > 0$.\end{cases*}\]
\end{enumerate}
Here $(-\infty, \mc{P}_r]=(\mc{P}_{\le r} \times \mc{P}_r) \cap \sle_{\mc{P}}$, a pospace over $\mc{P}_r$. For $p \in \mc{P}_r$ the nerve of the fiber $(-\infty, p]$ is a cone with vertex $p$ and base $N(\infty, p)$, so the fiber of $|N(-\infty, \mc{P}_r]| - |N(-\infty, \mc{P}_r)|$ over $p\in \mc{P}_r$ is the \emph{open cone} on $|N(-\infty, p)|$. In applications, the space $|N(-\infty, \mc{P}_r]| - |N(-\infty, \mc{P}_r)|$ is often a bundle over $\mc{P}_r$. For instance, for the configuration poscheme we study in \cref{s.conf-zc-hilb}, it is a disk (open simplex) bundle over $\mc{P}_r$ with $\sgn$ monodromy, which gives a topological explanation for the ubiquity of the $\sgn$ representation in \cref{s.conf-zc-hilb} and in our application to Vakil--Wood's conjecture (see \cref{s.smooth-hypersurface}).
\end{remark}

We conclude this subsection with some applications of \cref{main.spectral-sequences}.

\subsubsection{The first spectral sequence of a stratified space}
\label{sss.stratified-space}
Suppose a space/scheme $X$ is a finite union of disjoint locally closed sets $S_\alpha$, such that the closure $Z_{\alpha} := \overline{S_\alpha}$ is a union of strata $S_\beta$. The index set is a poset under the order $\alpha \geq \beta$ if $Z_\alpha \supseteq Z_\beta$.

We consider the pospace/poscheme $\mc{P}=\bigsqcup_{\alpha} Z_\alpha \rightarrow X$
where the relation is given by $z_\alpha \geq z_\beta$ if $z_\alpha = z_\beta$ as elements of $X$ and $Z_\alpha \supseteq Z_\beta$. It is a finite disjoint union of closed immersions into $X$, thus proper. The fiber over a geometric point $z$ has an isolated minimum, given by $z_\beta$ for $z \in S_\beta$. There is a map from $\mc{P}$ to the index poset, so a rank on the latter induces a rank on $\mc{P}$. Thus \cref{main.spectral-sequences}-(1) applies with $Z=X$ and yields (see \cref{sss.stratified-space-body}) the spectral sequence of a stratified space
\[ E_1^{p,q}=\bigoplus_{\rk(\alpha)=p+1}H^{p+q}_c(S_\alpha, \mc{F}) \Rightarrow H^{p+q}_c (X, \mc{F}). \]

\subsubsection{The Petersen spectral sequence of a stratified space}
\label{sss.petersen}
Here we adopt the same setup as above, but \emph{reverse the ordering}: now $\alpha \leq \beta$ if $Z_\alpha \supseteq Z_\beta$. We also assume there is a unique stratum $Z_\eta = X$, which we remove from $\mc{P}$. Now after fixing a rank function on the indices for this reversed order, we have
\[ \mc{P}_{p} = \bigsqcup_{\mr{rk}(\alpha)=p+1} Z_\alpha \qquad \textrm{ and } \qquad (-\infty, \mc{P}_{p+1})  = \bigsqcup_{\mr{rk}(\alpha)=p+1} Z_{\alpha} \times (\eta, \alpha) \]
Thus the pospaces/poschemes appearing in $E_1^{p,q}$ in \cref{main.spectral-sequences} are locally constant. In particular, if $X$ is a topological space or a variety over an algebraically closed field, then applying \cref{main.spectral-sequences}-(2), we obtain Petersen's \cite{petersen:spectral-sequence-stratified-space} spectral sequence:
\[ E_1^{p,q}= \bigoplus_{\rk(\alpha)=p} H^{p+q-2}_c(Z_\alpha, \tilde{C}^\bullet(\eta, \alpha, \mbb{Z} ) \otimes \mc{F}) \Rightarrow H^{p+q}_c(S_\eta, \mc{F}). \]
Taking $\star=*$ in \cref{body.spectral-sequences}-(2) gives the variant with supports of \cite{petersen:spectral-sequence-stratified-space}.

\subsubsection{Approximate inclusion-exclusion and Vassiliev-type sequences}
Given a proper map of varieties $f:Z \rightarrow X$, there are several natural poschemes one can construct that resolve the image $f(Z)$. We investigate some of these in \cref{s.conf-zc-hilb}, with an emphasis on the poscheme of relative effective zero cycles; the latter gives a resolution of $f(Z)$ when $f$ is finite, but is useful more generally because it can be used to give an approximate inclusion-exclusion principle when the finite locus of $f$ has complement of high codimension (see \cref{theorem.approx-ie}). The terms of the rank spectral sequence of \cref{main.spectral-sequences} in this case are given by the compactly supported sign cohomology of relative configuration spaces, and indeed the spectral sequence is closely related to the stable part of Vassiliev spectral sequences for discriminant loci appearing in \cite{tommasi:stability}. In the case where the map is not finite or well-approximated by a finite map, one can use the full Hilbert poscheme to get a full Vassiliev spectral sequence closely related to the sequences appearing in, e.g., \cite{vassiliev:homology-hypersurfaces}. We refer the reader to the introduction of \cref{s.conf-zc-hilb} for more on these points.

\subsubsection{Other applications} \cref{main.spectral-sequences} also recovers O.~Banerjee's \cite[Theorem~1]{banerjee} spectral sequence of a symmetric semisimplicial filtration --- see \cref{sss.banerjee}.

\subsection{Decategorifications and Grothendieck rings}\label{ss.intro-poschemes}

Recall that in the toy model of \cref{ss.intro-classical-inclusion-exclusion} one could also interpret the classical inclusion-exclusion formula as an identity in the (combinatorial) Grothendieck ring of the category of finite sets, which was moreover equal to the Grothendieck ring of constructible sheaves.

If $X$ is a Noetherian scheme, then this story is enriched: for $\ell$ invertible on $X$, we can form the Grothendieck ring of sheaves $K_0(D_\mr{Cons}(X,\mbb{Q}_\ell))$. In the setting of \cref{main.cd-criterion}, the filtration then induces corresponding Euler characteristic identities in $K_0(D_\mr{Cons}(X,\mbb{Q}_\ell))$. We can also form the (modified --- see \cref{ss.decat}) Grothendieck ring of varieties $K_0(\Var/X)$, and there is a compactly supported cohomology map
\[ K_0(\Var/X) \rightarrow K_0(D_\mr{Cons}(X,\mbb{Q}_\ell)). \]
Unlike the case of finite sets, however, it is not typically an isomorphism, so we have two different decategorifications (one combinatorial and one abelian). There is also an inclusion-exclusion principle in the combinatorial decategorification $K_0(\Var/X)$ lifting the inclusion-exclusion principle in the abelian decategorification $K_0(D_\mr{Cons}(X,\mbb{Q}_\ell))$ along this map: For a simplicial scheme $S_\bullet/X$, let
\[ \chi(S_\bullet) := \sum_{k=0}^\infty (-1)^k [\mc{S}^\circ_k] \in K_0(\Var/X)\]
where $\mc{S}^\circ_k$ denotes the subscheme of non-degenerate $k$-simplices and we only consider  $S_\bullet$ such that these spaces are of finite type over $X$ and this sum is finite. When $S_\bullet=N\mc{P}$ for a poscheme $\mc{P}/X$, the non-degenerate simplices are the strictly ordered chains, and the sum is finite exactly when $\mc{P}$ is of bounded length, i.e.\ there is some bound on the length of strict chains.  We show:

\begin{maintheorem}[Motivic inclusion-exclusion]\label{maintheorem.mot-inc-exc}
Let $X$ be a Noetherian scheme and let $\mc{P} /X$ be a poscheme over $X$ of finite type that admits weak centers over all geometric fibers. Then, in $K_0(\Var/X)$, $ [X/X] = \chi(N\mc{P}).$ If $\mc{P}$ is equipped with a rank function then this is furthermore equal to
\[ - \sum_{p} \tilde{\chi}(N(-\infty, \mc{P}_{p})) = -\sum_p \bigg( \bigg( \sum_{k=0}^\infty (-1)^k  [N(-\infty, \mc{P}_{p})^\circ_k]\bigg) - [\mc{P}_p]  \bigg) \in K_0(\Var/X)\]
\end{maintheorem}

\begin{remark} There is  no properness assumption in \cref{maintheorem.mot-inc-exc}. The rank just rearranges the terms as in the toy model \cref{ss.intro-classical-inclusion-exclusion}, but this is sometimes quite useful.
\end{remark}

This generalizes the motivic inclusion-exclusion principle of \cite{bilu-howe:mot-eul-mot-stat, vakil-wood:discriminants}, which is recovered by applying \cref{maintheorem.mot-inc-exc} to the configuration poscheme $\Conf^\bullet_X(Z)$ of a morphism $Z \rightarrow X$ (see \cref{theorem.approx-ie} for a closely related approximate motivic inclusion-exclusion principle along with a cohomological analog).

\subsection{Outline}\label{ss.outline} In \cref{s.preliminaries} we give a brief overview of simplicial spaces/schemes and filtered derived categories, recalling precisely the points that we need later on. In particular, we explain how to translate a geometric filtration into a support filtration on the derived category of sheaves on a simplicial space/scheme and establish some useful lemmas related to this procedure. In \cref{s.pospaces-poschemes} we define and study elementary properties of pospaces/poschemes and the nerve construction.

With these preliminaries established, in \cref{s.cohomological-inclusion-exclusion} we prove our cohomological inclusion-exclusion results, \cref{main.cd-criterion} and \cref{main.spectral-sequences}. The key ideas as described in \cref{ss.intro-sheaves} come from elementary poset topology, so that the main work is just to be careful with the technical details in this enriched setting. In \cref{s.mot-inc-exc} we prove our motivic inclusion-exclusion result, \cref{maintheorem.mot-inc-exc}; this follows the same general pattern as the cohomological case and is in some sense the simpler of the two, but some extra work is needed to translate the contracting homotopies used in \cref{s.cohomological-inclusion-exclusion} into identifications of enriched Euler characteristics.

In \cref{s.conf-zc-hilb}, we study the relative poscheme of effective zero cycles for a projective morphism of varieties $f:Z \rightarrow X$ --- in particular, we compute the graded components for the rank spectral sequence and the $E_1$ page of the skeletal spectral sequence, and prove an approximate inclusion-exclusion theorem (\cref{theorem.approx-ie}). These results play an important role in our application to homological stability, \cref{intro.stabilization}. Note that we systematically use divided powers in place of symmetric powers, which leads to a more technically satisfying theory in characteristic $p$.

In \cref{s.incidence-algebras} we explain how to construct motivic and sheaf theoretic incidence algebras attached to poschemes. In particular, we explain how to realize the Kapranov zeta function as an element of the reduced incidence algebra of the poscheme of effective zero cycles, then  use this to give a a new perspective on the inversion formula for the Kapranov zeta function as an instance of M\"obius inversion (that Hasse--Weil zeta functions can be realized inside incidence algebras of posets of zero-cycles is well-known and classical; see Kobin \cite{kobin} for a recent survey, which also raised the question of whether such an approach could exist for the Kapranov zeta function). Motivated by this construction, we also define cohomological special values of the inverse Kapranov zeta function; these describe the stable homology in \cref{intro.stabilization}.

Finally, in \cref{s.smooth-hypersurface} we prove \cref{intro.stabilization}. As indicated earlier, \cref{s.smooth-hypersurface} begins with a self-contained introduction that provides a much more detailed discussion of Vakil and Wood's conjecture and related work leading to \cref{intro.stabilization}.

\subsection{Acknowledgements} We thank Oishee Banerjee for helpful conversations (see \cref{remark.aumonier-banerjee}!). We thank Margaret Bilu for many helpful conversations, especially at the beginning of this work which grew in conjunction with \cite{bilu-das-howe}. Ronno Das was supported during later stages of the project by the European Research Council (ERC) under the European Union’s Horizon 2020 research and innovation programme (grant agreement No.\ 772960), as well as by the Danish National Research Foundation through the Copenhagen Centre for Geometry and Topology (DNRF151). Sean Howe was supported during a portion of the preparation of this work by the National Science Foundation under Award No. DMS-1704005.

\section{Preliminaries}\label{s.preliminaries}

\subsection{Simplicial objects}
\newcommand{\Cone}{\mathrm{Cone}}
\newcommand{\Cocone}{\mathrm{Cocone}}

\subsubsection{Definitions} We write $\Delta$ for the simplex category, which we take as the category of non-empty finite subsets of $\mbb{Z}_{\geq 0}$ with morphisms given by non-decreasing maps (this is the minimal version of $\Delta$ which allows the clean formulation of the join construction in \cref{example.simplicial-constructions}-(3) below). We write $[n]=\{0,\ldots,n\} \in \Delta$. A simplicial object in a category $\mc{C}$ is a functor $\Delta^{\op}\rightarrow \mc{C}$. We write $s\mc{C}$ for the category of simplicial objects in $\mc{C}$. We will often specify a simplicial object by specifying its restriction to the full subcategory spanned by $[n]$ for $n \geq 0$, which is canonically equivalent to $\Delta$, and given $A \in s\mc{C}$ we write $A_k:=A([k])$. We write $\Delta^k$ for the simplicial set $\mr{Hom}(\bullet, [k]).$

\begin{example}\label{example.simplicial-constructions}\hfill
	\begin{enumerate}
	\item Given an object $X \in \mc{C}$, we can form the constant simplicial object $X_\bullet \in s\mc{C}$ such that $X_k=X$ for all $k$ and all maps are $\Id_X$.
	\item If $\mc{C}$ admits coproducts then for any $A\in s\mc{C}$ and simplicial set $T \in s\mr{Set}$, we can form $A \times T \in s\mc{C}$ as in \cite[\href{https://stacks.math.columbia.edu/tag/017C}{Tag 017C}]{stacks-project}
		by the formula
	\[ (A \times T)_k=\bigsqcup_{t \in T_k} A_k. \]
		If $\mc{C}$ only admits finite coproducts this still makes sense for finite simplicial sets, in particular for $\Delta^k$.
	\item If $\mc{C}$ admits finite coproducts and products and a final object $*$, we define the join $A \star B$ of $A, B \in s\mc{C}$ by
		\[ (A \star B)_k = \bigsqcup_{j=-1}^{k} A([j]) \times B([k]\backslash[j]) \]
		where we intepret $[-1]=\emptyset$ and $A(\emptyset)=B(\emptyset)=*$. With the same hypotheses, for $A \in s\mc{C}$, we define
		\[ \Cone(A) = (*_\bullet) \star A \textrm{ and } \Cocone(A)= A \star (*_\bullet).\]
\end{enumerate}

\end{example}

\subsubsection{Homotopy} Suppose $\mc{C}$ admits finite coproducts. Then, for any simplicial object $A \in s \mc{C}$, we can form $A \times \Delta^1$ as in \cref{example.simplicial-constructions}-(2). If $f, g: A \rightarrow B$ are maps in $s\mc{C}$, a homotopy from $f$ to $g$ is a map  $h:A \times \Delta^1 \rightarrow B$ such that $\{0 \mapsto 0\}^*h=f$ and $\{0 \mapsto 1\}^*h=g$. We say $f$ and $g$ are homotopic if there is a chain of maps $f=f_1,f_2,f_3,\ldots, f_n=g$ such that for each $i$, there is a homotopy from $f_i$ to $f_{i+1}$ or from $f_{i+1}$ to $f_i$ over $X$.

A map $f: A \rightarrow B$ in $s\mc{C}$ is a homotopy equivalence if there exists  a map $g: B \rightarrow A$ such that $f\circ g$ is homotopic to $\Id_{A}$ and $g \circ f$ is homotopic to $\Id_{B}$. A functor $\mc{C}\rightarrow \mc{D}$ induces a functor $s\mc{C} \rightarrow s\mc{D}$ which preserves homotopies and thus the notion of homotopic maps and homotopy equivalences.  If $\mc{C}$ has a final object $*$, then we say $A\in s \mc{C}$ is contractible if the unique map $A \rightarrow *_\bullet$ is a homotopy equivalence.

We will mostly encounter the following special type of homotopy equivalence: we say $\iota: A \hookrightarrow B$ is a deformation retract if there is a map $r: B \rightarrow A$ such that $r \circ \iota = \Id_{A}$ and $\iota \circ r$ is homotopic to $\mr{Id}_{B}$. Note that if $A$ is contractible, then any $\iota: *_\bullet \hookrightarrow A$ is a homotopy inverse for $A \rightarrow *_\bullet$ and therefore a deformation retract.

\begin{example} If $\mc{C}$ admits finite coproducts and has a final object $*$ then, for any $A \in s\mc{C}$, $\Cone(A)$ and $\Cocone(A)$ are contractible (the homotopy inverse $\iota$ as above is given by inclusion of the vertex/tip).
\end{example}

\subsubsection{Augmentations and contractible objects}
For $X \in \mc{C}$ and $A \in s\mc{C}$, an augmentation from $A$ to $X$ is a map $\epsilon: A \rightarrow X_\bullet$. The category $s\mc{C}_{/X_\bullet}$ of simplicial objects equipped with an augmentation is canonically equivalent to $s\mc{C}_{/X} := s(\mc{C}_{/X})$, the simplical objects in the category $\mc{C}_{/X}$ of objects in $\mc{C}$ with a map to $X$. We say $A/ X_\bullet$ is contractible if it is contractible as an object of $s\mc{C}_{/X}$, i.e.\ if the map $\epsilon: A \rightarrow X_\bullet$ is a homotopy equivalence with homotopy inverse given by a section of $\epsilon$.

\subsubsection{Cosimplicial objects in abelian categories}
A cosimplicial object in $\mc{C}$ is a simplicial object in $\mc{C}^\op$. If $\mc{C}$ is abelian, there are
two natural functors from $s\mc{C}^\op$ to the category of cochain complexes in $\mc{C}$ concentrated in positive degree (see e.g.\ \cite[\href{https://stacks.math.columbia.edu/tag/0194}{Tag 0194}]{stacks-project}):
\begin{enumerate}
	\item The unnormalized cochain complex attached to $A \in s \mc{C}^\op$ is
	\[ A_0 \xrightarrow{d_1} A_1 \xrightarrow{d_2} A_2 \xrightarrow{d_3} \dotsm \]
	where $d_k = \sum_{i \in [k]} (-1)^i \delta_i$ for $\delta_i$ the morphism obtained by applying $A$ to to the standard $i$th face map $[k-1]\rightarrow [k]$ that skips the vertex $i \in [k]$.
	\item The normalized cochain complex attached to $A \in s \mc{C}^\op$ is the subcomplex obtained by replacing $A_k$ with the kernel of all degeneracy maps.
\end{enumerate}
The inclusion of the normalized complex in the unnormalized complex is a chain homotopy equivalence, and both functors send simplicial homotopies to chain homotopies.

\subsection{Sheaves on simplicial spaces/schemes}\label{ss.ssites-sheaves-descent}
For $X$ a topological space, we write $\Sheaves(X)$ for the category of abelian sheaves on $X$. For $X$ a scheme, we write $\Sheaves(X)$ for either the category of abelian étale sheaves or the category of pro-étale $\mc{O}_L$ or $L$ modules for $L/\mbb{Q}_\ell$ an algebraic extension (as defined in \cite{bhatt-scholze:pro-etale}).

If $A$ is a simplicial space or scheme, we extend these definitions in the standard way to define $\Sheaves(A)$: Concretely, we define a sheaf on $A$ to be a family of sheaves $\mc{F}_n$ on the simplex spaces $A_n$ equipped with a compatible family of maps $A(f)^* \mc{F}_n \rightarrow \mc{F}_m$ for $f:[n]\rightarrow [m]$. It can be shown that these form an abelian category that is naturally identified with the category of sheaves on a site built from $A$, so that, in particular, there are enough injectives and the standard formalism of derived categories of sheaves applies. The case of simplicial spaces is treated concretely in \cite[\href{https://stacks.math.columbia.edu/tag/09VK}{Tag 09VK}]{stacks-project}, or can be considered in parallel with the case of schemes by first passing to the corresponding simplicial site and then applying the formalism of   \cite[\href{https://stacks.math.columbia.edu/tag/09WB}{Tag 09WB}]{stacks-project}, Case  (A).

Given a simplicial space or scheme $A$ with an augmentation $\epsilon$ towards $X$, we obtain adjoint pushforward and pullback functors
\[ \epsilon_*: \Sheaves(A)\rightarrow \Sheaves(X), \epsilon^*:\Sheaves(X) \rightarrow \Sheaves(A) \]
 such that $\epsilon^*\mc{G}$ is given by the obvious system of  $\epsilon_n^*\mc{G}$ on $A_n$ and \[ \epsilon_*(\mc{F})=\mr{Eq} \left( {\epsilon_{0}}_*\mc{F}_0 \rightrightarrows  {\epsilon_{1}}_*\mc{F}_1 \right).\]
For $K \in D^+(A)$, the bounded below derived category of $\Sheaves(A)$, there is a functorial skeletal spectral sequence \cite[\href{https://stacks.math.columbia.edu/tag/0D7A}{Tag 0D7A}]{stacks-project}
\begin{equation}\label{eq.simplicial-ss}  E_1^{p,q}=R^q {\epsilon_p}_* K \Rightarrow R^{p+q}\epsilon_* K. \end{equation}
such that for each $q$ the cochain complex  $E_1^{\bullet,q}$ is the unnormalized cochain complex of the cosimplicial sheaf $[p] \mapsto R^q {\epsilon_p}_* K$ on $X$.

\begin{definition}\label{def.reduced-cohomology}
We write $\tilde{C}(A/X, \bullet)$ for the relative reduced cohomology complex functor $D^+(X) \rightarrow D^+(X)$, i.e.\ the  cone of the adjunction unit $u_{A/X}: \Id_{D^+(X)} \rightarrow R\epsilon_* \epsilon^*$, so that there is a functorial exact triangle for $K \in D^+(X)$
\[K \rightarrow R\epsilon_* \epsilon^*K\rightarrow \tilde{C}(A/X, K) \rightarrow K[1] \]
We write its cohomology sheaves as
\[ \tilde{H}^q(A/X, \bullet):=H^q(\tilde{C}(A/X, \bullet)). \]
We say $A/X$ satisfies cohomological descent on a full subcategory $D \subset D^+(X)$ if $u_{A/X}|_D$ is an isomorphism of functors, or equivalently if $\tilde{C}(A/X, K)\cong 0$  for each $K \in D$, or equivalently if $\tilde{H}^q(A/X,K)=0$ for each $K \in D$ and each $q \in \mbb{Z}$.
\end{definition}

\begin{lemma}\label{lemma:homotopy-equiv-and-contractible}
If $f: A/X \rightarrow B/X$ is a homotopy equivalence of simplicial spaces or schemes over $X$, then $f$ induces an isomorphism of functors on $D^+(X)$
\[ R{\epsilon_{A}}_* \epsilon_{A}^* \cong  R{\epsilon_{B}}_* \epsilon_{B}^*. \]
In particular, if $A/X$ is contractible, then $A/X$ satisfies cohomological descent.
\end{lemma}
\begin{proof}
We obtain a map $f^*$ between the skeletal spectral sequences \cref{eq.simplicial-ss} for $A/X$ and $B/X$. On the $q$th column $E_1^{\bullet, q}$ this is the map on unnormalized complexes coming from the map of cosimplicial sheaves
\[ f^*: \left([p]\mapsto R^q {\epsilon_{B_p}}_* \epsilon_{B_p}^* K \right)\rightarrow  \left([p]\mapsto R^q {\epsilon_{A_p}}_* \epsilon_{A_p}^* K.\right) \] Because $f$ is a homotopy equivalence, so is the induced map on cosimplicial sheaves: it comes from applying the functor on simplicial spaces/schemes over $X$ induced by the functor on spaces/schemes over $X$ sending $\pi: Y \rightarrow X$ to $R^q \pi_* \pi^* K$. A functor constructed in this way preserves homotopy equivalences.   The induced map on complexes is then a homotopy equivalence, thus an isomorphism on cohomology, so that we obtain an isomorphism of spectral sequences starting   at the $E_2$ page.
\end{proof}

\subsection{Filtered derived categories and simplicial filtrations}\label{ss.filtered-derived}

\subsubsection{Filtered derived categories} We follow \cite[\href{https://stacks.math.columbia.edu/tag/05RX}{Tag 05RX}, \href{https://stacks.math.columbia.edu/tag/015O}{Tag 015O}]{stacks-project}. To summarize: For an abelian category $\mc{A}$, we write $\mc{A}^f$ for the exact category of finitely filtered objects in $\mc{A}$. It admits exact filtered piece, graded part, and forgetful functors $F^p$, $\Gr^p, \mr{Forget}: \mc{A}^f \rightarrow \mc{A}$. For the category of sheaves on a space/scheme or simplicial space/scheme $X$, we write $DF^+(X)$ for the bounded below filtered derived category, ie. the bounded below derived category of $\Sheaves(X)^f$. The filtered piece, graded and forgetful functors induced triangulated functors $DF^+(X) \rightarrow D^+(X)$. Moreover, if $A/X$ is a simplicial space/scheme augmented towards $X$ then we have a filtered derived functor $R{\epsilon^f}_*: DF^+(A) \rightarrow DF^+(X)$ and a canonical isomorphism
\[ \Forget \circ R{\epsilon^f}_* = R\epsilon_* \circ \Forget. \]

\subsubsection{Simplicial filtrations} Let $A$ be a simplicial space/scheme. We say a map of simplicial spaces/schemes $B \rightarrow A$ is a closed (resp.\ open, resp.\ clopen) immersion if for all $k \geq 0$, $B_k \rightarrow A_k$ is a closed (resp.\ open, resp.\ clopen) immersion.

A filtration of $A$ is an increasing sequence of closed sub-simplicial spaces/schemes $\iota_i: F_i A \hookrightarrow A, i \in \mbb{Z}$. It is finite if $F_i A=\emptyset$ for $i \ll 0$ and $F_i A= A$ for $i \gg 0$. It is \emph{split} if $\iota_i$ is a clopen immersion for all $i$.

\begin{example}\label{example.cone-filtration}
For $A$ a simplicial space/scheme over $X$, $\Cone(A/X)$ (see \cref{example.simplicial-constructions}) has a natural finite split filtration with
\[ F_0 =X_\bullet \textrm{ and } F_1 \Cone(A/X)=\Cone(A/X),\]
and similarly for $\Cocone(A/X)$.
\end{example}

Attached to a filtration $F_\bullet$ of $A$ we have filtered and graded piece functors
\[ F^i: \Sheaves(A) \rightarrow \Sheaves(A), \mc{F} \mapsto F^i\mc{F} = \ker \left( \mc{F} \rightarrow {\iota_{i}}_*{\iota_{i}}^{*}\mc{F} \right) \]
\[ \Gr^i: \Sheaves(A) \rightarrow \Sheaves(A), \mc{F} \mapsto \Gr^i \mc{F} = F^i \mc{F} / F^{i+1} \mc{F}. \]
These functors are exact --- this can be checked on $m$-simplices, where it immediately reduces to the corresponding assertion for filtrations of spaces/schemes by closed subspaces/subschemes: indeed, if we write $j$ for the locally closed immersion $F_{i+1} A_m \backslash F_i A_{m}  \hookrightarrow A_m$, then there is a canonical identification of sheaves on $A_m$
\[ j_! j^* \mc{F}_m = (\Gr^i\mc{F})_m. \]
We can reinterpret this description of the graded pieces geometrically as the following useful lemma which also describes the simplicial structure:
\begin{lemma}\label{lemma.push-pull-graded} For $F_\bullet A$ a finite filtration, let $\iota_p: F_p A \rightarrow A$. The unit $F^p \mc{F} \rightarrow {\iota_{p+1}}_{*}\iota_{p+1}^* F^p \mc{F}$ induces an isomorphism $\Gr^p K ={\iota_{p+1}}_*\iota_{p+1}^* F^p \mc{F}$.
\end{lemma}

As a consequence of exactness, applying $F^i$ or $\Gr^i$ termwise to complexes induces a functor $D^+(A) \rightarrow D^+(A)$. When the filtration $F_\bullet A$ is finite, the functors $F^i$ assemble to a functor
\[ \Fil_{F_\bullet} : \Sheaves(A) \rightarrow \Sheaves(A)^f \]
that is exact (i.e.\ maps quasi-isomorphisms to filtered quasi-isomorphisms; this is equivalent to the assertion above that each filtered or graded piece functor is exact), so that termwise application to complexes induces  a functor $\Fil_{F_\bullet}:D^+(A) \rightarrow DF^+(A)$ whose associated filtered and graded piece functors are canonically identified with $F^i$ and $\Gr^i$ as above and such that there is a canonical identification $\Forget \circ \Fil = \Id_{D^+(A)}.$

\section{Pospaces and poschemes}\label{s.pospaces-poschemes}

In this section we treat some elementary properties of pospaces and poschemes and their nerves. In \cref{ss.pospacescheme-def-prop} we make the basic definitions, and in \cref{ss.contractibility-criterion} we develop some useful tools for establishing contractibility of the nerve. In particular, we show that the nerve of a pospace/poscheme with a maximum/minimum/center is contractible.

\subsection{Definitions and first properties}\label{ss.pospacescheme-def-prop}

\begin{definition}[Pospaces/poschemes]\hfill
\begin{enumerate}
\item For $X$ a topological space, a \emph{pospace over $X$} is a continuous map $\mc{P} \rightarrow X$ equipped with a closed poset relation $\sle_\mc{P} \subset \mc{P} \times_X \mc{P}$.
\item For $X$ a scheme, a \emph{poscheme over $X$} is a morphism $\mc{P} \rightarrow X$ equipped with a closed poset relation $\sle_\mc{P} \subset \mc{P} \times_X \mc{P}$ (in this case, by a poset relation we mean that it should induce a poset structure on $\mc{P}(T)$ for any $T/X$).
\end{enumerate}
A pospace/scheme $\mc{P}/X$ is \emph{proper} if $\mc{P}\rightarrow X$ is proper as a map of topological spaces/schemes. A map $f:\mc{P}_1 \rightarrow \mc{P}_2$  of pospaces/poschemes over $X$ is a map of spaces/schemes over $X$ respecting the order relation, i.e.\ such that $f\times f|_{\sle_{\mc{P}_1}}$ factors through $\sle_{\mc{P}_2}$.

\end{definition}

For $\mc{P}/X$ a pospace/poscheme, we write $\mathord{\geq}_{\mc{P}}$ for the closed relation obtained by swapping the coordinates; we have $\Delta_{\mc{P}}=\mathord{\geq}_{\mc{P}} \cap \sle_{\mc{P}}$, thus $\Delta_{\mc{P}}$ is closed, so  $\mc{P}/X$ is  separated as a map of topological spaces/schemes. We will also write $\slt_\mc{P}=\sle_{\mc{P}} \backslash \Delta_{\mc{P}}$, the complement of the diagonal in $\sle_\mc{P}$, and similarly for $\mathord{>}_{\mc{P}}$.

\begin{example}
Any poset $\mc{P}$ induces a constant pospace/scheme over any topological space/scheme $X$, $\mc{P} \times X$.
\end{example}

It will be convenient to use the following relative interval notation:
\begin{definition}[Intervals]
 Suppose $\mc{P}/X$ is a pospace/poscheme, and suppose given $T_1 \rightarrow \mc{P}$ and $T_2 \rightarrow \mc{P}$. The \emph{open interval} $(T_1,T_2)$ is the fiber product
\[\begin{tikzcd}
	{(T_1, T_2)} & {} & {} & {T_1 \times_X T_2} \\
	{\slt_{\mc{P}}\times_{\mc{P}} \slt_{\mc{P}} } & {\;} & {} & {\mc{P}\times_X \mc{P}}
	\arrow[from=1-1, to=2-1]
	\arrow[from=1-1, to=1-4]
	\arrow["{(a < b < c)\mapsto (a,c)}", from=2-1, to=2-4]
	\arrow[from=1-4, to=2-4]
\end{tikzcd}\]
viewed as a pospace/poscheme over $T_1 \times T_2$ with ordering pulled back from the  middle coordinate. The closed and mixed intervals $[T_1, T_2]$, $(T_1, T_2]$, $[T_1,T_2)$ over $T_1 \times T_2$ are defined by changing between $\leq$ and $\slt$ appropriately in the bottom left.

Given $T \rightarrow \mc{P}$, we similarly define the  intervals  $(-\infty, T)$, $(-\infty, T]$, $(T,\infty)$ and $[T, \infty)$ as poschemes/pospaces over $T$. These latter can be interpreted literally in the above notation as intervals in the  pospace/poscheme obtained by adding disjoint maximum and minimum sections $\infty$ and $-\infty$ to $\mc{P}/X$.
\end{definition}

\begin{definition}[Rankings and split pospaces/poschemes]\hfill
\begin{enumerate}
\item A \emph{ranking} on a pospace/poscheme $\mc{P}/X$ is a strictly increasing map $\rk:\mc{P} \rightarrow \mbb{Z} $, where here strictly means that for any $k \in \mbb{N}$, the restriction of $\sle_{\mc{P}}$ to $\rk^{-1}(\{k\})$ is the diagonal. The ranking is \emph{finite} if it factors through $\{-n, \ldots, n\} \subset \mbb{Z}$ for some $n$. A pospace/poscheme $\mc{P}/X$ is (\emph{finitely}) \emph{ranked} if it is equipped with a (finite) ranking, in which case we write $\mc{P}_k:=\rk^{-1}(\{k\})$, $\mc{P}_{\leq k}:=\rk^{-1}((-\infty, k])$, etc.
\item A pospace/poscheme $\mc{P}/X$ is \emph{split} if $\Delta_{\mc{P}}$ is open in $\sle_{\mc{P}}$ (or equivalently in $\mathord{\geq}_{\mc{P}}$), in which case $\sle_{\mc{P}}=\Delta_{\mc{P}} \sqcup \slt_{\mc{P}}$ (and $\mathord{\geq}_{\mc{P}} =\Delta_{\mc{P}} \sqcup \mathord{>}_{\mc{P}}). $
\end{enumerate}
\end{definition}

Any pospace/poscheme that admits a ranking is split, and our motivating applications all fall into this category. It is possible to construct examples that are not split, but imposing this condition will simplify some hypotheses later (because then maxima and minima are centers --- see \cref{remark:split-centers}).

\begin{definition}
If $\mc{P}$ is a pospace/poscheme over $X$,
\begin{enumerate}
\item a \emph{subpospace/subposcheme} $\mc{P}'\subset \mc{P}$ is a subspace/subscheme equipped with the induced relation
\[ \mathord{\leq}_{\mc{P}'} = \mathord{\leq}_{\mc{P}} \times_{\mc{P} \times_X \mc{P}} {(\mc{P}' \times_X \mc{P}')}.\]
\item  if  $\mc{P}' \subset \mc{P}$ is a subpospace/subposcheme, a \emph{retraction of $\mc{P}$ onto $\mc{P}'$} is a map $r: \mc{P}\rightarrow \mc{P}'$ of pospaces/poschemes over $X$ such that $r|_{\mc{P'}}=\Id_{\mc{P}'}.$
\end{enumerate}
\end{definition}

\begin{definition}\label{sss.retractions-centers-etc}\label{defn:center}
If $\mc{P}/X$ is a pospace or poscheme,
\begin{enumerate}
\item a \emph{maximum} of $\mc{P}/X$ is a section $m: X \rightarrow \mc{P}$ such that $\mc{P}=(-\infty, m]$
\item a \emph{minimum} of $\mc{P}/X$ is a section $m: X \rightarrow \mc{P}$ such that $\mc{P}=[m, \infty)$
\item a \emph{center} of $\mc{P}/X$ is a section $c: X \rightarrow \mc{P}$ such that
\[ \mc{P}=(-\infty,c) \sqcup c(X) \sqcup (c,\infty) \]
(in particular, $c$ is isolated, i.e.\ $c(X)$ is clopen).
\end{enumerate}
If $\mc{P}/X$ is a poscheme, a \emph{weak} maximum/minimum/center is a section such that the corresponding identity holds topologically but not necesarily scheme-theoretically.
\end{definition}

\begin{remark}\label{remark:split-centers} For a split pospace/poscheme, if $\mc{P} =  (-\infty,c] \cup [c,\infty)$ set theoretically then it follows immediately that $c$ is a center in the topological case and a weak center in the scheme theoretic case. In particular, maxima and minima are centers for split $\mc{P}$. For nonsplit $\mc{P}$ and $c$ satisfying $\mc{P} = (-\infty,c] \cup [c,\infty)$, there can be connected $T/X$ such that $c$ is not a center of $\mc{P}(T)$ in the naive sense, which is why we require that $c$ be isolated. This issue does not occur with maxima and minima even when they are not isolated.
\end{remark}

\begin{definition} The \emph{nerve} (or order complex) of a pospace/poscheme $\mc{P}/X$ is the simplicial space/scheme $N\mc{P}$, canonically augmented to $X$, whose space/scheme of $m$-simplices is the space/scheme of chains of length $m+1$ in $\mc{P}$,
\[ N\mc{P}_m := \underbrace{\sle_\mc{P} \times_\mc{P} \sle_\mc{P} \times_\mc{P} \dotsm \times_\mc{P} \sle_\mc{P} }_{m-1 \textrm{ terms}}  \]
with the obvious transition maps. In other words, it is the space/scheme over $X$ of pospace/poscheme maps from $[m]\times X$ to $\mc{P}$. We also write
\[ N\mc{P}^\circ_m := \underbrace{\slt_\mc{P} \times_\mc{P} \slt_\mc{P} \times_\mc{P} \dotsm \times_\mc{P} \slt_\mc{P} }_{m-1 \textrm{ terms}},  \]
for the subspace/subscheme of $N\mc{P}_m$ consisting of non-degenerate $m$-simplices (i.e.\ strictly ordered chains). We say $\mc{P}/X$ is \emph{of bounded length} if $N\mc{P}^\circ_m = \emptyset$ for $m$ sufficiently large, i.e.\ if there is a bound on the length of chains of proper inequalities in $\mc{P}/X$.
\end{definition}

Note that a finitely ranked pospace/poscheme is bounded, and that any ranked topologically Noetherian poscheme is finitely ranked.

\begin{example}\label{example.poset-nerve-constructions}\hfill
\begin{enumerate}
\item $N[n]=\Delta^n$.
\item $N\mc{P}_2$ is naturally identified with the closed interval $[\mc{P},\mc{P}]$ with the map to $\mc{P}\times \mc{P}$ induced by the edge $0 < 2$ of $\Delta^2$  and the ordering pulled back from the map to $\mc{P}$ induced by the vertex $1$ of $\Delta^2$. The non-degenerate $2$-simplices $N\mc{P}^\circ_2$ inside are identified with the open interval $(\mc{P},\mc{P}) \subset [\mc{P},\mc{P}].$
\item $N(N\mc{P})$ is the barycentric subdivision of $N\mc{P}$, where $N(N\mc{P})$ makes sense by the abuse of notation in which we treat $N\mc{P}$ as the pospace/poscheme $\bigsqcup_{k=0}^\infty N\mc{P}_k$ with ordering by inclusion of chains.
\item If $\mc{P} / X$ is the disjoint union of $\mc{A} < \mc{B}$ then $N\mc{P}$ is the join $N\mc{A} \star N\mc{B}$ (see \cref{example.simplicial-constructions}). In particular if $\mc{P}$ has a center $c$ then
\[N\mc{P} = N(-\infty, c] \star N(c,\infty) = N(-\infty,c) \star N[c,\infty).\]
In particular, if $\mc{P}/X$ admits an isolated minimum $m$ (resp.\ isolated maximum $M$) then $N\mc{P}$ is naturally identified with $\Cone(N(m,\infty)/X)$ (resp.\ $\Cocone( N(-\infty, M) /X)$).
\end{enumerate}
\end{example}

The following is immediate from the definitions:
\begin{lemma}\label{lemma.split-po-split-simplicial} If $\mc{P}/X$ is split (in particular, if it is ranked), then $N\mc{P}$ is split as a simplicial scheme, i.e.\ each degeneracy map is an isomorphism onto a clopen set; in particular $N\mc{P}^\circ_m$ is clopen in $N\mc{P}_m$ for all $m$.
\end{lemma}

\subsection{A contractibility criterion}\label{ss.contractibility-criterion}

We now show that if a pospace/poscheme $\mc{P}/X$ admits a maximum/minimum/center then $N\mc{P}/X$ is contractible. This is essentially an immediate consequence of \cref{example.poset-nerve-constructions}-(4) --- the case of isolated maxima and minima follows from the contractibility of cones and cocones, while the case of a center follows from contracting separately the two cones forming the join. We make this precise as a consequence of the following more general result, which will have further applications later on.

\begin{lemma}\label{lemma:retraction-homotopy} Let $\mc{P}/X$ be a pospace\slash poscheme and let $\mc{P}'\subset \mc{P}$ be a subpospace\slash subposcheme. If either
\begin{enumerate}
\item $(-\infty, \mc{P}] \cap \mc{P}' \times\mc{P}$ has a maximum (as a pospace/poscheme over $\mc{P}$), or
\item $[\mc{P}, \infty) \cap \mc{P} \times \mc{P}'$ has a minimum (as a pospace/poscheme over $\mc{P}$)
\end{enumerate}
then the inclusion of $N\mc{P}'$ in $N\mc{P}$ is a deformation retract over $X$.
\end{lemma}
\begin{proof}
We treat the first case: let $r$ denote the composition of the maximum, a map $ \mc{P} \rightarrow (-\infty, \mc{P}] \cap \mc{P}' \times\mc{P}$, with the projection to the first factor in $\mc{P}'$. It is evidently a retraction of $\mc{P}$ onto $\mc{P}'$, so we only need to show that $Nr$, viewed as a map from $N\mc{P}$ to itself, is homotopic to the identity. We define a homotopy
\[ N\mc{P} \times \Delta^1 \rightarrow N\mc{P} \]
on simplices as follows: given a $k$-simplex $(p_0, \ldots, p_k) \times \alpha$ where $\alpha:[k]\rightarrow [1]$ is non-decreasing, send it to $(r^{1-\alpha(0)}(p_0), r^{1-\alpha(1)}(p_1), \ldots )$ where  $r^0=\Id$ and $r^1=r$. The identity $r(p_i) \leq p_i$ ensures this is a chain because $\alpha$ is non-decreasing, and evidently it is a homotopy from $r$ to the identity.

For the second case, we define $r$ using the minimum, then obtain a homotopy from $r$ to $\Id_{N\mc{P}}$ similarly by replacing $r^{1-\alpha(i)}$ with $r^{\alpha(i)}$ above.
\end{proof}

\begin{remark} By Yoneda, the first statement is equivalent to the statement that, for any $t: T \rightarrow \mc{P}$, the intersection of the poset $(-\infty, t) \in \mc{P}(T)$ with $\mc{P}'(T)$ has a maximum, and similarly for the second statement. Thus one can read the statement as a version of the  Quillen fiber lemma in a particularly simple case. It would be interesting to see if this interpretation can be pushed any further.
\end{remark}

\begin{remark}\label{remark.falling-rising-retractions}
To give a map $r: \mc{P}\rightarrow \mc{P}'$ as in the proof in case (1) of \cref{lemma:retraction-homotopy} is equivalent to giving a retraction $r: \mc{P} \rightarrow \mc{P'}$ that, when viewed as a map $\mc{P} \rightarrow \mc{P}$, satisfies $r \leq \Id_{\mc{P}}$ in the poset $\mc{P}(\mc{P}).$  We call such a map a \emph{falling retraction}, since it can be equivalently characterized as a retraction such that, for all $t$, $r(t) \leq t$. On the other hand, one easily sees that, given a falling retraction $r$, $r \times \Id$ is a maximum of $(-\infty, \mc{P}] \cap \mc{P}' \times \mc{P}$, so to verify (1) in the statement of \cref{lemma:retraction-homotopy} is the same as to give a falling retraction. Similarly, to verify (2) is  the same as to give a rising retraction, i.e.\ a retraction $r: \mc{P} \rightarrow \mc{P}'$ such that $r(t) \geq  t$ for all $t$.
\end{remark}

As an application, we find:
\begin{lemma}\label{lemma:max-min-center-homotopy}
If $\mc{P}/X$ is a pospace/scheme with a maximum, minimum, or center, then $N\mc{P}/X$ is contractible.
\end{lemma}
\begin{proof}
Given a maximum (resp.\ minimum) $m$, we can apply \cref{lemma:retraction-homotopy} to the sub-poscheme $\mc{P}'=m(X)$. Given a center $c$, we can first apply \cref{lemma:retraction-homotopy}-(1) to $(-\infty, c] \subset \mc{P}$ --- here the maximum on $(-\infty, \mc{P}] \cap (-\infty, c] \times \mc{P}$ is given by $\Id \times \Id$ on the clopen set $(\infty, c]$ and by $c \times \Id$ on the clopen $(c, \infty]$. We then use that $c$ is a maximum of $(-\infty,c]$.
\end{proof}

\section{Cohomological inclusion-exclusion}\label{s.cohomological-inclusion-exclusion}

In this section we elaborate on the setup from \cref{ss.intro-sheaves} to prove \cref{main.cd-criterion} and \cref{main.spectral-sequences}. We recall that \cref{main.cd-criterion} gives a useful criterion for cohomological descent, while \cref{main.spectral-sequences} gives the spectral sequence of a rank function in the presence of cohomological descent.  In the introduction we restricted to a single sheaf and cohomology with compact support, but here we will obtain the more general statements alluded to in  \cref{remark.cd-complexes} and \cref{remark.filtration-functor-spectral-sequences}.

We first state a version of \cref{main.cd-criterion} that applies to complexes. The conditions imposed are those necessary to ensure a suitable proper base change theorem holds.

\begin{theorem}[Refinement of \cref{main.cd-criterion}]\label{body.cd-criterion}\hfill
\begin{enumerate}
\item If $\mc{P}/X$ is a proper {\bf pospace} and $\mc{P}_x$ admits a center for all $x \in X$, then $\mc{P}/X$ is of cohomological descent on $D^+(X)$.
\item If $\mc{P}/X$ is a proper {\bf poscheme} and $\mc{P}_x$ admits a weak center for every geometric point $x: \Spec \kappa \rightarrow X$, then
\begin{enumerate}
\item $\mc{P}/X$ is of cohomological descent on the subcategory of $D^+(X_\et)$ consisting of complexes with torsion cohomology sheaves.
\item If $X$ is furthermore Noetherian and $L$ is an algebraic extension of $\mbb{Q}_\ell$ for some $\ell$ invertible on $X$, then $\mc{P}/X$ is of cohomological descent on the subcategories of $D^+(X_\proet, \mc{O}_L)$ or $D^+(X_\proet, L)$ consisting of complexes with constructible cohomology sheaves.
\end{enumerate}
\end{enumerate}
\end{theorem}

In \cref{ss.technical-criteria} we explain the reduction of cohomological descent to pointwise contractibility via proper base change (and topological invariance of the étale site in the schematic case), then in \cref{ss.proof-of-cd-crit} we put this together with the contractibility criteria of \cref{ss.contractibility-criterion} to finish the proof of \cref{body.cd-criterion}.

For any ranked $\pi: \mc{P} \rightarrow X$, we obtain a rank filtration $F_\bullet^{\rk}$ on $N\mc{P}$. As in $\cref{ss.filtered-derived}$, this induces a filtration on sheaves, and if the ranking is finite we obtain a functor
\[ \Fil_{\mc{P},\mr{rk}} = R\epsilon^f_* \circ \Fil_{F_\bullet^\rk} \circ \epsilon^* : D^+(\Sheaves(X)) \rightarrow DF^+(\Sheaves(X)) \]
along with graded parts functors $\Gr^p_{\mc{P},\mr{rk}}$ (the latter are defined even if the filtration is not finite). Up to a shift, we will identify the graded parts with the reduced cohomology complex functors that were described in the introduction. As in the introduction (paragraph preceding \cref{main.spectral-sequences}), we extend the rank function to $\mc{P}^+/X$ (obtained by adding a disjoint minimum section $-\infty$), then define
\begin{align*}
 \tilde{C}(-\infty, \mc{P}_r^+, \bullet): & \;D^+(\Sheaves(X)) \rightarrow D^+(\Sheaves(\mc{P}_r^+))\\
 & \;K \mapsto \begin{cases} \tilde{C}(N(-\infty, \mc{P}^+_{r})/\mc{P}^+_{r}, \pi^*K) & r \neq \rk(-\infty) \\
 	K[2] & r=\rk(-\infty)
 \end{cases}
\end{align*}

Precisely, we show:

\begin{theorem}\label{body.filtration-functor}
Suppose $\pi: \mc{P}\rightarrow X$ is ranked. There are canonical isomorphisms of functors $D^+(\Sheaves(X))\rightarrow D^+(\Sheaves(X))$
\[ \Gr^p_{\mc{P}, \rk}(\bullet) \xrightarrow{\sim}  R\pi_* \circ \tilde{C}(-\infty, \mc{P}^+_{p+1}, \bullet)[-1]. \]
If  $\mc{P}/X$ is finitely ranked and of cohomological descent on a subcategory $\mc{C} \hookrightarrow D^+(\Sheaves(X))$, then there is a canonical isomorphism from the inclusion $\mc{C} \hookrightarrow D^+(\Sheaves(X))$ to $\Forget \circ \Fil_{\mc{P}, \rk}$.
\end{theorem}

This immediately implies a functorial spectral sequence for any derived functor restricted to $\mc{C}$. We state a useful general case now --- in \cref{ss.spectral-seq-proof} we will provide the justification, as well as details for the examples given in \cref{ss.intro-sheaves}.

\begin{theorem}[Refinement of \cref{main.spectral-sequences}] \label{body.spectral-sequences}
Suppose $X$ is a topological space\slash scheme, $K \in D^+(X)$, and $\pi:\mc{P} \rightarrow X$ is a proper finitely ranked pospace\slash poscheme. Let $Z=\pi(\mc{P})$ with complement $U=X\backslash Z$,  $j:U \hookrightarrow X$. Suppose also that  $\mc{P}|_Z$ is of cohomological descent for $K|_Z$. Let $\star=* \textrm{ or } !$ and let $f:X \rightarrow S$ be a morphism; in the scheme case, if $\star = !$, suppose further that $S$ is qcqs and $f$ is separated and of finite type (so $Rf_!$ is defined). Then there are spectral sequences, functorial in $K$,
\begin{enumerate}
\item For $Rf_\star$ on $Z$:
\[ E_1^{p,q}=R^{p+q-1} (f\circ \pi|_{\mc{P}_{p+1}})_\star\tilde{C}(-\infty, \mc{P}_{p+1}^+, K) \Rightarrow R^{p+q}{f|_{Z}}_\star(K|_Z). \]
\item For $Rf_\star$ on $U$:
\[ E_1^{p,q} = R^{p+q-2}(f\circ \pi|_{\mc{P}^+_{p}})_\star\tilde{C}(-\infty, \mc{P}^+_p, K)   \Rightarrow R^{p+q}{f}_\star j_!j^*K.	 \]
\end{enumerate}
\end{theorem}

\subsection{Technical criteria for cohomological descent}\label{ss.technical-criteria}

The following summarizes the application of proper base change that allows us to spread out cohomological descent from (geometric) points when it holds; in the schematic case we must put some restrictions in place to obtain a suitable proper base change theorem.

\begin{lemma}[Spreading out cohomological descent by proper base change]\label{lemma.spread-out-proper}\hfill
\begin{enumerate}
\item (Spaces). Let $X$ be a topological space, let $\mc{P}/X$ be a proper pospace, and let $K \in D^+(X)$. If $N\mc{P}_x/\{*\}$ is of cohomological descent for $K_x$ for all $x \in X$ then $N\mc{P}/X$ is of cohomlogical descent for $K$.
\item (Schemes étale). Let $X$ be a scheme, let $\mc{P}/X$ be a proper poscheme and let $K \in D^+(X_\et)$ have torsion cohomology sheaves. If $N\mc{P}_x$ is of cohomological descent for $K_x$ for all geometric points $x: \Spec \kappa \rightarrow X$ then $N\mc{P}/X$ is of cohomological descent for $K$.
\item (Schemes pro-étale). Let $X$ be a Noetherian scheme, let $\mc{P}/X$ be a proper poscheme, and let $K \in D^+(X_\proet, \mc{O}_E)$ or $D^+(X_\proet, L)$ for $L$ an algebraic extension of $\mbb{Q}_\ell$, $\ell$ invertible on $X$, have constructible cohomology sheaves. If $N\mc{P}_x/\Spec \kappa$ is of cohomological descent for $K_x$ for all geometric points $x: \Spec \kappa \rightarrow X$ then $N\mc{P}/X$ is of cohomological descent for $K$.
\end{enumerate}
\end{lemma}
\begin{proof}
To establish cohomological descent for a complex $K$, it suffices to show that $\tilde{C}(N\mc{P}/X, K) \cong 0$. To show this it suffices to verify that the stalks vanish at all points/geometric points in $X$ --- it suffices to check on the cohomology sheaves, and this statement is true for all topological or étale sheaves; it is not true in general for a pro-étale sheaf (there are of course many interesting pro-étale sheafs on a geometric point with no global sections!), but it is true for a constructible pro-étale sheaf, which reduces to the étale case. To conclude, we invoke proper base change in each setting to identify the stalk at a point $x$ with
\[ \tilde{C}(N\mc{P}_x/\{x\}, K_x), \]
which vanishes by the assumption that cohomological descent holds on fibers. The relevant proper base change theorems can be found as \cite[\href{https://stacks.math.columbia.edu/tag/09V6}{Tag 09V6}]{stacks-project} in the topological case, \cite[\href{https://stacks.math.columbia.edu/tag/0DDE}{Tag 0DDE}]{stacks-project} for the étale case, and \cite[Lemma 6.7.5 and Proposition 6.8.14]{bhatt-scholze:pro-etale} in the pro-étale case.
\end{proof}

To show a pospace/scheme over a point satisfies cohomological descent, we will typically appeal to contractibility of the nerve; in \cref{ss.contractibility-criterion} above we already established some useful tools that can be used to deduce this contractibility. For schemes we will often need one more reduction before we can apply these contractibility criteria; the result is encoded in the following lemma. We say $f: \mc{P}_1 /X \rightarrow \mc{P}_2/X$ is a universal homeomorphism of poschemes over $X$ if it is a map of poschemes over $X$ that is universal homeomorphism as a map of schemes and induces a universal homeomorphism $\sle_{\mc{P}_1} \rightarrow \sle_{\mc{P}_2}$.

\begin{lemma}[Contractibility and cohomological descent]\label{lemma.nerve-contractible-descent} \hfill
\begin{enumerate}
\item (Spaces) If $\mc{P}/X$ is a pospace and $N\mc{P}/X$ is contractible then $N\mc{P}/X$ is of cohomological descent on $D^+(X)$.
\item (Schemes) If $f: \mc{P} /X \rightarrow \mc{Q}/X$ is a universal homeomorphism of poschemes over $X$ and $N\mc{P}/X$ is contractible, then $N\mc{Q}/X$ satisfies cohomological descent on $D^+(X_\et)$, $D^+(X_\proet, \mc{O}_L)$, and $D^+(X_\proet, L)$.
\end{enumerate}
\end{lemma}
\begin{proof}
The statement for topological spaces is a reformulation of
\cref{lemma:homotopy-equiv-and-contractible}. The schemes statement combines \cref{lemma:homotopy-equiv-and-contractible} with the topological invariance of the étale site \cite[\href{https://stacks.math.columbia.edu/tag/03SI}{Tag 03SI}]{stacks-project}  --- see \cite[Lemma 5.4.2]{bhatt-scholze:pro-etale} for an explanation of how this leads also to an equivalence of the corresponding pro-étale topoi. The key point is then that a universal homeomorphism of poschemes induces a universal homeomorphism on the simplex spaces of the nerve.
\end{proof}

\subsection{Weak centers and the proof of \scref{body.cd-criterion}{Theorem}}\label{ss.proof-of-cd-crit}
The topological case of \cref{body.cd-criterion} can now be established: we apply \cref{lemma.spread-out-proper}, where the punctual condition is satisfied by combining \cref{lemma:max-min-center-homotopy} and \cref{lemma.nerve-contractible-descent}.

In order to apply the same argument in the scheme-theoretic case, we must briefly return to the notion of a weak center: Recall from \cref{sss.retractions-centers-etc} that for $\mc{P}/X$ a poscheme, $c \in \mc{P}(X)$ is a \emph{weak} center/minimum/maximum if the defining decomposition of $\mc{P}$ holds topologically but not necessarily scheme-theoretically.
\begin{example} If $\kappa$ is a field, the poscheme $\mc{P}=\Spec \kappa[\epsilon]/(\epsilon^2)$ over $\Spec \kappa$ with trivial ordering relation $\sle_{\mc{P}} = \Delta_{\mc{P}}$ admits a weak center given by the reduced subscheme  $\Spec \kappa = \mc{P}^\mr{red} \hookrightarrow \mc{P}$.
\end{example}

To conclude, we consider the following elementary reinterpretation:
\begin{lemma}\label{lemma.weak-universal-replacement} A poscheme $\mc{P}/X$ admits a weak center/maximum/minimum if and only if there is a poscheme $\mc{P}'/X$ with a center/maximum/minimum and a universal homeomorphism of poschemes $\mc{P}'/X \rightarrow \mc{P}/X$.
\end{lemma}
\begin{proof}
We treat the case of a center, the others being essentially the same.
If such a $\mc{P}'$ exists, then, for $c \in \mc{P}'(X)$ a center, the induced point in $\mc{P}(X)$ is clearly a weak center. Conversely, if $c \in \mc{P}(X)$  is a weak center of $\mc{P}$, we can take $\mc{P'}$ to be the closed subposcheme $\{< c\} \bigsqcup c(X) \bigsqcup \{> c\}$ of $\mc{P}$.
\end{proof}

Given \cref{lemma.weak-universal-replacement}, the argument for \cref{body.cd-criterion} is exactly as in the topological case.

\subsection{Construction of filtration and computation of graded components (proof of \scref{body.filtration-functor}{Theorem})} \label{ss.construction-filtration-graded-components}
Let $\mc{P}/X$ be a finitely ranked pospace\slash poscheme (with no assumptions of cohomological descent or contractibility!). We consider a split simplicial filtration $F_i N\mc{P}=\rk^{-1}(-\infty, i]$ of $N\mc{P}$. By the formalism of \cref{ss.filtered-derived}, we obtain
\[ \Fil_{F_\bullet N\mc{P}}: D^+(N\mc{P}) \rightarrow DF^+( N\mc{P})  \]
and a canonical isomorphism
\[ \Id_{D^+(N\mc{P})} = \Forget \circ \Fil_{F_\bullet N\mc{P}}.\]
We then consider the functor $\Fil_{\mc{P},\rk}: D^+(X) \rightarrow DF^+(X)$ given by composition as
\[ \Fil_{\mc{P},\rk}:= R\epsilon^f_* \circ \Fil_{F_\bullet N(\mc{P})} \circ \epsilon^*, \]
where $\epsilon:N\mc{P}\rightarrow X$ is the augmentation. We have a canonical natural transformation
\begin{equation}\label{eq.canonical-id-forget} \Id_{D^+(X)} \rightarrow \Forget \circ \Fil_{\mc{P},\rk} \end{equation}
given as the composition of
\[ \Id_{D^+(X)} \xrightarrow{u_{\mc{P}}} R\epsilon_* \circ \epsilon^* \cong R\epsilon_* \circ \Forget \circ \Fil_{F_\bullet N\mc{P}} \circ \epsilon^* \cong  \Forget \circ R\epsilon^f_* \circ \Fil_{F_\bullet N\mc{P}} \circ \epsilon^*. \]

In particular, \cref{eq.canonical-id-forget} restricts to an isomorphism of functors $D \rightarrow D^+(X)$ on any subcategory $D$ on which $\mc{P}/X$ is of cohomological descent.

It remains to compute the graded pieces. Before doing so, we consider a useful example that computes the graded pieces for a simplicial cocone:
\begin{example}
If $\epsilon:A \rightarrow X$ then $\epsilon_{\Cocone(A)}: \Cocone(A) \rightarrow X$ is contractible thus satisfies cohomological descent. If we take the filtration as in \cref{example.cone-filtration}, then
\[ R{\epsilon_{\Cocone(A)}}_* \Gr^p_{F_\bullet\Cocone(A)/X} K \cong \begin{cases*}R{\epsilon_{A}}_*\epsilon_{A}^*K & if $p=0$  \\
\tilde{C}(A/X, K)[-1] & if $p=1$.  \end{cases*}\]
Indeed, since the cone is contractible, after applying $R{\epsilon_{\Cocone(A)}}_*$, the quotient to $\Gr^1$ with kernel $F^1$ is an exact triangle
\[R{\epsilon_{\Cocone(A)}}_* (F^0\epsilon_{\Cocone(A)}^* K) \rightarrow K \rightarrow  R{\epsilon_{A}}_*\epsilon_{A}^*K \rightarrow R{\epsilon_{\Cocone(A)}}_* (F^0\epsilon_{\Cocone(A)}^* K)[1] \]
and the claim follows from the definition of $\tilde{C}(A/X, K)$.
\end{example}
The argument for computing the graded pieces then is a reduction to this example via an excision. Consider for each $p$ the natural map
\[ r_p: \Cocone(N(-\infty, \mc{P}_p)) \cong N(-\infty, \mc{P}_p] \rightarrow  N\mc{P}. \]

\begin{lemma} For any $K \in D^+(N\mc{P})$, $\Gr^{p-1}_{\Fil_\bullet N\mc{P}} K = R{r_p}_*\Gr^1 r_p^* K$ where the graded on the right is for the cocone filtration as in \cref{example.cone-filtration}.
\end{lemma}
\begin{proof}
	The point is that $(r_p)_m$ maps $F_0C(N(-\infty, \mc{P}_p) )$ into $F_{p-1} N(\mc{P})$ and on $m$-simplices restricts to an isomorphism
	\[ \Cocone(N(-\infty, \mc{P}_p)/\mc{P}_p)_m \backslash F_0 \Cocone(N(-\infty, \mc{P}_p)/\mc{P}_p)_m \xrightarrow{\sim} F_p N\mc{P}_m \backslash F_{p-1} N\mc{P}_m. \]
	Suppose we represent $K$ by a complex of injectives. Then, $K_m$ is a complex of injectives for each $m$, and same for $F^{p-1}K_m$.  Then $(r_p^*F^{p-1}K)_m = (F^1 r_p^*K)_m$ is a complex of injectives. Thus pushforward of this complex computes $R{r_p}_*$, but that is just identified with the pushpull from $F_p N\mc{P}$ and by \cref{lemma.push-pull-graded} we conclude this is $\Gr^{p-1} K$.
\end{proof}

Combined with \cref{example.cone-filtration}, we obtain functorially in $K$,
\begin{align}\label{eq.body-graded-computation} \Gr_{\mc{P}, \rk}^p K & = R\epsilon_* \Gr^p \epsilon^*K \\
 \nonumber & = R{(\epsilon \circ r_p)}_* \Gr^1 (\epsilon \circ r_p)^*K \\
 \nonumber & = R{(\pi \circ \epsilon_{\Cocone(N(-\infty, \mc{P}_{p+1}))})}_* \Gr^1(\pi \circ \epsilon_{\Cocone(N(-\infty, \mc{P}_{p+1}))})^* K \\
\nonumber &= R\pi_* \tilde{C}(N(-\infty, \mc{P}_{p+1})/\mc{P}_{p+1}, \pi^*K)[-1].
\end{align}

This completes the proof of \cref{body.filtration-functor}.

\subsection{Spectral sequences}\label{ss.spectral-seq-proof}

\subsubsection{Proof of \cref{body.spectral-sequences}}

Write $\tilde{K}=\Fil_{\mc{P},\rk}(K|_Z)$. Then, since we assume $\mc{P}/Z$ is of cohomological descent for $K|_Z$, we have $\Forget(\tilde{K})=K|_Z$. We then apply \cite[\href{https://stacks.math.columbia.edu/tag/015W}{Tag 015W}]{stacks-project} to obtain the spectral sequence (1), whose terms are described by \cref{eq.body-graded-computation}; it remains only to observe that $\pi$ is proper by assumption so that when $\star=!$, $R(f \circ \pi)_!=Rf_! \circ R\pi_*$. For part (2) we extend to a filtration of $j_!j^*K$ by
\[ i_*i^*K[-1]\rightarrow j_!j^*K\rightarrow K \textrm{ for } j: U \hookrightarrow X \textrm{ and } i:Z\hookrightarrow X\]
before applying the spectral sequence (and make a similar observation when $\star=!$).

\subsubsection{The first spectral sequence of a stratified space}
\label{sss.stratified-space-body}
Recall the setup from \cref{sss.stratified-space}: $X = \bigcup_{\alpha} S_\alpha$ for disjoint locally closed sets $S_\alpha$ and $Z_{\alpha} := \overline{S_\alpha} = \bigcup_{\beta \le \alpha} S_\beta$.
We consider the pospace\slash poscheme given by $\mc{P} = \bigsqcup_{\alpha} Z_\alpha \rightarrow X$ with $z_\alpha \geq z_\beta$ if $z_\alpha = z_\beta$ as elements of $X$ and $\alpha \ge \beta$ and with a ranking given by a ranking on the indexing set.

We have, $\mc{P}_{p+1} = \bigsqcup_{\rk(\alpha)=p+1} Z_\alpha$, and we will also write $j:U_{p+1} \rightarrow \mc{P}_{p+1}$ and $i: \partial \mc{P}_{p+1} \rightarrow \mc{P}_{p+1}$ where
\[ U_{p+1}:=\bigsqcup_{\rk(\alpha)=p+1} S_\alpha, \qquad \partial \mc{P}_{p+1}:=\mc{P}_{p+1}\backslash U_{p+1}=\bigsqcup_{\rk(\alpha)=p+1}Z_\alpha \backslash S_\alpha. \]
We then claim that $\tilde{C}(N(-\infty, \mc{P}_{p+1})/\mc{P}_{p+1}, K)=j_!j^*K[1]$.
Indeed,  $(-\infty, \mc{P}_{p+1})/\mc{P}_{p+1}$ is proper with image $\partial \mc{P}_{p+1}$, and the fiber over any geometric point in $\partial \mc{P}_{p+1}$ has a minimum. Thus, writing $\epsilon: N(-\infty, \mc{P}_{p+1}) \rightarrow \mc{P}_{p+1}$, $R\epsilon_* \epsilon^* K= i_*i^*K$ and thus the cone is $j_!j^*K[1]$, as desired.
In particular, if $X$ is a topological space or a variety over an algebraically closed field, then taking $f$ to be the structure map to a point and $\star=!$,  we obtain the standard spectral sequence for a stratified space
\[ E_1^{p,q}=\bigoplus_{\rk(\alpha)=p+1}H^{p+q}_c(S_\alpha, K) \Rightarrow H^{p+q}_c (X, K). \]

\subsubsection{The Banerjee spectral sequence of a symmetric semisimplicial filtration}\label{sss.banerjee}
Here we recover the spectral sequence of O.~Banerjee \cite[Theorem~1]{banerjee}. In that setting, we are given ``face maps'' $f_i : M^p \times X_{n + e} \to M^{p-1} \times X_n$ for $0 \le i < p$, fixed $e$ and $n, p > 0$, satisfying certain conditions that define a \emph{symmetric semisimplicial filtration of $\{X_n\}$ by powers of $M$}.
Take
\[\mc{P} = \bigsqcup_{p > 0} \Sym^p M \times X_{n - ep}\]
over $Z_n := f_0(M \times X_{n-p}) \subset X_n$, where $\Sym^p M = M^p/\mfS_p$ and $\sle_{\mc{P}}$ is given by all possible compositions of the face maps (this is well-defined on $\Sym^\bullet M$ by the assumptions on the face maps \cite[Definition~2.10]{banerjee}).
Given $x \in Z_n$, the ``equalizer'' and ``embedding'' assumptions imply that there is a maximal $p$ and a unique $(S, x) \in \Sym^p(M) \times X_{n - ep}$ that maps to $x$. To us, this means that the pospace $\mc{P}$ admits a fiberwise maximum.
Then \cref{main.spectral-sequences} gives the desired
\[E_1^{p,q} = H_c^q(M^p \times X_{n - ep}; \bbQ) \otimes_{\mfS_p} \sgn \Rightarrow H_c^{p + q}(X_n - Z_n; \bbQ).\]
The sign representation $\sgn$ appears here for the same reason as in \cref{corollary.sgn-cohomology-computation}.

\section{Motivic inclusion-exclusion}\label{s.mot-inc-exc}
In this section we prove \cref{maintheorem.mot-inc-exc}. To set the stage, in \cref{ss.decat} we first discuss an abelian decategorification of the results of the previous section for constructible sheaves and explain how this relates to the finer combinatorial decategorification in the Grothendieck ring of varieties given by \cref{maintheorem.mot-inc-exc}. In \cref{ss.motivic-contractibility} we prove some combinatorial contractibility criteria and deduce \cref{maintheorem.mot-inc-exc}.

\subsection{Decategorifications}\label{ss.decat}
Suppose $X$ is a noetherian scheme; fix $L/\mbb{Q}_\ell$ an algebraic extension for $\ell$ invertible on $X$. Then we can consider the abelian category $\Cons(X,L)$ of constructible $L$-sheaves on (the pro-étale site of) $X$ and its Grothendieck ring $K_0(\Cons(X,L))$. We consider the constructible derived category $D_{\Cons}(X,L)$, the subcategory of the bounded derived category $D^b(X,L)$ consisting of complexes with constructible cohomology sheaves. There is an Euler characteristic  $\mr{Ob}(D_{\Cons}(X,L)) \rightarrow K_0(\Cons(X,L))$,
\[ K \mapsto [K]=\sum_k (-1)^k [H^i(K)]. \]

The constructible derived category is preserved by proper pushforward, thus we can decategorify the work of the previous section:
For $\pi: \mc{P} \rightarrow X$ a proper finitely ranked poscheme of cohomological descent for $K$, $\Fil_{\mc{P},\rk}(K)$ induces
\begin{equation}\label{eq.decat-rank} [K] =  \sum_p [\Gr^p_{\mc{P},\rk}(K)] = - \sum_p [R\pi_* \tilde{C}(-\infty, \mc{P}_{p+1}, K)] \end{equation}
where the shift in \cref{eq.body-graded-computation} manifests as a minus sign. In fact, there is an inclusion-exclusion formula independent of any rank function:  under the same hypotheses, we have
\begin{equation}\label{eq.decat-simplices} [K] = \sum_{k\geq 0} (-1)^k [R{\epsilon^\circ_k}_*{\epsilon^\circ_k}^*K)], \end{equation}
where $\epsilon_k^\circ$ denotes the restriction of $\epsilon_k$ to
the non-degenerate $k$-simplices $N\mc{P}_k^\circ$ (i.e.\ the scheme of strict $(k+1)$ chains). Indeed, since $\mc{P}$ admits a rank function, \cref{lemma.split-po-split-simplicial} shows that $N\mc{P}$ is split, i.e.\ that each degeneracy map is an isomorphism onto connected components. Then, we apply the spectral sequence \cref{eq.simplicial-ss} and use that each column on the $E_1$ page is quasi-isomorphic to the normalized complex given by the kernel of the degeneracy maps, which by the above consideration is exactly the restriction to the space of non-degenerate simplices.
 The choice of a rank function gives a way to break up each simplex space into connected components, thus breaks up each term of \cref{eq.decat-simplices}; rearranging and reassembling, one can recover \cref{eq.decat-rank} --- indeed, we saw exactly this rearrangement phenomenon already in our toy model for classical inclusion-exclusion in the introduction (\cref{ss.intro-classical-inclusion-exclusion}).

 It turns out we can also give a combinatorial decategorification of the inclusion-exclusion formula that lifts \cref{eq.decat-rank} and \cref{eq.decat-simplices}: we work in the modified Grothendieck ring of varieties $K_0(\Var/X)$ --- recall that in characteristic zero this is the standard Grothendieck ring defined by cut and paste relations, but in non-zero characteristic one must also mod out by radicial surjective maps more general than constructible decompositions (e.g.\ purely inseparable field extensions).  We refer to \cite{bilu-howe:mot-eul-mot-stat} for a detailed discussion and other perspectives. Here let us just highlight that there is a natural compactly supported cohomology homomorphism
 \[ K_0(\Var/X) \rightarrow K_0(D_{\Cons}(X,L)),  [f:Y \rightarrow X] \mapsto [Rf_! E], \]
 so that it makes sense to ask for a combinatorial lift of \cref{eq.decat-simplices} to $K_0(\Var/X)$. This lift is exactly what is provided by \cref{maintheorem.mot-inc-exc}, which says
 that if $\mc{P}/X$ is a bounded poscheme with weak geometric centers then
 \[ [X/X]=\chi(N\mc{P}/X)=\sum_{k\geq 0} (-1)^k [N\mc{P}_k^\circ/X] \textrm{ in } K_0(\Var/X).\]
 \begin{remark} We discuss some interesting points of comparison between the abelian and combinatiorial decategorifications:
 \begin{enumerate}
 \item In $K_0(\Var/X)$ we only give a statement for $[X/X]$. However, the analogous statement for $[Y/X]$ or any other class is obtained simply by multiplying by $[Y/X]$. In fact, the statement in $K_0(\Cons(X,L))$ also reduces to just the statement for the constant sheaf $L$ by the projection formula \cite[Lemma 6.7.4]{bhatt-scholze:pro-etale}.
 \item We do not know if \cref{maintheorem.mot-inc-exc} holds if we only require that each fiber is weakly contractible --- we are only able to prove the identity in the Grothendieck ring by using an Euler characteristic analog of \cref{lemma:retraction-homotopy} which directly cancels out isomorphic components in different simplex spaces.
 \item The statement of \cref{maintheorem.mot-inc-exc} does not require any properness hypothesis. Combined with the projection formula (or by directly running the same proof for sheaves), we obtain under the same hypotheses as \cref{maintheorem.mot-inc-exc}
\[ [K] = \sum_k (-1)^k [R{\epsilon_k^\circ}_!{\epsilon_k^\circ}^*K] \textrm{ in } K_0(D_{\Cons}(X,L)). \]
\end{enumerate}
\end{remark}

\subsection{Euler characteristic contractibility criterion and proof of \scref{maintheorem.mot-inc-exc}{Theorem}}\label{ss.motivic-contractibility}

To prove \cref{maintheorem.mot-inc-exc}, we will spread out our geometric weak centers to reduce to the following elementary Euler characteristic version of \cref{lemma:max-min-center-homotopy}.

\begin{lemma}\label{lemma:isolated-center-motivic-inclusion-exclusion}
Suppose $X$ is a Noetherian scheme and $\mc{P}/X$ is a finite type bounded poscheme with a weak center/maximum/minimum. Then
\[ \chi(N\mc{P}/X) = [X/X] \textrm{ in } K_0(\Var/X). \]
\end{lemma}
\begin{proof}
Since we are working in the Grothendieck ring, we can apply  \cref{lemma.weak-universal-replacement} to assume there is a genuine center/maximum/minimum. Then, in the case of a maximum or minimum, by passing to a constructible decomposition we can assume it is isolated so that it is a center.

Thus we may assume we have a center $c$. We write $A_k \subset N\mc{P}^\circ_k$ for the clopen subscheme of strict chains passing through $c$, and $B_k$ for its complement, the strict chains that do not pass through $c$. Then for $k \geq 1$, $B_k \cong A_{k+1}$ by insertion of $c$ in the unique possible spot. Thus the sum
\[ \sum_{k \geq 0} (-1)^k[N\mc{P}^\circ_k] = \sum_k (-1)^k ([A_k] + [B_k]) \]
telescopes and is equal to $A_0=[X]$.
\end{proof}

\begin{proof}[Proof of \cref{maintheorem.mot-inc-exc}]
By Noetherian induction it suffices to show that if $X$ is irreducible and reduced then the identity holds after restriction to a non-empty open $U \subset X$. We write $\eta$ for the generic point of $X$ and fix an algebraic closure $\overline{K(\eta)}$ of $K(\eta)$ and a weak center $c \in  \mc{P}(\overline{K(\eta)})$. We first observe that $c$ is defined over a finite purely inseparable extension $L/K(\eta)$: first, since $\mc{P}/X$ is of finite type, $c$ can be defined over some finite normal extension $M/K(\eta)$. Writing $G=\mr{Aut}(M/K(\eta))$, we have $L=M^G/K(\eta)$ is purely inseparable, thus it suffices to show $c$ is fixed by $G$. So, suppose $\sigma \in G$. Then, by definition of a weak center, since $c$ is a field-valued point, either $c \leq \sigma(c)$ or $c \geq \sigma(c)$; by replacing $\sigma$ with $\sigma^{-1}$, we can assume $c \leq \sigma(c)$. Then, since $\sigma$ preserves the order relation (because it is defined over $K(\eta)$), we find $\sigma(c) \leq \sigma^2(c)$, $\sigma^2(c) \leq \sigma^3(c)$, etc., so that for $k>1$ a multiple of the order of $G$,
\[ c \leq \sigma(c) \leq \dotsm \leq \sigma^k(c)=c. \]
Thus $c \leq \sigma(c) \leq c$ so $\sigma(c)=c$, as desired.

Now, we can spread out $c:\Spec L \rightarrow \Spec K(\eta)$ to a finite radicial surjective map $\tilde{U}\rightarrow U$ over a non-empty open $U \subset X$. By shrinking $U$ further, we can assume this spreading out is a weak center of $\mc{P} \times_X \tilde{U}$. \cref{lemma:isolated-center-motivic-inclusion-exclusion} then gives
\[ \sum_{k\geq 0} (-1)^k[N(\mc{P} \times_X \tilde{U} )^\circ_k/ \tilde{U}] = [\tilde{U}/\tilde{U}] \textrm{ in } K_0(\Var/\tilde{U}). \]
By composing with the map $\tilde{U} \rightarrow U$, we find
\begin{align*}\sum_{k \geq 0}(-1)^k[N\mc{P}^{\circ}_k|_U][\tilde{U}/U]  & = \sum_{k\geq 0} (-1)^k[N\mc{P}^\circ_k \times_X \tilde{U} /U] \\
&= \sum_{k\geq 0} (-1)^k[N(\mc{P} \times_X \tilde{U} )^\circ_k/U] = [\tilde{U}/U] \textrm{ in } K_0(\Var/U).
\end{align*}
Recalling that in the modified Grothendieck ring $[\tilde{U}/U]=[U/U]=1$, we conclude.
\end{proof}

Because it is will be useful in later sections, we also give an analog of \cref{lemma:retraction-homotopy} using a similar argument. Both this and the previous criterion should be special cases of a more general poscheme Euler characteristic version of Quillen's fiber lemma.

\begin{theorem}\label{theorem.poscheme-contraction-euler} Let $X$ be a Noetherian scheme, let $\mc{P}/X$ be a poscheme that is bounded and  of finite type. Let $\mc{P}'\subset \mc{P}$ be a sub-poscheme and suppose that, for every geometric point $t: \Spec \kappa \rightarrow X$, one of the following holds:
\begin{enumerate}
\item $(-\infty, \mc{P}_t] \cap \mc{P}_t' \times\mc{P}_t /\mc{P}_t $ has a weak maximum, or
\item $[\mc{P}_t, \infty) \cap \mc{P}_t \times \mc{P}_t' / \mc{P}_t$ has a weak minimum.
\end{enumerate}
Then \[ \chi(N\mc{P} /X )=\chi(N\mc{P}' /X ) \textrm{ in } K_0(\Var/X). \]
\end{theorem}
\begin{proof}
By Noetherian induction, it suffices to assume $X$ is irreducible and to show that there exists a non-empty open $U \subset X$ where the identity holds. We write $\eta$ for the generic point and $\overline{\eta}: \Spec \overline{K(\eta)} \rightarrow X$ for a geometric point above $\eta$. We assume (1) holds at $\overline{\eta}$, the case (2) being similar. We write
\[ m: \mc{P}_{\overline{\eta}} \rightarrow (-\infty, \mc{P}_{\overline{\eta}}] \cap (\mc{P}' \times_X \mc{P})_{\overline{\eta}}\]
for the weak maximum. First note that $m$ is defined over a finite subextension $M/K(\eta)$, which we may take to be normal. Then, if we let $G=\mr{Aut}(M/K(\eta))$, we must have that $m$ is fixed by the action of $G$ because it is a maximum and the order relation is defined over $K(\eta)$. It follows that $m$ is defined over $L=M^G$, a finite purely inseparable extension of $K(\eta)$: that is, writing $\eta_L: \Spec L \rightarrow X$, we obtain the section $m$ already as
\[ m: \mc{P}_{\eta_L} \rightarrow (-\infty, \mc{P}_{\eta_L}]\cap (\mc{P}'\times_X\mc{P})_{\eta_L}. \]
It is a weak maximum still because this can be checked on geometric points (for a section to be a weak maximum it is necessary and sufficient that it be a maximum on any set of geometric points). Now, we can spread out the purely inseparable map $\Spec L \rightarrow \Spec K(\eta)$ to a radicial surjective $\tilde{U} \rightarrow U$ where $U$ is open in $X$ such that $m$ spreads out to
\[ m: \mc{P}_{\tilde{U}} \rightarrow (-\infty, \mc{P}_{\tilde{U}}]\cap (\mc{P}'\times_X\mc{P})_{\tilde{U}}. \]
Now $(\infty, m]^c$ is open, so its image in $\tilde{U}$ is constructible. Since this image does not contain $\eta_L$, we deduce that its complement, the locus where $m$ is a weak maximum, contains a non-empty open set. So, replacing $\tilde{U}$ with this non-empty open set, we can assume furthermore that $m$ is a weak maximum.

Now, for any geometric point for any $p_0 < p_1 < \dotsm < p_k$ of ${N\mc{P}_{\tilde{U}}^\circ}_k$, either the chain stays entirely in $\mc{P}'_{\tilde{U}}$ or first leaves at an index $i$. Of those that leave, we can break them up into those such that for this first $i$, $p_{i-1}=m(p_i)$, and those where this is not satisfied; this gives a constructible decomposition
\[ N{\mc{P}_{\tilde{U}}}^\circ_k = N{\mc{P}'_{\tilde{U}}}^\circ_k \sqcup A_k \sqcup B_k. \]
Moreover, we claim that $[A_{k+1}/\tilde{U}]=[B_k/\tilde{U}]$. Indeed, we can decompose $B=\bigsqcup_{i=0}^k B_{k,i}$ where $B_{k,i}$ is the constructible set that first leaves at $i$. Then we have a map $B_{k,i} \rightarrow A_{k+1}$ such that
\[ p_0  < \dotsm < p_k \mapsto p_0 <  \dotsm < p_{i-1} < m(p_i) < p_i \dotsm < p_k \]
and the induced map $\bigsqcup_{i=0}^k B_{k,i} \rightarrow A_k$ is a bijection on geometric points, so gives the desired equality in the Grothendieck ring.
\end{proof}

\section{The configuration, effective zero-cycle, and Hilbert poschemes}\label{s.conf-zc-hilb}
\newcommand{\SG}{\mr{SG}}
\newcommand{\HG}{\mr{HG}}

Let $Z \rightarrow X$ be a map of schemes. We write $\Conf_X^k(Z)$ for the $k$th unordered configuration space of $Z$, relative to $X$,
\begin{equation}\label{eq.conf-quotient} \Conf_X^k(Z)= \left(\underbrace{Z \times_X Z \times_X \dotsm \times_X Z}_{k } \backslash \Delta\right)/\mfS_k \end{equation}
where $\Delta$ is the big diagonal where any two coordinates agree and $\mfS_k$ is the symmetric group on $k$ elements acting by permutation. We define the configuration poscheme of $Z$ over $X$
\[ \Conf^\bullet_X(Z) := \bigsqcup_{k=1}^\infty \Conf^k_X(Z) \textrm{ and its augmented variant }  \Conf^{+,\bullet}_X(Z):=\bigsqcup_{k=0}^\infty \Conf^k_X(Z).\]
The order relation is by inclusion  --- to make this precise and verify that this indeed defines a poscheme one may, for example, identify $\Conf^k_X(Z)$ with the reduced locus in the relative Hilbert scheme of length $k$ subschemes. The formation of the configuration poscheme commutes with arbitrary change of base.

When $Z/X$ is finite étale surjective of degree $d$, then $\Conf^\bullet_X(Z)$ has a maximum $X \cong \Conf^d_X(Z)$ (whose fiber over any geometric point $\overline{x}$ is the configuration of all $d$-points in $Z_{\overline{x}}$). As a consequence, if $Z \rightarrow X$ is quasi-finite surjective and $X$ is Noetherian, then the geometric fibers of $\Conf^\bullet_X(Z)$ have weak maxima, and thus \cref{maintheorem.mot-inc-exc} applies. This can be applied fruitfully to an arbitrary finite type $Z/X$ by considering the truncations $\Conf^{\leq k}_X(Z)$ which satisfy the hypotheses of \cref{maintheorem.mot-inc-exc} after restriction to an open locus where $Z \rightarrow X$ is quasi-finite of degree $\leq k$. This yields an approximate motivic inclusion-exclusion formula, \cref{theorem.approx-ie}-(1), which captures one of the main combinatorial methods used in the motivic stabilization arguments of \cite{vakil-wood:discriminants, bilu-howe:mot-eul-mot-stat}.

When $Z/X$ is projective, we will give a matching cohomological approximate inclusion-exclusion formula in \cref{theorem.approx-ie}-(2). To obtain it, we need to compactify $\Conf^\bullet_X(Z)$. There are (at least) two obvious candidates:
\begin{enumerate}
\item The Chow poscheme of effective zero-cycles.
\item The Hilbert poscheme of finite length subschemes (or its good component).
\end{enumerate}
Here by \emph{the} Chow poscheme of effective zero cycles, we actually mean the divided powers scheme $\Gamma^\bullet_X(Z)$, ordered by inclusion of zero cycles. This scheme was introduced by Rydh \cite{rydh:thesis}, and provides a canonical scheme structure on the Chow variety with respect to any sufficiently ample projective embedding.

The divided power $\Gamma^k_X(Z)$ is closely related to the symmetric power
\begin{equation}\label{eq.symmetric-power-definition} \Sym^k_X(Z):=\left(\underbrace{Z \times_X Z \times_X \dotsm \times_X Z}_{k}\right)/ \mfS_k. \end{equation}
Indeed, there is a natural universal homeomorphism
\[ \SG: \Sym^k_X(Z): \rightarrow \Gamma^k_X(Z)\]
that is an isomorphism when $Z/X$ is flat (in particular, when $X=\Spec \kappa$ for $\kappa$ any field), or when $X$ is of characteristic zero. It is the divided powers schemes, however, that provide a natural interpolation of symmetric powers from fields to arbitrary bases --- in particular, the formation of $\Gamma^\bullet_X(Z)$ is stable under arbitrary base change (symmetric powers are not!). This is more than just an aesthetic choice --- we are actually not certain whether the monoid structure on the symmetric powers induces a poscheme structure in full generality (i.e.\ whether it is cancelative in a scheme-theoretic sense), whereas we can prove this for $\Gamma^\bullet_X(Z)$.

For the purposes of proving an approximate cohomological inclusion-exclusion formula, it is possible to work with either the poscheme of effective zero cycles or the Hilbert poscheme. In either case, the cohomology of the graded pieces for the rank filtration will be identified with the compactly supported sign cohomology of configuration spaces, so that the specific choice of compactification is irrelevant --- more precisely, there is a natural map $\HG: \HP^\bullet_X(Z) \rightarrow \Gamma^{\bullet}_X(Z)$, and  it induces an isomorphism of the rank spectral sequences for cohomology.

In the majority of this section, we will thus focus on the approach via the poscheme of effective zero-cycles because it best highlights the role of symmetric powers and the relation with the Kapranov zeta function. However, in \cref{ss.hilbert}, we will briefly summarize the argument using punctual Hilbert schemes and also study a larger Vassiliev-style Hilbert poscheme that gives an exact cohomological inclusion-exclusion formula (at the price of introducing difficult-to-compute terms).

We now outline the contents of this section: in \cref{ss.conf-sym-posets} we study the configuration and symmetric posets of a finite set $Z$. These are simple and classical objects: the configuration poset of $Z$ is the lattice of subsets of $Z$, and the symmetric poset is the lattice of multisets of $Z$. In the latter case, it is often useful to interpret the lattice of multisets as the free commutative monoid on $Z$, with the poset ordering induced by the monoid multiplication. We prove the (surely well-known) result that the nerve of the symmetric poset deformation retracts to the nerve of the configuration poset, and observe that, for $Z=[n]=\{0, \ldots, n\}$, the nerve of the configuration poset is the barycentric subdivision of $\Delta^n$. Our computations in the scheme-theoretic case  are accomplished by reduction to these elementary results.

In \cref{ss.poscheme-zero-cycles}, we define the poscheme of effective zero cycles using the divided powers scheme of \cite{rydh:thesis}. The reader interested only in the characteristic zero (or topological) case may replace these with symmetric powers; the key point in any case is to show that the natural monoid structure induces a poscheme structure (see \cref{prop.gamma-monoid-cancellative}, \cref{remark.sym-monoid-cancellative}, and the paragraph following them).

In \cref{ss.rank-graded-zero-cycles} we show the cohomology of the graded pieces for the rank filtration on the poscheme of effective zero cycles are naturally identified (up to a shift) with the extension by zero of the sign local system on configuration spaces.

In \cref{ss.approx-ie} we prove approximate motivic and cohomological inclusion-exclusion, \cref{theorem.approx-ie}.  In both the motivic and cohomological setting, the result can be thought of as describing how closely the $k$-truncated poscheme of relative effective zero-cycles approximates the image of the morphism.

In \cref{ss.skeletal} we study the skeletal spectral sequence for the poscheme of effective zero cycles. In the case of rational coefficients, there is a natural quasi-isomorphism from the $E_1$-page to a complex consisting of the sign part of the cohomology of powers of the cartesian products --- we first learned of this latter complex from O.~Banerjee, who has studied it from a different perspective and has announced a spectral sequence with this complex on its $E_1$ page (that is surely closely related to the one studied here!). The terms of this complex can be identified with the terms of the $E_1$-page for the rank spectral sequence, but the advantage of the skeletal sequence is that the differential is completely explicit. There seems to be an intimate relation between the skeletal filtration and the rank filtration: in particular the skeletal spectral sequence is compatible with the rank filtration, and we can use it to also study also the differential on the $E_1$ page of the rank spectral sequence. In the rational case we find that it has the same kernel as the differential on the Banerjee complex, and it seems likely that the $E_1$ page for the rank spectral sequence is in fact quasi-isomorphic to the Banerjee complex, though we do not know how to prove this, or whether any deeper comparisons hold --- see \cref{remark.rank-from-skeletal}.

\subsection{Configuration and symmetric posets}\label{ss.conf-sym-posets}
\newcommand{\Support}{\mr{Support}}
For $Z$ a non-empty set, the configuration poset $\Conf^\bullet(Z)$ is the lattice of finite non-empty subsets of $Z$. The symmetric semigroup $\Sym^\bullet(Z)$ is the semigroup of finite non-empty multisets of $Z$ under disjoint union, i.e.\ the free commutative semigroup on $Z$. It is contained in the symmetric monoid $\Sym^{\bullet,+}(Z)$, the free commutative monoid on $Z$, where we allow also the empty set (which gives an identity element for the monoid operation). We may view $\Sym^{\bullet,+}(Z)$ (and thus also the subset $\Sym^\bullet(Z)$) as a poset, where $a \leq c$ if and only if there is a (necessarily unique) $b$ in $\Sym^{\bullet,+}(Z)$ such that $ab = c$. In other words, if we think of this poset as a category, then each morphisms is labeled by an element of $\Sym^{\bullet,+}(Z)$ (i.e.\ the underlying graph is the directed Cayley graph of the monoid with the identity vertex removed). In particular, we can think of elements of the nerve as labeled by multisets in two different ways. The poset interpetration gives that a $k$-simplex in $N(\Sym^{\bullet,+}(Z))$ is a chain of multisets $I_0 \leq \dotsm \leq I_k$, while the monoid interpretation gives that a $k$-simplex is an ordered list of finite multisets in $Z$, $(J_0, J_1, \ldots, J_k)$, with the bijection given by
$I_s = J_0 + \dotsm +J_s$, $J_s = I_s - I_{s-1}$ (here set $I_{-1} = \emptyset$), and where the subtraction exists by definition of the order relation and is unique because the monoid is cancellative. The simplicial subset $N\Sym^{\bullet} \subseteq N\Sym^{\bullet,+}$ consists of the simplices where $I_0=J_0 \neq \emptyset$, and $N\Conf^{\bullet,+} \subset N\Sym^{\bullet,+}$  consists of the simplices where each $I_s$ is a set (i.e.\ each element has multiplicity zero or one) or where the $J_s$ are pairwise disjoint.

There is a natural rank function $\Sym^{\bullet,+}(Z) \rightarrow \mbb{Z}_{\geq 0}$, given by cardinality, and the rank $k$ component can be described as
\[ \Sym^{k}(Z) = Z^k / \mfS_k. \]
In the monoid interpretation of the nerve, we can write
\[ N\Sym^{\bullet,+}(Z)_p = \bigsqcup_{ \underline{k}=(k_0, \ldots, k_p) \in \mbb{Z}_{>0}^p} \Sym^{\ul{k}}(Z) \]
where $\Sym^{\ul{k}}(Z) = \prod_{i=0}^{p} \Sym^{k_i}(Z).$ The nerve of $\Sym^{\bullet}(Z)$ consists of those simplices where $k_0 \neq 0$. In this description, the face map
\[ \delta_i: N\Sym^{\bullet,+}(Z)_p \rightarrow N\Sym^{\bullet,+}(Z)_{p-1} \]
is described as follows: for $0 \leq i < p$, it merges the $i$th and $(i+1)$th set, while for $i=p$ it forgets the $p$th set. The nerve of the configuration poset $\Conf^{\bullet}(Z)$ consists of the simplex spaces
\[ \Conf^{\underline{k}}(Z) = \left( Z^{\sum \underline{k}} \backslash \Delta \right) /  \mfS_{\underline{k}} \]
where $\mfS_{\underline{k}} = \prod_{i = 0}^p \mfS_{k_i}$ denotes the subgroup of $\mfS_{\sum \underline{k}}$ preserving subsequent blocks of size $k_i$.

\begin{example}\label{example.barycentric-simplex}
For $[n]=\{0, \ldots, n\}$, $N\Conf^\bullet([n])$ can be identified with the barycentric subdivision of the standard $n$-simplex $\Delta^n$. Indeed, the vertices correspond to non-empty subsets of $[n]$, i.e.\ to collections of vertices of $\Delta^n$, and a $k$-simplex corresponds to a chain of such under inclusion. If we consider the sub-poset $\Conf^{\leq n}([n])$ of configurations of rank at most $n$, obtained by removing the maximum from $C^\bullet([n])$, then the subcomplex $N\Conf^{\leq n}([n])$ is the barycentric subdivision of $\partial \Delta^n$.
\end{example}

An element of $\Sym^\bullet(Z)$ can be written as a formal sum $\sum_{z \in Z} n_z Z$ where $n_z \in \mbb{Z}_{\geq 0}$, $n_z = 0$ for all but finitely many $z \in Z$ and $n_z>0$ for at least one $z$; the rank is given by $\sum_{z \in Z} n_z$. There is a natural support map
\[ \Support: \Sym^\bullet(Z) \rightarrow \Conf^\bullet(Z), \sum n_x z \mapsto \{ z \in Z | n_z>0 \}. \]

We equip both $N\Sym^\bullet(Z)$ and $N\Conf^{\bullet}(Z)$ with the rank filtration. We then obtain immediately:
\begin{lemma}\label{lemma.sym-poset-def-retract} The support map is falling retraction of ranked posets (see \cref{remark.falling-rising-retractions}). Thus, for any sub-poset $\mc{P}$ of $\Sym^\bullet(Z)$ containing $\Conf^\bullet(Z)$,  $N\Conf^\bullet(Z)$ is a filtered deformation retract of $N\mc{P}$. In particular, if $Z$ is a finite set, then $N\mc{P}$ is contractible.
\end{lemma}

\subsection{The poscheme of effective zero-cycles}\label{ss.poscheme-zero-cycles}
In the following we assume that $Z/X$ satisfies the condition (AF)  that any finite set in a single fiber is contained in a quasi-affine open (see \cite[Paper III - Appendix A.1]{rydh:thesis}) --- this holds, in particular if $Z/X$ is quasi-projective. The symmetric powers $\Sym^k_X(Z)$ defined by the quotient \cref{eq.symmetric-power-definition} then exist as schemes, and we consider the symmetric power monoid
\[ \Sym^{\bullet,+}_X(Z) = \bigsqcup_{k=0}^\infty \Sym^k_X(Z) \]
where the monoid multiplication
\[ \Sym^{k_1}_X(Z) \times \Sym^{k_2}_X(Z) \rightarrow \Sym^{k_1+k_2}_X(Z) \]
 is induced by the quotient property from the natural map
\[  Z^{\times_X k_1} \times_X Z^{\times_X k_2} \rightarrow Z^{\times_X k_1+k_2} \rightarrow \Sym^{k_1+k_2}_X(Z) \]
and the identification
\[  \left( Z^{\times_X k_1} \times_X Z^{\times_X k_2} \right) / (\mfS_{k_1} \times \mfS_{k_2}) = \Sym^{k_1}_X(Z) \times \Sym^{k_2}_X(Z). \]
In characteristic zero, this provides a good notion of moduli of effective zero cycles, compatible with arbitrary base change, and the monoid map can be used to define a poscheme structure compatible with the obvious poset structure on geometric points (see below). In positive and mixed characteristic there are some well-known perversities of symmetric powers (see e.g.\ \cite{lundkvist:counterexamples}), and in particular it is not clear that the monoid structure defines a poscheme structure for a general $Z/X$.

We can address this by using divided powers schemes: in \cite[Paper III]{rydh:thesis}, there is defined, for any scheme $Z/X$ and $r \geq 0$, a divided powers schemes $\Gamma^r(Z/X)$, which we write here as $\Gamma^r_X(Z)$.  Because we have assumed that $Z/X$ satisfies the condition (AF), each $\Gamma^r_X(Z)$ is represented by a scheme. In loc.~cit.\ it is shown that
\begin{enumerate}
\item The formation of $\Gamma^r_X(Z)$ is stable under arbitrary change of base $X' \rightarrow X$.
\item If $X=\Spec A$, $Z=\Spec B$, then $\Gamma^r_X(Z)=\Spec\Gamma^r_A(B)$, where $\Gamma^r_A(B)$ is the $r$th divided power of $B$ as an $A$-module.
\item There is a commutative monoid structure on
\[ \Gamma^{\bullet,+}_X(Z)=\bigsqcup_{k=0}^\infty \Gamma^k_X(Z) \] compatible with the degree, i.e.\ restricting to maps
\[ \Gamma^{r}_X(Z) \times_X \Gamma^{s}_X(Z) \rightarrow \Gamma^{r+s}_X(Z). \]
\item There is a canonical universal homeomorphism that induces isomorphisms on residue fields
\[ \mr{SG}:\Sym^{\bullet,+}_X(Z) \rightarrow \Gamma^{\bullet,+}_X(Z). \]
The map $\mr{SG}$ is compatible with the monoid structures, and if $Z/X$ is flat (e.g.\ if $X=\Spec \kappa$ for $\kappa$ a field) or if $X/\mbb{Q}$, then it is an isomorphism.
\item There is a dense open non-degenerate locus $\Gamma^r_X(Z)^{\mr{nd}}$ such that $\mr{SG}$ restricts to an isomorphism
$\Conf^r_X(Z) \rightarrow \Gamma^r_X(Z)^{\mr{nd}}.$
\item If $Z/X$ is projective then, for any sufficiently high power $\mc{L}^n$ of a relatively ample bundle $\mc{L}$, $\Gamma^r_X(Z)$ is naturally identified with the Chow scheme of effective zero-cycles of degree $r$ on $Z$ for $\mc{L}^n$ (or, more accurately, its reduced subscheme is identified with the Chow variety, and this identification equips the latter with a natural scheme structure).
\end{enumerate}
By (4), $\Gamma^\bullet_X(Z)$ agrees with $\Sym^\bullet_X(Z)$ in characteristic zero and when $X$ is a geometric point, but $\Gamma^\bullet_X(Z)$ is better behaved in positive characteristic families. Over a geometric point, we will sometimes write the more familiar $\Sym^\bullet$ instead of $\Gamma^\bullet$, with the implicit understanding that these are canonically isomorphic in this case.

The following proposition says that the monoid scheme $\Gamma^{\bullet,+}_X(Z)$ is cancellative in a scheme-theoretic sense. As a consequence, we will see that the monoid multiplication induces a natural poscheme structure. In the proof, we use the following two standard properties of divided powers modules; for justifications we refer to \cite[Paper III]{rydh:thesis} and the references therein.
\begin{enumerate}
\item If $M \twoheadrightarrow N$ is a surjection of $A$-modules, then the natural map $\Gamma^r_A(M) \twoheadrightarrow \Gamma^r_A(N)$ is a surjection.
\item If $M$ is a flat $A$ module, then the natural map $\Gamma^r_A(M)\rightarrow (M^{\otimes_A r})^{\mfS_r}$ is an isomorphism.

\end{enumerate}

\begin{proposition}\label{prop.gamma-monoid-cancellative} For any $Z/X$ satisfying (AF), the map
\begin{align*} \Gamma^{r}_X(Z) \times \Gamma^{s}_X(Z) & \rightarrow \Gamma^{r}_X(Z) \times \Gamma^{r+s}_X(Z) \\
(x,y) & \mapsto (x, x+y)
\end{align*}
is a closed immersion.
\end{proposition}
\begin{proof}
We may assume $X=\Spec A$, $Z=\Spec B$, so we must show the  map
\begin{align*} \Gamma^{r}_A(B) \otimes \Gamma_A^{r+s}(B)& \rightarrow \Gamma^{r}_A(B) \otimes \Gamma_A^{s}(B) \\
	x \otimes 1 & \mapsto x \otimes 1 \\
	1 \otimes y & \mapsto m(y)
\end{align*}
is surjective. This diagram is functorial in the $A$-algebra $B$, so if we choose a surjection from a free $A$-algebra $F \twoheadrightarrow B$, we obtain a commutative diagram
\[\begin{tikzcd}
	{\Gamma^{r}_A(B) \otimes \Gamma_A^{r+s}(B)} & {\Gamma^{r}_A(B) \otimes \Gamma_A^{s}(B)} \\
	{\Gamma^{r}_A(F) \otimes \Gamma_A^{s}(F)} & {\Gamma^{r}_A(F) \otimes \Gamma_A^{s}(F)}
	\arrow[two heads, from=2-1, to=1-1]
	\arrow[from=1-1, to=1-2]
	\arrow[two heads, from=2-2, to=1-2]
	\arrow[from=2-1, to=2-2]
\end{tikzcd}\]
where the two vertical arrows are surjections because $\Gamma^{k}_A(\bullet)$ preserve surjections. To verify the top horizontal arrow is surjective, it thus suffices to show the lower horizontal arrow is surjective.

Now, since $F$ is free as an $A$-module, the divided powers are identified with the symmetric tensors, and under this identification the map $m$ is restriction
\[ (F^{\otimes r+s})^{\mfS_{r+s}} \hookrightarrow (F^{\otimes r+s})^{\mfS_r \times \mfS_s} = (F^{\otimes r})^{\mfS_r} \otimes (F^{\otimes s})^{\mfS_s}. \]
We thus need to show $(F^{\otimes r})^{\mfS_r} \otimes (F^{\otimes s})^{\mfS_s}$ is generated as an $A$-algebra by
\[ (F^{\otimes r})^{\mfS_r} \otimes 1 \textrm{ and } (F^{\otimes r+s})^{\mfS_{r+s}}. \]
Clearly it suffices to show the subalgebra $T$ generated by these contains
\[ 1 \otimes (F^{\otimes s})^{\mfS_s}. \]
Given a multiset $S$ of elements in $F$ with $|S| \leq n$, we define $t_n(S) \in (F^{\otimes n})^{\mfS_n}$ by choosing any ordering of elements $S= \{f_1\} + \dotsm + \{f_k\}$, $f_i \in F$, then summing up all elements in the $\mfS_n$-orbit of $f_1 \otimes \dotsm \otimes f_k \otimes 1 \otimes \dotsm \otimes 1$ to obtain $t_n(S)$. Since $F$ is a free $A$-module, it is easy to check that as we vary over all multisets $S$ with $|S|=n$, these span $(F^{\otimes n})^{\mfS_n}$. Thus it suffices to show that $T$ contains $1 \otimes t_s(S)$ for any multiset $S$ with $|S| \leq s$.

We argue this by induction on $|S|\leq s$: the base case $|S|=0$ is $1 \otimes 1 \in T$. Suppose it holds for $j < k$ and let $S$ be a multiset with $|S|=k$. We have $t_{r+s}(S) \in T$, but on the other hand we also have
\[ t_{r+s}(S) = 1 \otimes t_s(S) + \sum_{\emptyset \neq S' \leq S} t_{r}(S')\otimes t_{s}(S-S'). \]
By the inductive hypothesis it is clear that all terms in the sum on the right are also contained in $T$. Thus we conclude.
\end{proof}

\begin{remark}\label{remark.sym-monoid-cancellative}
If we use symmetric powers instead of divided powers, then arguing with geometric points we easily deduce that the corresponding map
\begin{align*} \Sym^{r}_X(Z) \times \Sym^{s}_X(Z) & \rightarrow \Sym^{r}_X(Z) \times \Sym^{r+s}_X(Z) \\
(x,y) & \mapsto (x, x+y)
\end{align*}
is finite and a universal homeomorphism onto its image, but we do not know if it is a closed immersion outside of the cases covered by the proposition (i.e.\ $X/\mbb{Q}$ or $Z/X$ flat, when the divided powers and symmetric powers agree). In the proof above we have crucially used right exactness of divided powers to reduce to the case of a free module, but this right exactness does not hold for symmetric powers \cite{lundkvist:counterexamples}.
\end{remark}

In the setting of \cref{prop.gamma-monoid-cancellative}, we find that $\Gamma^{\bullet,+}_X(Z)$ is a ranked poscheme with $\sle_{\Gamma^{\bullet,+}_X(Z)}$ the closed subscheme defined by the closed immersion
\begin{align*} \Gamma^{\bullet,+}_X(Z) \times \Gamma^{\bullet, +}_X(Z) & \rightarrow \Gamma^{\bullet,+}_X(Z) \times \Gamma^{\bullet,+}_X(Z) \\
(x,y) &\mapsto (x, x+y). \end{align*}
Indeed, the only non-trivial poscheme axiom to verify is the transitivity axiom (on $T$-valued points for an arbitrary test scheme $T$) that $a \leq b$ and $b \leq c$ implies $a \leq c$, and this is shown by representing $b=a+a'$ and $c=b+b'$ to obtain $c=a+(a'+b')$.

\subsubsection{Notation for symmetric powers, divided powers, and configurations}

It will be helpful to introduce some notation for labeled symmetric powers, divided powers, and configuration spaces, where the labelings are prescribed by a multiset. Explicitly, a finite multiset on a set $S$ can be identified with a function $\ul{a}: S \rightarrow \mbb{Z}_{\geq 0}$ with finite support, i.e.\ $\ul{a} \in \mbb{Z}_{\geq 0}$ such that $\ul{a}(s)=0$ for all but finitely many $s \in S$. Given such a multiset $\ul{a}$, we write
\[ \Gamma^{\ul{a}}_X(Z) := \prod_{s \in S} \Gamma^{\ul{a}(s)}_X(Z) \textrm{ and } \Sym^{\ul{a}}_X(Z)=\prod_{s \in S} \Sym^{a(s)}_X(Z).  \]
We can also make the canonical identification
\[ \Sym^{\ul{a}}_X(Z) = \left( \prod_{s\in S} Z^{\times_X \ul{a}(s)} \right) / \prod_{s \in S} \mfS_{\ul{a}(s)}. \]
There is a natural universal homeomorphism
\[ \SG: \Gamma^{\ul{a}}_X(Z) \rightarrow \Sym^{\ul{a}}_X(Z) \]
and inside the latter there is an open configuration locus $\Conf^{\ul{a}}_X(Z)$ where \emph{all} points are distinct, not just those in each group (i.e.\ formed by taking the quotient after removing the big diagonal). The homeomorphism $\SG$ restricts to an isomorphism above this locus, so that we may also consider $\Conf^{\ul{a}}_X(Z)$ as an open subscheme of $\Gamma^{\ul{a}}_X(Z)$.

We note, in particular, that if we write $\Gamma^{p}_X(Z)$, this is the $p$th relative divided powers scheme as above, but if we write $\Gamma^{[p]}_X(Z)$, then this means to interpret $[p]=\{0, 1,\ldots,p\}$ as a finite (multi)set so that by the above
\[ \Gamma^{[p]}_X(Z) = \Gamma^1_X(Z) \times \dotsm \times \Gamma^1_X(Z) = Z \times_X Z \times_X \dotsm \times_X Z = Z^{\times_X |[p]|}\]
where there are $p+1=|[p]|$ terms in the fiber product.

\begin{example}\label{remark.monoidal-nerve-poscheme-nerve}
As in the case of finite sets, it is useful to observe that we obtain a monoidal description of $N\Gamma^{\bullet,+}_X(Z)$. In this description, the $p$-simplices are
\begin{equation}\label{eq.monoidal-nerve-gamma} \bigsqcup_{\ul{a} \in \mbb{Z}_{\geq 0}^{[p]}} \Gamma^{\ul{a}}_X(Z), \end{equation}
and the map to the poscheme nerve is given by
\[ (c_0, \ldots, c_p) \mapsto (c_0, c_0+c_1, \ldots, c_0+c_1 + \dotsm + c_p). \]
The face and degeneracy maps are described as in the case of finite sets: in particular, the face map $\delta_i$ for $0 \leq i < k$ is given by summing the $i$th and $(i+1)$th coordinates, so sends the term corresponding $\ul{a}=(a_0, \ldots, a_p)$ to the term $\ul{a}=(a_0, \ldots, a_i + a_{i+1}, a_{i+2}, \ldots, a_p)$ and $\delta_k$ forgets the last coordinate, so sends the term corresponding to $\ul{a}=(a_0, \ldots, a_p)$ to $\ul{a}=(a_0, \ldots, a_{p-1})$.

Inside of \cref{eq.monoidal-nerve-gamma}, we can identify the $p$-simplices of $N\Conf^{\bullet,+}_X(Z)$ as $\bigsqcup_{\ul{a} \in \mbb{Z}_{\geq 0}^{[p]}} \Conf^{\ul{a}}_X(Z).$
In these interpretations $N\Gamma^{\bullet}_X(Z)$ and $N\Conf^{\bullet}_X(Z)$ correspond to $\ul{a} \neq 0$, the ranking is given by $\sum{\ul{a}}$, and the non-degenerate simplices are those corresponding to $\ul{a} \in \mbb{Z}_{>0}^{[p]}$ as well as $\ul{a}=0$ when $p=0$.

\end{example}

\subsection{Graded pieces for the rank filtration}\label{ss.rank-graded-zero-cycles}

Suppose $Z/X$ is projective. In this section, we compute the cohomology of the graded pieces for the rank filtration on $N\Gamma^\bullet_X(Z)$.

\begin{lemma}\label{lemma:punctual-sym-contract} Suppose $\kappa$ is  algebraically closed and $Z/\Spec \kappa$ is finite and reduced. If $\mc{P}$ is a subposcheme of $\Sym_{\Spec \kappa}^\bullet(Z)$ containing $\Conf_{\Spec \kappa}^\bullet(Z)$, then $N\mc{P}$ is contractible.
\end{lemma}
\begin{proof}
The category of finite reduced schemes over $\Spec \kappa$ is equivalent to the category of finite sets, and this equivalence is compatible with the formation of configuration spaces and symmetric powers. The result is then immediate from \cref{lemma.sym-poset-def-retract}.
\end{proof}

\begin{proposition}\label{prop.sym-reduced-cohomo-computation}
Let $Z/X$ be projective, and let $A=\mc{O}_L$ or $L$ for $L$ an algebraic extension of $\mbb{Q}_\ell$, $\ell$ invertible on $X$. Then, for $K \in D_{\Cons}(X,A)$,
\begin{align*} \tilde{C}(N(-\infty, \Gamma^k_X(Z))/\Gamma^k_X(Z), K ) &= j_!\tilde{C}(N(-\infty, \Conf^k_X(Z))/\Conf^k_X(Z), K) \\
&=  j_! (\ul{\sgn}[2-k] \otimes K), 	\end{align*}
where $j: \Conf^\bullet_X(Z) \rightarrow \Gamma^\bullet_X(Z)$ and $\ul{\sgn}$ denotes the sign local system.
\end{proposition}
\begin{proof}
By the projection formula, it will suffice to treat $K=A$. The poscheme $(-\infty, \Gamma^k_X(Z))/ \Gamma^k_X(Z)$ is proper, so, for $t:\Spec \kappa \rightarrow \Gamma^k_X(Z)$ a geometric point,
\[ \tilde{C}(N(-\infty, \Gamma^k_X(Z))/\Gamma^k_X(Z), A )_t  = \tilde{C}(N(-\infty, t)/\Spec \kappa, A) \]
We first show that this is 0 if $t$ does not factor through $\Conf^k_X(Z)$: in this case, the geometric support of $t$ is a closed reduced finite subscheme of degree strictly less than $k$, $F\subset Z_t$. There is a natural closed immersion of poschemes
\[ \Sym_{\Spec \kappa}^\bullet (F) = \Gamma^{\bullet}_{\Spec \kappa}(F)\hookrightarrow \Gamma^\bullet_X(Z) \]
and the map $t$ factors through $t': \Spec \kappa \rightarrow \Sym^\bullet_{\Spec \kappa} F$ and induces an isomorphism of poschemes $(-\infty, t') \rightarrow (-\infty, t)$. By assumption $C^\bullet_{\Spec \kappa}(F) \subset (-\infty, t')$, so \cref{lemma:punctual-sym-contract} shows $N(-\infty, t')$ is contractible and \cref{lemma.nerve-contractible-descent}-(2) concludes.

It remains to establish the identity over $\Conf^k_X(Z)$. In this case, if we pullback the entire situation to $\Conf^{(1,\ldots,1)}_X(Z)$, then we are considering the reduced cohomology complex of the constant poscheme $(-\infty, \{0, \ldots, k-1\}) \subseteq \Conf^\bullet(\{ 0, \ldots, k-1\})$ over $\Conf^{(1,\ldots,1)}_X(Z)$. By \cref{example.barycentric-simplex}, the nerve of this poset is identified with the barycentric subdivision of $\partial\Delta^{k-1}$. The reduced cohomology complex is naturally isomorphic to $A$ supported in degree $k-2$, and the action of $\mfS_{[k]}$ by reordering the vertices of the $(k-1)$-simplex yields the $\sgn$ representation on this copy of $A$, thus we conclude.
\end{proof}

We also have a version in the Grothendieck ring:
\begin{proposition}\label{theorem.symmetric-hilb-relative-computation-euler}
Let $Z/X$ be quasi-projective. Then
 \[ \tilde{\chi}(N(-\infty, \Gamma^k_X(Z))) = \tilde{\chi}(N(-\infty, \Conf^k_X(Z))) \in K_0(\Var/\Gamma^k_X(Z)) \]
\end{proposition}
\begin{proof}
Argue as in the previous proof to invoke \cref{theorem.poscheme-contraction-euler} over the complement of the configuration locus.\end{proof}

\subsection{Approximate-inclusion exclusion}\label{ss.approx-ie}

Let $X$ be a Noetherian scheme and let $f:Z\rightarrow X$ be a quasi-projective variety over $X$. We write $\dim_{/X}Z$ for the maximum over the dimensions of all irreducible components of geometric fibers of $f$. We write $X_{\geq k}$ for the closure of the image of $\Conf^k_X(Z)$ in $X$,  i.e.\ the closure of the locus of geometric points $\overline{x}$ of $X$ such that $Z_{\overline{x}}$ contains at least $k$ geometric points. If $f$ is proper, then note that $X_{\geq 1}=f(Z)$. We define $X_I$ for $I$ an interval in the obvious way --- e.g.\ $X_{[1,k]}=X_{\geq 1}\backslash X_{\geq k+1}$. We write $X_\infty = \bigcap_{k} X_{\geq k}$.

The main idea is that over $X_{[1,k]}$, the $k$-truncated poscheme of effective zero-cycles satisfies cohomological descent, so that under further constraints on the dimensions of the fibers its nerve provides a good approximation of $f(Z)$.

\begin{theorem}\label{theorem.approx-ie}
Let $\kappa$ be an algebraically closed field, let $X/\kappa$ be a variety, and let $f:Z \rightarrow X$ be a surjective map of varieties. Suppose $k>0$ is such that
\begin{equation}\label{eq.approx-ie-dimension-hypothesis} \dim(X_{>k}) \geq \dim  X_\infty + k \dim_{/X}Z. \end{equation}
\begin{enumerate}
	\item \emph{Motivic approximate inclusion-exclusion:}
\begin{align*} [X] & \equiv \chi(N\Conf^{\leq k}_{X}(Z)) \\
&\equiv \sum_{p \geq 0} (-1)^p \sum_{ \ul{a} \in \mbb{Z}_{>0}^{[p]}, \sum{\ul{a}} \leq k} [\Conf^{\ul{a}}_X(Z)] \\
& \mod \Fil^{-\dim X_{>k}} K_0(\Var/\kappa).
\end{align*}
\item \emph{Cohomological approximate inclusion-exclusion:}\\
Suppose furthermore that $f$ is projective, and let $\mc{F} \in \Cons(X, A)$ for $A=L$ or $\mc{O}_L$, $L$ an algebraic extension of $\mbb{Q}_\ell$ with $\ell$ invertible in $\kappa$. The adjunction unit $\mc{F} \rightarrow R\epsilon_*\epsilon^*\mc{F}$ for the augmentation $\epsilon: N\Gamma^{\leq k}_X(Z) \rightarrow X$ induces isomorphisms for all $q \geq k+ 2\dim X_{>k} + 1$
\[ H^q_c(X, \mc{F}) \xrightarrow{\sim} H^q_c(X, R\epsilon_* \epsilon^*\mc{F}), \]
where we note that if $X/\kappa$ is proper then the righthand side is equal to
\[ H^q(N\Gamma^{\leq k}_X(Z) , \epsilon^*\mc{F}).\]
\end{enumerate}
\end{theorem}

\begin{remark}
At the price of complicating the proof by working with a finer stratification, the inequality \cref{eq.approx-ie-dimension-hypothesis} can be replaced with
\[ \dim X_{>k} \geq \max_{x \in X}{ (k\dim Z_{x} + \dim \overline{\{x\}})}. \]
\end{remark}

\begin{remark}\label{remark.rank-spectral-truncated-conf}
Cohomological approximate inclusion-exclusion as in \cref{theorem.approx-ie}-(2) can be completed to a cohomological inclusion-exclusion \emph{formula} by combining with the rank spectral sequence for $\Gamma^{\leq k}_X(Z)$,
\[ E_1^{p,q}=\begin{cases}H^{q - p + 1}_c(\Conf^{p+1}_X(Z), \ul{\sgn}\otimes \mc{F}) & 0 \leq p \leq k-1 \\
0 & \textrm{otherwise} \end{cases} \Rightarrow H^{p+q}_c(X, R\epsilon_*\epsilon^*\mc{F}), \]
where the computation of the terms $E_1^{p,q}$ follows from \cref{prop.sym-reduced-cohomo-computation}. This rank spectral sequence will also be used fiberwise over $X$ in the proof of \cref{theorem.approx-ie}.

In the case of rational coefficients, it is more useful to use the closely related skeletal spectral sequence in place of the rank spectral sequence --- see \cref{ss.skeletal}.
\end{remark}

\begin{proof}[Proof of \cref{theorem.approx-ie}]\hfill\\
\emph{Motivic case:} We will use motivic inclusion-exclusion. To do so, we consider the $k$-truncated configuration poscheme $\Conf^{\leq k}_X(Z)$. For any geometric point $\overline{x}$ in $X_{[1,k]}$, $\Conf^{\leq k}_{X_{[1,k]}}(Z)_{\overline{x}}=\Conf^{\leq k}(Z_{\underline{x}})$ has a weak maximum because, by definition of $X_{[1,k]}$, $Z_{\underline{x}}^{\red}$ is finite reduced of degree $\leq k$. By \cref{maintheorem.mot-inc-exc}, we then find
\[ [X]=[X_{[1,k]}] + [X_{>k}]=\chi(N\Conf^{\leq k}_{X_{[1,k]}}(Z)) + [X_{>k}]. \]
Clearly we have $[X]\equiv [X]+[X_{>k}]  \mod \Fil^{-\dim_{/S} X_{>k}}$, so that it remains to compute the terms appearing in the Euler characteristic. If we compute the latter using \cref{remark.monoidal-nerve-poscheme-nerve}, then from the identity
\[ [\Conf_{X}^{\ul{a}}(Z)] = [\Conf_{X_{[1,k]}}^{\ul{a}}(Z)] + [\Conf_{X_{>k}}^{\ul{a}}(Z)], \]
we find that it suffices to show for any $\ul{a} \in \mbb{Z}_{>0}^{[p]}$ with $\sum \ul{a} \leq k$, that
\[ \dim_{/S}\Conf_{X_{>k}}^{\ul{a}}(Z) \leq \dim_{/S} X_{>k}. \]

To see this, clearly it suffices to treat the case of the unordered configuration space $\Conf^j$ for $j\leq k$. Then
\[ \dim_{/S} \Conf^j{X_{>k}}(Z) \leq \max(\dim \Conf^j_{X_{(k,\infty)}}(Z), \dim \Conf^j_{X_{\infty}}(Z)) \]
Since $Z$ is quasifinite over $X_{(k,\infty)}$, the first term in the maximum is $\leq \dim X_{>k}$, while the second term is bounded by $j\dim_{/X}Z + \dim X_\infty$, so that the hypothesis \cref{eq.approx-ie-dimension-hypothesis} yields the result.

\hfill\\
\noindent\emph{Cohomological case:}
Let $K$ denote the cone of $	\mc{F} \rightarrow R\epsilon_*\epsilon^*\mc{F}$ so that there is an exact triangle $\mc{F} \rightarrow R\epsilon_*\epsilon^*\mc{F} \rightarrow K.$  By the corresponding long exact sequence, it suffices to show $H^q_c(X, K)$ vanishes in degree $q\geq k+ 2\dim_{X_{>k}}$.

We first note that the hypotheses imply that $\mc{P}$ satisfies cohomological descent over $X_{[1,k]}$, so $K$ is supported in the closed set $X_{>k}$. Indeed: for $\overline{x}$ a geometric point of $X_{[1,k]}$, $\Sym^{\leq k}(Z_{\overline{x}})={\Gamma^{\leq k}_X(Z)}_{\overline{x}}$ is universally homeomorphic to $\Sym^{\leq k}( (Z_{\overline{x}})^{\red})$. The latter is contractible by \cref{lemma:punctual-sym-contract}, so that we obtain cohomological descent by \cref{lemma.nerve-contractible-descent}.

We now decompose $X_{>k}$ into the open $X_{(k,\infty)}$ and closed $X_\infty$, so that we have a long exact sequence
\[ \dotsm \rightarrow H^q_c(X_{(k,\infty)}, K) \rightarrow H^q_c(X_{>k}, K)=H^q_c(X, K) \rightarrow H^q_c(X_\infty, K) \rightarrow \dotsm \]
so it suffices to prove the vanishing on $X_{(k,\infty)}$ and $X_{\infty}$ separately.

Over $X_{(k,\infty)}$, $f$ is finite. As a consequence, we claim $K$ is supported in degrees $\leq k-1$ over $X_{(k,\infty)}$: It suffices to check this for $R\epsilon_* \epsilon^* \mc{F}$, and, by proper base change it suffices to check at a geometric point of $X_{(k,\infty)}$. Applying the rank spectral sequence as in \cref{remark.rank-spectral-truncated-conf} to such a fiber, we find the $E_1$ terms are zero for $q>0$ and for $p>k-1$, thus we find that $R\epsilon_*\epsilon^* \mc{F}$ is supported in degrees $\leq k-1$, verifying the claim. Then,  $H^q_c(X_{(k,\infty)}, K)=0$ for
\[ q \geq k + 2\dim X_{>k} > k-1 + 2\dim X_{(k,\infty)}. \]

Similarly, over $X_{\infty}$, if we apply the rank spectral sequence to geometric fibers, we find that $K$ is supported in degrees $\leq 2 k \dim_{/X}(Z) + k - 1$.  Thus $H^q_c(X_{>k}, K)=0$ for
\[ q \geq 2\dim X_{>k} + k > 2 \dim X_\infty + 2 k \dim_{/X}(Z) + k-1 \]
where here we have used \cref{eq.approx-ie-dimension-hypothesis}.
\end{proof}

\begin{remark}
The instance of \cref{maintheorem.mot-inc-exc} used in the proof of \cref{theorem.approx-ie}-(1) can be replaced with the earlier version of \cite[Theorem 7.2.4]{bilu-howe:mot-eul-mot-stat} that is proved using motivic Euler products (by a method analogous to the proof of the classical inclusion-exclusion formula described after \cref{eq.indicator-inc-exc-product}).
\end{remark}

\subsection{The skeletal spectral sequence}\label{ss.skeletal}

We now study the skeletal spectral sequence for the effective zero cycles poscheme. We assume $Z/X$ is projective and $X/S$ is proper. Below we will consider cohomology relatively over $S$; to that end, we write $\epsilon$ for the augmentation $N\Gamma^\bullet_X(Z) \rightarrow S$.

\newcommand{\Hom}{\mr{Hom}}

\subsubsection{The skeletal spectral sequence}
Suppose $K$ is a bounded below complex of sheaves on $N\Gamma^\bullet_X(Z)$. Using the monoidal description of the nerve (see \cref{remark.monoidal-nerve-poscheme-nerve}), the skeletal sequence is
\[ E_1^{s,t}(K) = \bigoplus_{ \underline{a} \in \mbb{Z}_{>0}^{[s]}} R^t(\Gamma_X^{\ul{a}}(Z)\rightarrow S)_* K|_{\Gamma_X^{\ul{a}}(Z)} \Rightarrow R^{s+t}\epsilon_* K \]
Here the differential from $E^{s-1,t}$ to $E^{s,t}$ is given by  $\sum_{i=0}^{s} (-1)^i \delta_i^*$ where $\delta_i$ is as described in \cref{remark.monoidal-nerve-poscheme-nerve}.

There is a natural rank filtration on the complex $E_1(K)$, defined by taking $\Fil^p(E_1(K))$ to be the subcomplex consisting of summands with $\sum \underline{a}>p$. It is clear that this filtration commutes with the rank filtration on $K$: the natural map of spectral sequences $E(\Fil^p(K)) \rightarrow E(K)$ is an injection on the first page and identifies $E(\Fil^p(K))_1$ with $\Fil^p(E(K))_1$. More generally we obtain a canonical identification $E(\Gr^{[a,b]}K)_1=\Gr^{[a,b]}E(K)_1$. Where $\Gr^{[a,b]}=\Fil^a/\Fil^{b+1}$.

In all of these, we may pass to the quasi-isomorphic complex $E_1^\circ$ of cochains in the kernel of the degeneracy maps. Term by term, this is given by the summands corresponding to $\underline{a} \in \mbb{Z}_{>0}^{[s]}$ (this follows from the description of the degenerate simplices in \cref{remark.monoidal-nerve-poscheme-nerve}).

\subsubsection{The Banerjee complex}
For $K$ a complex of sheaves on $X$, we will compare $E(K)_1^\circ$, together with its filtration, to a simpler complex, which we define now: we first consider the non-degenerate complex whose degree $p$ term is
\[ R^\bullet(Z^{\times_X [p]} \rightarrow S)_*K \]
where the differential from degree $p-1$ to $p$ is $\sum_{i=0}^{p} (-1)^i \alpha_i^*$, for $\alpha_i$ is the forget the $i$th coordinate map. This is the non-degenerate subcomplex of the $E_1$ page of a spectral sequence for the cohomology over $S$ of the pullback of  $K$ to the simplicial scheme $[p] \mapsto Z^{\times_X [p]}$, but we will not use this here. As we learned from O.~Banerjee, the isotypic components for the $\sgn$ character form a subcomplex
\[ B(K)^p = R^\bullet(Z^{\times_X [p]} \rightarrow S)_*K[\sgn]. \]
That this is a subcomplex follows from  the elementary
\begin{lemma}\label{lemma.perm-delete-commutation}
For $\sigma \in \mfS_{[p]}$,
\[ \alpha_j \circ \sigma =    (p\; p-1 \; \dotsm \; j) \circ \sigma \circ (i \; i+1 \; \dotsm \; p)  \circ \alpha_i \]
where $\sigma(i)=j$.
\end{lemma}
\begin{proof}
Left to reader.
\end{proof}

Indeed, given this lemma, we find that for $c$ in degree $p$,
\begin{align*}
 \sigma\cdot d(c)&= (\sigma^{-1})^* \sum_{j=0}^{p} (-1)^j \alpha_j^* c \\
 &= \sum_{j=0}^{p}  (-1)^j \alpha_{\sigma(j)}^* \left( (p\; p-1 \; \dotsm \; j) \circ \sigma^{-1} \circ (\sigma(j) \; \sigma(j)+1 \; ... \; p) \right) ^*(c) \\
 & = \sum_{j=0}^{p} (-1)^j \alpha_{\sigma(j)}^*\left( (p \; \dotsm \; \sigma(j)) \sigma (j\; \dotsm \; p) \right) \cdot c \\
 &= \sum_{j=0}^{p}  (-1)^j \alpha_{\sigma(j)}^*  \sgn(\sigma) (-1)^{\sigma(j)-p}(-1)^{p-j}  c \\
 & = \sgn(\sigma) \sum_{j=0}^{p}  (-1)^{\sigma(j)} \alpha_{\sigma(j)}^* c \\
 & = \sgn(\sigma) d(c).
\end{align*}

\newcommand{\asym}{\mathrm{asym}}
We define an antisymmetrization map
\[ \asym:  E(K)_1^\circ \rightarrow B(K) \]
summand by summand as follows:
\begin{enumerate}
\item For $c$ in the degree $p$ summand corresponding to $\ul{a}=[p]=(1,\ldots,1)$, for which $\Gamma^{[p]}_X(Z)=Z^{\times_X [p]}$, we define
\[ \asym(c)=\sum_{\sigma \in \mfS_{[p]}} \sgn(\sigma)\sigma \cdot c \]
\item For $c$ in any other summand, $\asym(c)=0$.
\end{enumerate}

\begin{theorem}\label{theorem.banerjee-qiso}
	For any bounded below complex $K$, the map
	\[ \asym: E(K)_1^\circ \rightarrow B(K) \]
	is a map of filtered complexes. If $K$ is a complex of $L$-modules, then it is a filtered quasi-isomorphism.
\end{theorem}
\begin{proof}
It is clear from the definition that the map preserves the filtration, but we must verify that it is a map of complexes, i.e.\ that $d\circ \asym=\asym \circ d$. We  check this summand by summand.

Suppose $c$ is a local section of the degree $p-1$ summand corresponding to $\ul{a}$. We consider three different cases:
\begin{enumerate}
\item $\ul{a}=[p]$, i.e.\ $\ul{a}(i)=1$ for all $i \in [p-1]$
\item $\ul{a}(i)=2$ for exactly one $i \in [p-1]$ and is $1$ for all other $i$.
\item the remaining $\ul{a}$.
\end{enumerate}

In the third case, the $[p]$ component of $\delta_i^*c$ is zero for all $0 \leq i \leq p$, so that $\asym(d(c))=0=d(0)=d(\asym(c)),$ as desired.

In the second case, let $0\leq i \leq p-1$ be the unique index such that $\ul{a}(i)=2$. Then $\delta_j^*c$ has non-trivial $[p]$-component if and only if $i=j$. However, $\delta_j^*c$ is invariant under action of the transposition $(i \; i+1)$, thus its antisymmetrization is zero, so $\asym(d(c))=0=d(0)=d(\asym(c))$.

In the first case, $\delta_i^*c$ is non-zero only for $i=p$, in which case it is concentrated in the $[p]$-component. Thus $d(c)=(-1)^p \delta_p^*c = (-1)^p \alpha_p^*c.$ Using \cref{lemma.perm-delete-commutation}, its image under antisymmetrization is
\begin{align*} (-1)^p \sum_{\sigma \in \mfS_{[p]}}  \sgn(\sigma) \sigma \cdot  \alpha_p^*c &=  (-1)^p \sum_{\sigma \in \mfS_{[p]}} \sgn(\sigma)  (\alpha_p \circ \sigma^{-1})^* c \\
&= (-1)^p \sum_{\sigma \in \mfS_{[p]}} \sgn(\sigma) \left( \sigma^{-1} (\sigma(p)\; \sigma(p)+1 \; \dotsm \; p) \circ \alpha_{\sigma(p)} \right)^* c \\
&= (-1)^p \sum_{\sigma \in \mfS_{[p]}}  \sgn(\sigma) \alpha_{\sigma(p)}^*  (p\; \dotsm \; \sigma(p)) \sigma \cdot c \\
&=  \sum_{\sigma \in \mfS_{[p]}}  (-1)^{\sigma(p)} \alpha_{\sigma(p)}^*  \left( \sgn((p\; \dotsm \; \sigma(p))\sigma) (p\; \dotsm \; \sigma(p)) \sigma \cdot c \right)
\end{align*}
If we group terms according to the coset for $\mfS_{[p-1]}$, i.e.\ according to $\sigma(p)$, then we obtain the differential in $B(K)$ applied to the antisymmetrization of $c$, as desired.

It remains to see that when $K$ is a complex of $E$ modules, then this is a filtered quasi-isomorphism. We then consider the map
\[ \Gr^p(E(K)_1^\circ) \rightarrow \Gr^{p}(B(K)). \]
The complex on the right is concentrated in degree $p$ where it equals
\[ R^\bullet(Z^{\times_X [p]} \rightarrow S)_*K [\sgn]. \]
Multiplication by $1/(p+1)!$ and inclusion into the $[p]$-component of $\Gr^p(E(K)_1^\circ)$ give a section. This section is a quasi-isomorphism: one can check that
\[ \Gr^p(E(K)_1^\circ) \cong \left(\tilde{C} \otimes_E R^\bullet (Z^{\times_X [p]} \rightarrow S)_*K  \right)^{\mfS_{[p]}} \]
where $\tilde{C}$ is the non-degenerate relative cohomology complex for $N\Conf^{\leq p}([p]) \subset N\Conf^{\bullet}([p])$. By \cref{example.barycentric-simplex}, this inclusion is naturally identified with the barycentric subdivision of the inclusion of $\partial\Delta^p$ in $\Delta^p$, thus $\tilde{C}$ is quasi-isomorphic as a complex of $L[\mfS_{[p]}]$-modules to $\ul{\sgn}[-p]$, and we conclude.
\end{proof}

Suppose now that $K \in D_{\Cons}(X,L)$. Then \cref{theorem.banerjee-qiso} gives an alternate computation of $R\epsilon_*\Gr^p \epsilon^*K$ (computed earlier by \cref{prop.sym-reduced-cohomo-computation}). Indeed, since the map $\asym$ is a filtered quasi-isomorphism, we have
\[ E_1(\Gr^p \epsilon^*K)=\Gr^{p}E_1(K) \simeq \Gr^{p}B(K)= R^\bullet (Z^{\times_X [p]} \rightarrow S)_* K[\sgn] \]
where the right hand-side is supported in the column $s=p$. Thus the sequence degenerates at $E_1$ to give isomorphisms
\[R\epsilon_*^{i+p}(\Gr^{p}\epsilon^*K) \cong R^i(Z^{\times_X [p]} \rightarrow S)_*K[\sgn]. \]
As we learned from O.~Banerjee, there is a simple direct argument that shows these match with the expression in terms of $\ul{\sgn}$ cohomology of relative configuration spaces as computed by \cref{prop.sym-reduced-cohomo-computation}. This will be explained and compared to  Grothendieck ring computations of Vakil and Wood in \cref{ss.symmetric-reduced-incidence}.

\begin{remark}\label{remark.rank-from-skeletal}
Using the skeletal spectral sequence and \cref{theorem.banerjee-qiso}, we can also compute information about the differential of the rank spectral sequence. Indeed, this differential is induced by the connecting homomorphism in cohomology from the short exact sequence
\begin{equation}\label{eq.connecting-exact} 0 \rightarrow \Gr^{p}\epsilon^*K \rightarrow \Gr^{[p-1, p]}\epsilon^*K  \rightarrow \Gr^{p-1} \epsilon^*K \rightarrow 0. \end{equation}
In particular, the kernel of the connecting homomorphism in $R^i\epsilon_*\Gr^{p-1} \epsilon^*K$ is equal to the image of $R^i\epsilon_* \Gr^{[p-1,p]}\epsilon^*K $ under the natural map. Using the filtered quasi-isomorphism of \cref{theorem.banerjee-qiso} and compatibility of the rank filtration with the $E_1$ page of the skeletal sequence, we can deduce that the skeletal sequence for $\Gr^{[p-1,p]}\epsilon^*K$ degenerates on $E_2$ (because it is supported in the columns $s=p$ and $s=p-1$), and then compute this connecting map on the $E_2$ page. We deduce that the kernel of the differential on the $E_1$ page of the rank spectral sequence is the same as the kernel of the differential in the Banerjee complex (assuming the coefficients are rational). As indicated in the introduction to this section, it would be interesting to understand if there is a deeper connection between the skeletal and rank spectral sequences.
\end{remark}

\subsection{Hilbert schemes and Vassiliev sequences}\label{ss.hilbert}

As indicated in the introduction to this section, analogs of \cref{prop.sym-reduced-cohomo-computation} and \cref{theorem.approx-ie} also exist using the Hilbert poscheme of points
\[ \HP^\bullet_X(Z):=\bigsqcup_{k=1}^\infty \HP^k_X(Z). \]
Here the poscheme structure is induced by inclusion of closed subschemes (note there is no monoid structure here!), and there is a natural open immersion $\Conf^\bullet_X(Z) \subseteq \HP^\bullet_X(Z)$ induced by the universal configuration over $\Conf^\bullet_X(Z)$. The key step in carrying out the argument in this setting is to prove the correct analog of the contractibility \cref{lemma:punctual-sym-contract}; in this case the right statement is that, for $Z/\Spec \kappa$ finite (but not necessarily reduced), geometric support induces a deformation-retraction from $\HP^\bullet_X(Z)^{\red}$ to $\Conf^\bullet_X(Z)^\red=\Conf^\bullet_X(Z^\red)$. There is a natural Hilbert--Chow morphism $\HG: \HP^\bullet_X(Z) \rightarrow \Gamma^\bullet_X(Z)$ that restricts to an isomorphism on $\Conf^\bullet_X(Z)$ and induces the identity map on the $E_1$ page of the corresponding rank spectral sequences, so that in a strong sense one does not obtain anything new from this approach.

If we instead take a larger subset of the relative Hilbert poscheme allowing fibers of arbitrary dimension, then we obtain  an exact inclusion-exclusion sequence and formula. This is closely related to certain Vassiliev spectral sequences, as we explain briefly now. These results will not be used in the rest of the paper.

Let $X$ be a locally Noetherian scheme and let $Z/X$ be quasi-projective. Then we can form the relative Hilbert scheme, $\Hilb_{X}(Z)$ --- recall \cite{FGA, FGA-explained} that this is the moduli scheme over $X$ of proper flat families of closed subschemes of $Z$; and it is equipped with a natural poscheme structure by the inclusion relation.  By definition, its formation is compatible with base change.

If we fix a relatively ample $\mc{L}/Y$, then for any numerical polynomial $f$ we obtain a clopen subscheme $\Hilb^{\mc{L},f}_{Y/X}$ parameterizing families with Hilbert polynomial $f$. Given any  set of numerical polynomials $S$, we write $\Hilb^{\mc{L},S}_{Y/X}$ for the union of these clopen subschemes, whose formation is again compatible with base change.

\begin{proposition} If $Z/X$ is projective and $S$ is finite, then $\Hilb_{Z/X}^{\mc{L},S}$ is a projective poscheme over $Z$ admitting a rank function. Moreover, there is a finite set of numerical polynomials that appear as the Hilbert polynomial of a geometric fiber $Z_{\overline{x}}$, and if $S$ contains this finite set then $\Hilb_{Z/X}^{\mc{L},S}$ admits maxima over geometric points in the image of $Z$.
\end{proposition}
\begin{proof}
Properness is a standard result for Hilbert schemes, and it admits a rank because it admits a strictly increasing map to the finite poset $S$ (we say $f \leq g$ if $f(n) \leq g(n)$ for $n \gg 0$), and any finite poset admits a rank function.

By Noetherian induction and generic flatness, there is a finite set of polynomials that appear as the Hilbert polynomial of $(Z_x, \mc{L}_y)$ for $x$ a geometric point of $X$. The maximum over a geometric point $x$ is, of course, the point of the Hilbert scheme corresponding to the fiber $Z_x$ (by the same argument one obtains a maximum over any $T \rightarrow X$ such that $Z_T/T$ is flat).
\end{proof}

As a consequence, \cref{main.spectral-sequences} and the choice of a rank function give an inclusion-exclusion spectral sequence for the cohomology of the closed subscheme $f(Z)$ as well as the compactly supported cohomology of its complement, and \cref{maintheorem.mot-inc-exc} allows us to compute the class $[f(Y)]$ (and thus trivially also the class of its complement $[X- f(Z)]=[X]-[f(Z)]$). For a suitable choice of $S$, the Hilbert poscheme $\Hilb_{Z/X}^{\mc{L},S}$ includes a truncated punctual Hilbert poscheme $\HP^{\leq k}_X(Z)$, so that, for a suitable choice of rank function, the resulting spectral sequence will admit a natural map to the approximate inclusion-exclusion sequence studied above that is a quotient map on $E_1$; the other terms thus give precise control over the error in approximate inclusion-exclusion.

\begin{example}\label{example.singular-locus} Suppose we have a family of varieties $f: V \rightarrow X$ (e.g.\ the universal degree $d$ hypersurface in $\mbb{P}^n$ as in \cref{s.smooth-hypersurface}), and $Z \subseteq V$ is the relative singular locus. Then we can write the the discriminant locus $\Delta$ (consisting of the points $x \in X$ such that the fiber $V_x$ is singular) as $\Delta=f(Z)$.  The spectral sequence from the Hilbert poscheme of $f|_Z$ is a Vassiliev-type spectral sequence for the compactly supported cohomology of $\Delta$, and motivic inclusion-exclusion for $f|_Z$ gives an exact combinatorial analog in the Grothendieck ring.
\end{example}

\subsubsection{Geometrically reduced variant}

\begin{lemma}There is a finite set of numerical polynomials $S_\mr{gr}$ that can occur as the Hilbert polynomial of $Z_{\overline{x}}^{\red}$ for $\overline{x}$ a geometric point.
\end{lemma}
\begin{proof}
We argue by Noetherian induction: assume $X$ is irreducible and let $\overline{\eta}$ be a geometric point lying above the generic point $\eta$ of  $X$.  Suppose $Z_{\overline{\eta}}$ is reduced --- then, by \cite[\href{https://stacks.math.columbia.edu/tag/0C0E}{Tag 0C0E}]{stacks-project}, the geometric fibers are reduced in an open locus, and, by generic flatness, if we restrict to a potentially smaller open then they all have the same Hilbert polynomial.  Otherwise, we can make a purely-inseparable extension $k(\eta)'$ of $k(\eta)$ such that $Z_{k(\eta)'}$ is no longer reduced and its reduced subscheme $(Z_{k(\eta)'})^\red$ is geometrically reduced. We may spread out to a dominant map $X' \rightarrow X$ for $X'$ irreducible of the same dimension such that $((Z_{X'})^\red)_{\overline{\eta}}$ is reduced, then we conclude as above.
\end{proof}

Thus we obtain a variant by taking $S \supseteq S_{\mr{gr}}$ and then considering the open (by \cite[\href{https://stacks.math.columbia.edu/tag/0C0E}{Tag 0C0E}]{stacks-project}) geometrically reduced locus $\Hilb_{X/Y}^{\mr{gr}, \mc{L}, S} \subset \Hilb_{X}^{\mc{L}, S}(Z)$. This is a (no longer proper) poscheme over $X$ with geometric maxima over points in the image of $f$ --- the maximum at a geometric point $\overline{x}$ is given by the point corresponding to $Z_{\overline{x}}^{\mr{red}}$. The corresponding motivic inclusion-exclusion formula is more useful than the non-reduced variant and mirrors cohomological calculations as in \cite{vassiliev:homology-hypersurfaces, das:cubic-surfaces}. For example, in the setting of \cref{example.singular-locus}, this allows us to consider the geometric support of the singular locus rather than its scheme structure.

\begin{remark}\label{remark.vassiliev} Since this poscheme is not proper, we do not obtain a cohomological spectral sequence.
In the special case of discriminant loci for polynomials over $\mbb{C}$ of \cite{vassiliev:homology-hypersurfaces}, however, one obtains a spectral sequence in cohomology by equipping the geometric realization of the $\mbb{C}$-points of $\Hilb_{X}^{\mr{gr}, \mc{L}, S_\mr{gr}}(Z)$ with a different topology, obtained as a quotient topology from the geometric realization of its closure by contracting non-reduced schemes at the boundary to their geometric support. We do not see a way to understand such a construction using purely simplicial methods without passing to a geometric realization, so we cannot apply it in the general scheme-theoretic case. However, in the part of the Hilbert poscheme corresponding to finite subschemes, this reduction is equivalent to our earlier computation that the terms on the $E_1$ page only depend on the configuration spaces.
\end{remark}

\section{Incidence algebras and Möbius inversion}\label{s.incidence-algebras}

In this section we construct motivic and sheaf-theoretic incidence algebras of poschemes and study Möbius inversion therein. The most important case for us, treated in \cref{ss.symmetric-reduced-incidence} below, is the reduced incidence algebra of the poscheme of effective zero cycles, where the Möbius element gives a lift of the inverse Kapranov zeta function --- from this we will recover Vakil--Wood's inversion formula for the Kapranov zeta function and also motivate a definition of cohomological special values of the Kapranov zeta function (see \cref{ss.cohom-special-values}). Kobin \cite{kobin} has also recently raised the question of finding an incidence algebra interpretation of the Kapranov zeta function.

Because reduced incidence algebras are a slightly ad hoc construction (they require a notion of when two intervals in a poset are the same), in \cref{ss.general-incidence-algebras} we first give the construction of the non-reduced incidence algebra for a general poscheme. In both cases  the discussion for poschemes is preceded by some recollections on the classical constructions for posets.

\subsection{Incidence algebras}\label{ss.general-incidence-algebras}
We briefly recall incidence algebras of posets and then give a categorification.

\subsubsection{The incidence algebra of a locally finite poset}
A poset $\mc{P}$ is locally finite if for any $a,b \in \mc{P}$, the interval $[a,b]$ is a finite poset.
In this case, one defines the incidence algebra $I(\mc{P})$ as the space of functions (with values in a commutative ring) on $1$-simplices in $N(\mc{P})$ (i.e.\ length $2$ chains $a \leq b$) with the convolution product
\[ f \star g (a \leq b) = \sum_{a \leq x \leq b} f(a \leq x)g(x \leq b). \]
The multiplicative identity is the function $f(a\leq b)=1$ if $a=b$, $0$ if $a \neq b$. The zeta function $\zeta_{\mc{P}}\in I(P)$ is the constant function $\zeta_{\mc{P}}(a \leq b) = 1$, and the Möbius function $\mu_{\mc{P}} \in I(\mc{P})$ is the inverse of $\zeta_P$ for the convolution product --- this is one formulation of Möbius inversion. We recall the topological expression for $\mu_{\mc{P}}$:

\begin{lemma}\label{lemma.poset-mobius} For $\tilde{\chi}=\chi - 1$ the reduced Euler characteristic,
\[ \mu_{\mc{P}}(a\leq b) = \begin{cases} \tilde{\chi}(N(a,b)) & \textrm{ if $a<b$} \\ 1 & \textrm{ if $a=b$} \end{cases}\]
where $(a,b)$ denotes the open interval of $p \in \mc{P}$ such that $a<p<b$.
\end{lemma}
\begin{proof}
Taking $\mu$ to be the element defined by the equation in the statement, we have to show that $\mu \star \zeta_{\mc{P}} (a \leq b)$ is $1$ if $a =b$ and $0$ otherwise. When $a=b$ this is trivial. For $a<b$, this 0 will be interpreted as 1 minus the Euler characteristic of the contractible poset $(a,b]$. Indeed, we have
\[ \mu \star \zeta_{\mc{P}} (a\leq b) = 1+ \sum_{a < x \leq b} \tilde{\chi}(N (a,x) ) \]
and we can group simplices contributing to the sum as follows: a non-degenerate simplex of $N(a,x)$ can be identified with a non-degenerate simplex of one degree higher in $N(a,b]$ by adding $x$ to the end. This establishes a bijection between the non-degenerate $k$-simplices of $N(a,x)$ and the non-degenerate $k+1$ simplices ending at $x$ of $N(a,b]$. Thus we miss only the $0$-simplices in $N(a,b]$, however, the reduced Euler characteristics appearing in the sum, which subtract off $1$ for each $a<x<b$, accounts for these, so that we have established
\[ \sum_{a < x \leq b} \tilde{\chi}( N(a,x) ) = - \chi(N(a,b])=-1, \]
(where the minus sign comes in because our bijection went from $k$-simplices to $k+1$-simplices). Adding $1$ gives zero, as desired.
\end{proof}

\subsubsection{Categorification}
\newcommand{\motIA}{I_{\mathrm{mot}}}
\newcommand{\shIA}{I_{\mathrm{sh}}}
Let $X$ be a Noetherian scheme and let $L$ be an algebraic extension of $\mbb{Q}_\ell$ with $\ell$ invertible in $X$. Let $\mc{P}/X$ be of finite type over $X$, so that $\sle_{\mc{P}}$ is also a Noetherian scheme. We consider the sheaf theoretic and motivic incidence algebras $\shIA(\mc{P}) \textrm{ and } \motIA(\mc{P})$: as abelian groups, we have
\[ \shIA(\mc{P}) = K_0 ( \Cons( \sle_{\mc{P}}, L)) \textrm{ and } \motIA(\mc{P}) = K_0(\Var/\sle_{\mc{P}}). \]
The multiplication, however, is given by the convolution product: on $\shIA(\mc{P})$, this is induced by the convolution functor
\begin{align*} D_{\Cons}( \sle_{\mc{P}}, L) \times D_{\Cons} ( \sle_{\mc{P}}, L ) & \xrightarrow{\star}D_{\Cons} ( \sle_{\mc{P}}, L ) \\
 (K_1, K_2) & \mapsto K_1 \star K_2:= R\delta^1_! \left( {\delta^2}^* K_1 \otimes_{L} {\delta^0}^* K_2 \right). \end{align*}
where $\delta^i$ are the standard face maps $N(\mc{P})_2 \rightarrow N(\mc{P})_1 = \mathord{\leq}_{\mc{P}}$. On $\motIA(\mc{P})$ it is induced by
\begin{align*} \Var/ \sle_{\mc{P}} \times \Var/\sle_{\mc{P}} &\rightarrow \Var/\sle_{\mc{P}} \\
(Y_1/\sle_{\mc{P}}) \times (Y_2/\sle_{\mc{P}}) & \mapsto Y_1 \star Y_2 = {\delta^2}^* Y_1 \times_{N(\mc{P})_2} {\delta^0}^* Y_2 \\ & \textrm{with structure map to $\sle_{\mc{P}}$ induced by $N(\mc{P})_2 \xrightarrow{\delta^1}  \sle_{\mc{P}}$.} \end{align*}

By standard arguments  $\shIA(\mc{P})$ (resp.\ $\motIA(\mc{P})$) is a commutative algebras over $K_0(\Cons(X,L))$ (resp.\ $K_0(\Var/X)$) and the compactly supported cohomology map is a map of $K_0(\Var/X)$-algebras $\motIA(\mc{P}) \rightarrow \shIA(\mc{P})$ (where the latter is a $K_0(\Var/X)$-algebra through the map $K_0(\Var/X) \rightarrow K_0(\Cons(X,L))$).

\begin{remark}
When describing the incidence algebra for posets above we allowed locally finite posets instead of finite posets --- this is quite useful in practice (e.g.\ for realizing Hasse--Weil zeta functions). In the geometric setting, we can generalize similarly: we say a poscheme $\mc{P}/X$ is locally finite if it is locally of finite type as a map of schemes and, for any finite type open subscheme $U/X \subset \mc{P}/X$,
\[ [U, U] = U \times_X \mc{P} \times_X U \cap N(\mc{P})_2 \]
is of finite type over $X$. In this case, for each such $U$ one can define the convolution products over $U \times U \cap \sle_{\mc{P}}$, then take the inverse limit of Grothendieck rings over all such $U$. We treat just the simpler finite type case here, as it illustrates the main points without the technical burden of making the previous sentence precise. In \cref{ss.symmetric-reduced-incidence}, however, we will study a version of this more general construction for the reduced incidence algebra of the poscheme of effective zero cycles (which is only locally finite), using the monoid structure to make it completely explicit.
\end{remark}

The identity for the convolution product in the incidence algebra $I_\mot(\mc{P})$ is $[\Delta_{\mc{P}}/\sle_{\mc{P}}]$. We define $\zeta_{\mc{P}}= [\sle_{\mc{P}}/\sle_{\mc{P}}]$. We then have the following Euler characteristic formula for Möbius inversion formula generalizing \cref{lemma.poset-mobius}:

\begin{theorem}\label{theorem.general-incidence-mobius}
Let $(\mc{P}, \mc{P})$ be the open interval poscheme, viewed as a poscheme over $\slt_{\mc{P}}\subset \mc{P}\times_X \mc{P}$; i.e.\ $N(\mc{P})_2^\circ$ viewed as a scheme over $\slt_{\mc{P}}=N(\mc{P})_1^\circ$ via $\delta^1$ ($a < b < c \mapsto a < c$) and ordered by pullback from the projection to middle of the chain $(a < b < c \mapsto b)$. If $\mu_{\mc{P}} :=1 + \tilde{\chi}(N(\mc{P},\mc{P})) \in \motIA(\mc{P})$ then
\begin{equation}\label{eq.mot-mob-inv-eq} \mu_\mc{P} \star \zeta_{\mc{P}}= 1 \in \motIA(\mc{P}),
\end{equation}
where we recall that in the definition of $\mu_{\mc{P}}$ and \cref{eq.mot-mob-inv-eq}, $1 = [\Delta_{\mc{P}}/\sle_{\mc{P}}] \in \motIA(\mc{P}).$
\end{theorem}
\begin{proof}
The proof is the same as the proof for posets --- we simply need to upgrade our bijection between simplices to a radicial surjective map (in fact an isomorphism). To that end, suppose
\[ t_0 < t_1 < \dotsm < t_k \in \delta_2^*N(\mc{P}, \mc{P})^\circ_k(T) \]
with image $a<b<c$ in $N(\mc{P})_2^\circ(T)$ --- note that this means
\[ a < t_0 < t_1 < \dotsm < t_k < b. \]
Then we can map this to the point
\[ t_0 < t_1 < \dotsm < t_k < b \in N(\mc{P},\mc{P}]_{k+1}^\circ(T) \]
above the point $a < c \in \sle_{\mc{P}}(T)$. By Yoneda, this gives an isomorphism of schemes over $\sle_{\mc{P}}(T)$ between $\delta_2^*N(\mc{P},\mc{P})^\circ_k$, viewed as a scheme over $\sle_{\mc{P}}$ by $\delta^1$, with $N(\mc{P},\mc{P}]^\circ_{k+1}$. We can hit the zero simplices except for the endpoint $a < c$ also by taking the reduced Euler characteristic; we thus obtain
\[ \tilde{\chi}(N(\mc{P},\mc{P})) \star \zeta_{\mc{P}}= -\chi(N(\mc{P},\mc{P}]) = -[\slt_{\mc{P}}/ \sle_{\mc{P}}], \]
where the final equality is by \cref{maintheorem.mot-inc-exc} because  $(\mc{P},\mc{P}]/ \slt_{\mc{P}}$ has a maximum. We conclude that
\[ \mu_{\mc{P}}\star \zeta_{\mc{P}}=(1 + \tilde{\chi}(N(\mc{P},\mc{P}))) \star \zeta_{\mc{P}} = [\sle_{\mc{P}}/\sle_{\mc{P}}] - [\slt_{\mc{P}}/\sle_{\mc{P}}]=[\Delta_{\mc{P}}/ \sle_{\mc{P}}]=1. \qedhere\]
\end{proof}

\begin{remark}
Specializing via $\motIA(\mc{P})\rightarrow \shIA(\mc{P})$, we obtain a Möbius inversion formula also in $\shIA(\mc{P})$.
\end{remark}

\subsection{The  reduced incidence algebra for the poscheme of effective zero-cycles}\label{ss.symmetric-reduced-incidence}

In \cref{ss.general-incidence-algebras}, we considered incidence algebras for pospaces and poschemes. There is another closely related notion in classical poset theory: for any locally finite poset $\mc{P}$ such that there is a good notion of two intervals $[a,b]$ and $[a',b']$ being ``the same", one considers the reduced incidence algebra consisting of functions on intervals such that $f( a\leq b)= f(a'\leq b')$ whenever $[a,b]$ is the same as $[a',b']$ --- it is a subalgebra of the full incidence algebra. This applies to the poset attached to a free commutative monoid such as the divisor poset in $\mbb{Z}_{>0}$ (which is the free commutative monoid generated by primes under multiplication, and where $[1, m]$ is ``the same'' as $[n, mn]$), whose reduced incidence algebra is a natural combinatorial home for the Riemann zeta function and other formal  Dirichlet series.

More generally, the poset of effective zero-cycles on a finite type scheme over $\mbb{Z}$ gives a reduced incidence algebra containing the zeta function of the variety as a natural element. We now lift this to the Grothendieck ring by using the poscheme of effective zero cycles; this gives a new interpretation of Bilu's \cite{bilu:motivic-euler-products} lift of the Kapranov zeta function used in the definition of motivic Euler products, and, through the Möbius inversion formula, Vakil and Wood's \cite{vakil-wood:discriminants} formula for its inverse.

For $Z/X$, we consider reduced motivic (resp.\ sheaf theoretic) incidence algebras defined via  the divided powers schemes of \cref{s.conf-zc-hilb}:
\begin{align*} \tilde{I}_{\mr{mot}}(\Gamma^{\bullet,+}_X(Z)) &:=  \prod_{k=0}^\infty K_0(\Var/\Gamma^{k}_X(Z))  \\
\tilde{I}_{\mr{sh}}(\Gamma^{\bullet,+}_X(Z)) &:= \prod_{k=0}^\infty K_0(\Cons(\Gamma^k_X(Z),L)) \end{align*}
equipped with the convolution products $x\star y = m_!(\pi_1^* (x) \pi_2^*(y))$ (resp.\ $Rm_!$\ldots)
where $\pi_1, \pi_2,$ and $m$ are the first projection, second projection, and multiplication maps, respectively
\[ \Gamma^{\bullet,+}\times\Gamma^{\bullet,+} \rightarrow \Gamma^{\bullet,+}, \]
and we recall that for $f:U \rightarrow V$, $f_!$ denotes the map on the relative Grothendieck ring of varieties which sends $[g:T\rightarrow U]$ to $[f \circ g: T \rightarrow V].$

We note that $\pi: \Gamma^k_X(Z) \rightarrow X$ induces natural algebra homomorphisms
\begin{align*} \pi_!: \tilde{I}_{\mr{mot}}(\Gamma^{\bullet,+}_X(Z)) & \rightarrow \tilde{K}_0(\Var/X)[[t]] \textrm{ and }\\
 R\pi_!: \tilde{I}_{\mr{sh}}(\Gamma^{\bullet,+}_X(Z)) &\rightarrow K_0(\Cons(X,L))[[t]]. \end{align*}
More generally, a map $Z \rightarrow Z'$ over $X$ induces maps of the incidence algebras, and the above are the maps obtained from the structure map $Z/X \rightarrow X/X$ and the identification $\Gamma^{\bullet,+}_X (X)=\mbb{Z}_{\geq 0} \times X / X$.

The (relative to $X$) Kapranov zeta function $Z_{Z/X}^\mr{Kap}(t) \in K_0(\Var/X)[[t]]$ naturally lifts along $\pi_!$ to the motivic incidence algebra as
\[ \zeta = (1,1,1,\ldots), \textrm{ where we note } 1 = [\Gamma^k_X(Z) / \Gamma^k_X(Z)] \in K_0(\Var/\Gamma^k_X(Z)). \]
Here we recall from \cref{s.conf-zc-hilb} that there is a natural universal homeomorphism $\Sym^k_X(Z) \rightarrow \Gamma^k_X(Z)$ so that, in particular, they have the same class in the modified Grothendieck ring $K_0(\Var/X)$ and one can define the Kapranov zeta using either $\Sym^\bullet$ or $\Gamma^\bullet$.

\begin{remark}\label{remark.bilu-incidence}
The convolution structure on $\tilde{I}_{\mr{mot}}(\Gamma^{\bullet,+}_X(Z))$ along with this lift of the Kapranov zeta function was essentially considered by Bilu \cite{bilu:thesis, bilu:motivic-euler-products} (with $\Sym^\bullet$ in place of $\Gamma^\bullet$) in her definition of motivic Euler products, but the interpretation as a reduced incidence algebra appears to have gone unnoticed.
\end{remark}

\subsubsection{Möbius inversion formula for the Kapranov zeta function}

\begin{theorem}[Möbius inversion and Kapranov zeta]\label{theorem.mobius-inversion-kapranov}
In $\tilde{I}_{\mr{mot}}(\Gamma^{\bullet,+}_X(Z))$, writing the Möbius elements as $\mu=(\mu_0,\mu_1, \ldots)$, for $k \geq 1$,
\begin{align*} \mu_k &= \sum_{p=0}^\infty (-1)^{p-1} \sum_{\ul{a} \in \mbb{Z}_{>0}^{[p]}, \sum \ul{a}=k} \left[\Gamma^{\ul{a}}_X(Z)  / \Gamma^k_X(Z)\right] \\
&=  \sum_{p=0}^\infty (-1)^{p-1} \sum_{\ul{a} \in \mbb{Z}_{>0}^{[p]}, \sum \ul{a}=k} \left[\Conf^{\ul{a}}_X(Z)  / \Gamma^k_X(Z)\right].\end{align*}
In $\tilde{I}_{\mr{sh}}(\Gamma^{\bullet,+}_X(Z))$, writing the Möbius elements as $\mu=(\mu_0,\mu_1, \ldots)$, for $k \geq 1$,
\begin{align*} \mu_k &= (-1)^k\left[ (s_*E)[\sgn] \right] \textrm{ for } s: Z\times_X \dotsm \times_X Z \rightarrow \Gamma^k_X(Z) \\
&= (-1)^k \left[ (s^\circ_!E)[\sgn] \right] \textrm{ for } s^\circ:=s|_{\Conf^{(1, \ldots, 1)}_X(Z)} \\
& = (-1)^k\left[j_!\ul{\sgn}\right] \textrm{ for } j: \Conf^k_X(Z) \hookrightarrow \Gamma^k_X(Z).
\end{align*}
where $\sgn$ denotes the $\sgn$ character of $\mfS_k$ on $L$, $\ul{\sgn}$ denotes the corresponding local system on $\Conf^k_X(Z)$, and $V[\sgn]$ denotes the isotypic part for a $\mfS_k$-action on $V$.
\end{theorem}

\begin{remark}\label{remark.vakil-wood-inversion-formula}
Passing the second equality in $\tilde{I}_{\mr{mot}}$ to $K_0(\Var/\kappa)[[t]]$ via $\pi_!$  recovers Vakil--Wood's formula
\[ Z_{Z/X,\Kap}(t)^{-1}=1 + \sum_{p=0}^\infty (-1)^{p+1} \sum_{\ul{a} \in \mbb{Z}_{>0}^{[p]}} \left[\Conf^{\ul{a}}_X(Z)  / X\right] t^{\sum{\ul{a}}}. \]
Passing the first equality in $\tilde{I}_{\mr{sh}}$ to a cohomological Grothendieck ring recovers  \cref{eq.zeta-cohom-value}.
\end{remark}

\begin{proof}
The first two equalities in $\tilde{I}_{\mr{sh}}$ can be deduced immediately from the two equalities in $\tilde{I}_{\mr{mot}}$ by using the formula for the character of the sign representation in terms of permutation representation given by the reduced Euler characteristic of the simplicial cohomology complex computing $H^\bullet(\partial \Delta^{k-1})$ for the barycentric subdivision of $\partial \Delta^{k-1}$ (see \cref{ss.conf-sym-posets}). The third equality in $\tilde{I}_{\mr{sh}}$ is a reformulation of the second. See the paragraph following this proof for an alternative deduction of these sheaf identities without passing through the motivic identities.

It remains to treat the motivic case: by essentially the same argument as \cref{theorem.general-incidence-mobius},
\begin{equation}\label{eq.mot-mobius-euler-char-formula} \mu= 1 + \tilde{\chi}( N(-\infty, \Gamma^{\bullet}_X(Z))). \end{equation}

Thus, to compute a formula for $\mu$ we need only to compute this Euler characteristic. From the definitions and the monoidal description of the nerve in \cref{remark.monoidal-nerve-poscheme-nerve}, one then obtains
\[ \mu|_{\Gamma^k_X(Z)} =  \sum_{p=0}^\infty (-1)^{p-1} \sum_{\ul{a} \in \mbb{Z}_{>0}^{[p]}, \sum \ul{a} =k} \left [\Gamma^{\ul{a}}_X(Z) / \Gamma^k_X(Z)\right] \]
where the $(-1)^{p-1}$ is because a point $(t_0,...,t_{p}) \in \Gamma^{\ul{a}}_X(Z)(T)$ corresponds to the $p-1$ simplex $t_0 < t_0+t_1 < \dotsm < t_0+\dotsm + t_{p-1}$ in $(-\infty, t_0+t_1 + \dotsm + t_p)$.

On the other hand, by \cref{theorem.symmetric-hilb-relative-computation-euler},
\[ \tilde{\chi}( N(-\infty, \Gamma^{\bullet}_X(Z))) = \tilde{\chi}( N(-\infty, \Conf^{\bullet}_X(Z))) \]
and we also obtain the formula in terms of configuration spaces similarly.
\end{proof}

In the proof we gave a direct argument in the motivic case and deduced the sheaf-theoretic version by a character identity. Arguing instead directly in the sheaf-theoretic case and invoking \cref{prop.sym-reduced-cohomo-computation}, one obtains naturally the third expression, which is trivially equivalent to the second already at the level of the sheaves on $\Gamma^k_X(Z)$. Arguing with the skeletal spectral sequence as in \cref{ss.skeletal}, one would instead obtain the first expression. Alternatively, an argument we learned from O.~Banerjee using the spectral sequence of a stratified space shows directly the equality between these two expressions. As it will be useful in \cref{s.smooth-hypersurface}, we record this result here in the case of a constructible sheaf using a slightly different proof:

\begin{proposition}\label{prop.sgn-cohomology-support}
Suppose $Z/X$ quasi-projective, $S$ is a finite set, and $\mc{F}$ is a constructible sheaf of $L$-modules on $X$. Let $\pi: (Z/X)^{S}=\Gamma_X^S(Z) \rightarrow \Gamma^{|S|}_X(Z)$ denote the addition of cycles map, and let $j: \Conf^{|S|}_X(Z) \hookrightarrow \Gamma^{|S|}_X(S)$. Then
\[ \pi_*\mc{F}|_{\Gamma^{S}_X(Z)}[\sgn] =\pi_! \mc{F}|_{\Conf^{S}_{X}(Z)}[\sgn]=j_!\left( \mc{F}|_{\Conf^{|S|}_X(Z)} \otimes \ul{\sgn}\right) \]
where $[\sgn]$ denotes the isotypic component for the sign representation of $\mfS_S$, i.e.\ the image of the idempotent
\begin{equation}\label{eq.sgn-idempotent} \frac{1}{|S|!} \sum_{\sigma \in \mfS_S} \sgn(\sigma)\sigma. \end{equation}
In particular, if $X \rightarrow Y$, then
\begin{align*} \left(R^\bullet(\Gamma_X^{S}(Z)\rightarrow Y)_! \mc{F}|_{\Gamma_X^{S}(Z)} \right)[\sgn] &= \left(R^\bullet(\Conf_X^{S}(Z)\rightarrow Y)_! \mc{F}|_{\Conf_X^{S}(Z)} \right)[\sgn]\\
&= R^\bullet(\Conf_X^{|S|}(Z) \rightarrow Y)_! (\mc{F}|_{\Conf_X^{|S|}(Z)}\otimes \ul{\sgn}).
\end{align*}
\end{proposition}
\begin{proof}
The statement about proper pushforwards are immediate from the sheaf-theoretic equalities since $\pi$ is a finite map so $\pi_!=R\pi_!$. For the sheaf-theoretic equalities, the second is almost tautological, so it remains to show the first. Since $\pi_* \mc{F}$ is constructible, so is $\pi_*\mc{F}[\sgn]$, a direct summand, and it suffices to show that its stalk vanishes at every geometric point in the closed set $\Gamma^{|S|}_X(Z) \backslash \Conf^{|S|}_X(Z)$. So, let $\overline{c}: \Spec \kappa \rightarrow  \Gamma^{|S|}_X(Z) \backslash \Conf^{|S|}_X(Z)$ for $\kappa$ algebraically closed, and write $\overline{x}$ for its image in $X$. Then $\overline{c}$ corresponds to an effective zero-cycle
\[ \sum_{z \in Z_{\overline{x}}(\kappa)} a_z Z \]
with $\sum a_Z = |S|$ and $a_z \geq 2$ for at least one $z$. Since $\pi$ is finite and $\mc{F}$ is pulled back from $X$,
\[ (\pi_* \mc{F})_{\overline{c}} = \bigoplus_{\pi^{-1}(\overline{c})} \mc{F}_{\overline{x}}. \]
The preimage $\pi^{-1}(\overline{c}) \subset \Gamma^{S}_X(Z)(\kappa)$ indexing the direct sum consists of the maps
\[ \tilde{c}: S \rightarrow Z_{\overline{x}}(\kappa) \]
such that $\sum_{s \in S} \tilde{c}(s) = \overline{c}$ (as effective zero-cycles) with the obvious action of $\mfS_{S}$.

Passing to the $\sgn$ component commutes with taking stalks, so it suffices to show that the idempotent \cref{eq.sgn-idempotent} acts by zero on this stalk. However, since $a_z \geq 2$ for some $z$, any $\tilde{c}$ as above is preserved by a transposition in $\mfS_S$. Breaking the sum in the definition of the idempotent \cref{eq.sgn-idempotent} into cosets of this transposition shows that anything in the image of the idempotent is zero in the $\tilde{c}$ component, and since this holds for each $\tilde{c}$, we conclude that the idempotent is identically zero.
\end{proof}

In particular, applying the Kunneth formula, one obtains

\begin{corollary}[O.~Banerjee]\label{corollary.sgn-cohomology-computation} For any variety $Y$ over $\kappa$ algebraically closed (and with $\Lambda_{\mr{gr}}$ denoting the graded exterior power):
\begin{align}\label{eq.sgn-cohomology-formulas} H^\bullet_c(\Conf^{p}(Y), \ul{\sgn}) & =H^\bullet_c(Y^p)[\sgn] \\
& = \Lambda_{\mr{gr}}^p H_c^\bullet(Y, \mbb{Q}_\ell) \\
& = \bigoplus_{i=0}^p \Lambda^i H_c^{\mr{even}}(Y, \mbb{Q}_\ell) \otimes \operatorname{Sym}^{p-i}H_c^{\mr{odd}}(Y, \mbb{Q}_\ell).
\end{align}
\end{corollary}

\subsection{Cohomological special values of Kapranov zeta}\label{ss.cohom-special-values}

Suppose now that $X/\kappa$ is a smooth projective algebraic variety over an algebraically closed field. Building on the constructions above, we define cohomological special values of the inverse Kapranov zeta function as the bigraded (by weight and degree) vector spaces
\begin{align}\label{eq.intro-cohomo-zeta} \zeta_{X}^{-1,\mr{Coh}}(n) &:=  \bigoplus_{k=0}^\infty (\zeta_{X}^{-1,\mr{Coh}})_k  \otimes \mbb{L}_\mr{Coh}^{\otimes -nk}, \; \mbb{L}_{\mr{Coh}}:= H^\bullet_c(\mbb{A}^1),
 \\
 \nonumber (\zeta_{X}^{-1,\mr{Coh}})_k &:=  H^{\bullet-k}_c(\Conf^k(X), \ul{\sgn}) = H^{\bullet-k}(X^k)[\sgn]
  \end{align}
where the last equality is by \cref{corollary.sgn-cohomology-computation}. Note that by our conventions, the $k$th summand in $\zeta_{X}^{-1,\mr{Coh}}(n)$ sits in weight $-k$.

This definition is motivated as follows: we have seen above that $Z_{X}^\Kap(t)^{-1}$ lifts naturally to the Möbius element $\mu$ in $\tilde{I}_\mr{Mot}$. On the other hand, \cref{eq.mot-mobius-euler-char-formula} gives a formula for the $k$th component, $k \geq 1$, as the Euler characteristic
\[ \mu_k = \tilde{\chi}(N(-\infty, \Gamma^k(X))). \]
Thus, to obtain a cohomological analog, we should replace the Euler characteristic with the corresponding reduced cohomology sheaf $j_! \ul{\sgn}[2-k]$ as computed in \cref{prop.sym-reduced-cohomo-computation}. The analog of the forgetful map from $K_0(\Var/\Gamma^k(X))$ to $K_0(\Var/\kappa)$ is compactly supported cohomology, and, after a shift by $2$, this yields the formula \cref{eq.intro-cohomo-zeta} for special values (using the obvious interpretation of $\mbb{L}$). The shift by two here is natural for various reasons (in particular, in applications it is cancelled out by the same shift by two that occurs in \cref{main.spectral-sequences}-(2)), so that we incorporate it into the definition.

In \cref{s.smooth-hypersurface}, a special role is played by the special value for $n=\dim X + 1$. By the above definition and Poincar\'{e} duality,
\begin{align}\label{eq.zeta-cohom-expression} \zeta_X^{-1,\Coh}(\dim X + 1) &= \bigoplus_{k=0}^\infty H_{\bullet-k}(X^k)[\sgn] \otimes \mbb{L}_\Coh^{\otimes{-k}}.
\end{align}
This is naturally identified with the graded symmetric algebra (for the Koszul sign rule with commutativity constraint given by degree) of $H_{\bullet-1}(X)(1)$, where the degree shift and Tate twist place $H_i(X)$ in degree $-i-1$ and weight $-1$; this will be explained again in the introduction to \cref{s.smooth-hypersurface} where it connects our stabilization results to those of Aumonier \cite{aumonier}; see also \cref{eq.sgn-cohomology-formulas}.

\section{Stability for the space of smooth hypersurface sections}\label{s.smooth-hypersurface}

In this section we prove \cref{intro.stabilization}. We first discuss the context of this result and related work, expanding on the discussion in the introduction and starting with the Grothendieck ring stabilization of Vakil and Wood \cite{vakil-wood:discriminants, vakil-wood-errata} (which we will reprove below in parallel with \cref{intro.stabilization} to illustrate the close relation between the methods). To ensure this discussion is accessible to readers who have skipped here directly from the introduction, we will recall some notation along the way.

Let $X$ be a smooth projective variety over an algebraically closed field $\kappa$. Let $\mc{L}$ be an ample line bundle on $X$ and let $V_d$ be the affine space of global sections of $\mc{L}^d$. Let $U_d \subset V_d$ be the open subscheme of sections with non-singular vanishing locus --- its complement $D_d$ is the image of the incidence variety of $I_d$ parameterizing $(x,f) \in X \times V_d$ such that the degree one Taylor expansion of $f$ at $x$ is zero.

We write $\mbb{L}=[\mbb{A}^1]$ in any Grothendieck ring of varieties. For $\kappa$ any field, Vakil and Wood showed that, in the completion of $K_0(\Var/\kappa)[\mbb{L}^{-1}]$ for the dimension filtration (where elements of the form $[Z]/\mbb{L}^n$ are small if $n\gg \dim Z$),
\begin{equation}\label{eq.intro-vakil-wood}\lim_{d \rightarrow \infty} \frac{[U_d]}{\mbb{L}^{\dim U_d}}=\zeta_{X}^\Kap(\dim X + 1)^{-1},\end{equation}
an inverse special value of the Kapranov zeta function
\[\zeta_X^\Kap(n) := Z_X^\Kap(\mbb{L}^{-n}),\, Z_X^\Kap(t) := \sum_{k=0}^\infty [\Sym^k X]\mbb{L}^k \in K_0(\Var/\kappa)[[t]], \; \Sym^k X := X^k/\mfS_k. \]

If $X$ is defined over a finite field $\bbF_q$, the $d$th term in the limit on the left of \cref{eq.intro-vakil-wood} specializes by point-counting to the probability that a random smooth degree $d$ hypersurface section of $X$ defined over $\bbF_q$ is smooth. The Kapranov zeta function specializes by point-counting to the Hasse--Weil zeta function, and the result of Vakil and Wood is a Grothendieck ring analog of a point-counting result of Poonen \cite{poonen:bertini} for varieties over finite fields that shows these probabilities converge to the same special value of the Hasse--Weil zeta function (recall however that the Grothendieck ring stabilization does not imply the point-counting stablization because point-counting is not continuous for the dimension topology; see \cite{bilu-das-howe} for a recent dicussion). This point-counting result is itself an extension from curves to arbitrary varieties of the function field analog of the classical statement that the asymptotic probability that an integer is squarefree is  $\zeta_{\mbb{Z}}(2)^{-1}$ for $\zeta_{\mbb{Z}}(s)=\sum \frac{1}{n^s}$ the Riemann zeta function (see \cite{poonen:bertini} or \cite[\S1.1]{bilu-howe:mot-eul-mot-stat} for more details on this point).

We now assume $\kappa$ is algebraically closed, and denote by $H^\bullet(-)$ (resp.\ $H_c^\bullet(-)$, resp.\ $H_\bullet(-)$) either $\ell$-adic étale cohomology for $\ell$ invertible in $\kappa$ or rational singular cohomology if $\kappa = \mbb{C}$ (resp.\ compactly supported cohomology, resp.\ homology). There is a natural compactly supported Euler characteristic map from $K_0(\Var/\kappa)[\mbb{L}^{-1}]$ to a weight-graded cohomological Grothendieck ring $K_0^\Coh$ (e.g.\ of Hodge structures or germs of Galois representations),
\[ [Y] \mapsto \sum_i (-1)^i [H^i_c(Y)], \]
 and by Poincar\'{e} duality the class of $[U_d]/\mbb{L}^{\dim V_d}$ is sent to
\[ \sum_i (-1)^i [H_i(U_d)]. \]
The result of Vakil and Wood implies that this generalized homological Euler characteristic stabilizes as $d \rightarrow \infty$ in the completion of $K_0^\Coh$ for the weight grading to the image in the same ring of $\zeta_X^\Kap(\dim X +1)^{-1}$. Based on this observation, Vakil and Wood conjectured that the rational homology of $U_d$ also stabilizes, but without specifying a natural stable value except in  cases where the special value is of a particularly simple form. In those cases, they conjectured an Occam's razor principle that the cohomology should be in a sense the simplest possible.

Tommasi \cite{tommasi:stability}  established homological stability in the case $X=\mbb{P}^n$ by combining a Vassiliev-type spectral sequence with an $E_1$-degeneration argument specific to the case of $\mbb{P}^n$. Interestingly, Tommasi's computation showed that the most naive Occam's razor does not hold in this case, but for good reasons -- in this case, the orbit map for the natural action of $\mr{PGL}_{n+1}(\mbb{C})$ describes the cohomology completely, so that the stable cohomology is equal to the cohomology of $\mr{PGL}_{n+1}(\mbb{C})$. The Vassiliev spectral sequence of loc.~cit.\ applies to general $X/\mbb{C}$, and indeed our approximate inclusion-exclusion formula is an algebro-geometric version of this sequence. The degeneration argument and the simple description of the stable cohomology as that of $\mr{PGL}_{n+1}(\mbb{C})$ are very specific to the case $X=\mbb{P}^n$.

Recently Aumonier \cite{aumonier} has obtained stabilization for general smooth projective $X/\mbb{C}$ via an $h$-principle comparing continuous and holomorphic sections of a jet bundle. The end result includes a beautiful and simple description of the stable cohomology of $U_d$ in terms of the cohomology of $X$ itself. In hindsight, there is a simple heuristic that leads directly from Vakil and Wood's stabilization to Aumonier's description and can be viewed as a refined Occam's razor: The Kapranov zeta function defines a pre-$\lambda$ ring structure on the Grothendieck ring and, using the notation of the corresponding power structure \cite{glm:power-structure} (see also, e.g., \cite{howe:mrv1, bilu-howe:mot-eul-mot-stat}), we can write $Z_{X,\Kap}(t)^{-1}=(1-t)^{[X]}$. Specializing to the cohomological Grothendieck ring, we may then expand as
\begin{align}\nonumber (1-t)^{\sum_i (-1)^i [H^{i}(X)]}&=(1- t)^{[H^{\mr{even}}(X)]}(1-t)^{-[H^{\mr{odd}}(X)]} \\
 \nonumber &= \left( \sum_j (-1)^j  \left[\bigwedge^j\nolimits H^{\mr{even}}(X)\right] t^j \right) \left( \sum_j \left[ \operatorname{Sym}^j H^\mr{odd}(X) \right] t^j \right) \\
\label{eq.zeta-cohom-value} &= \sum_{j} (-1)^j \left[\bigwedge_{\mr{gr}}^j  H^\bullet(X)\right]t^j
\end{align}
Here the subscript $\mr{gr}$ on the last line denotes exterior power is of graded vector spaces with the Koszul sign rule. If we substitute $t=\mbb{L}^{-(\dim X +1)}$, then this is the natural class attached to the graded symmetric algebra of $H_{\bullet-1}(X)(1)$, and in \cite{aumonier} it is shown that the cohomology ring of $U_d$ stabilizes to the dual graded symmetric algebra of $H^{\bullet-1}(X)(-1)$ (our result here does not describe the cup product).

On the other hand, the $k$th graded exterior power of a graded vector space $V$ is isomorphic to the $\sgn$-isotypic summand for the $\mfS_k$-action on the graded tensor product $V^{\otimes k}$ (with the Koszul sign rule), so by Kunneth \cref{eq.zeta-cohom-value} equals
\begin{equation}\label{eq.stab-intro-sgn-zeta-inverse}  \sum_{k=0}^\infty (-1)^{k} \left[ H^\bullet(X^k)[\sgn] \right] t^k. \end{equation}
This incarnation is how the special value appears in our homological stabilization.

To explain this, recall that there are two key ingredients in Vakil and Wood's proof of the motivic stabilization \cref{eq.intro-vakil-wood} (see also \cite{bilu-howe:mot-eul-mot-stat}): an approximate motivic inclusion-exclusion formula describing $D_d$ using the resolution $I_d \rightarrow D_d$, and an inversion formula for $Z_X^\Kap(t)$. In our language, the approximate inclusion-exclusion is provided by \cref{theorem.approx-ie}-(1), which is of a particularly simple form because for $d$ sufficiently large the relative configurations of $I_d$ are vector bundles over configurations of $X$. The inverse formula was treated in \cref{s.incidence-algebras}, where we explained how \cref{eq.stab-intro-sgn-zeta-inverse} and other closely related formulas can be obtained from M\"obius inversion on the poscheme of effective zero cycles on $X$. Since M\"obius inversion is an incarnation of inclusion-exclusion, it is no surprise that the two should be related.

For homological stabilization, the approximate inclusion-exclusion formula is provided by \cref{theorem.approx-ie}-(2). For the inversion formula, in section \cref{ss.cohom-special-values}, motivated by incidence algebra constructions, we defined bigraded (by degree and weight) cohomological special values of the inverse Kapranov zeta function
 \begin{align*} \zeta_{X}^{-1,\mr{Coh}}(n) &=  \bigoplus_{k=0}^\infty H^{\bullet-k}(X^k)[\sgn] \otimes \mbb{L}_\mr{Coh}^{\otimes -nk}, \; \mbb{L}_{\mr{Coh}}:= H^\bullet_c(\mbb{A}^1).
 \end{align*}
 In particular, the weight $-k$ summand of $\zeta_{X}^{-1,\mr{Coh}}(-\dim X + 1)$ is
 \[ H^{\bullet-k}(X^k)[\sgn] \otimes \mbb{L}_\mr{Coh}^{\otimes -(\dim X +1)k} = H_{\bullet+k}(X^k)[\sgn] \otimes \mbb{L}_\mr{Coh}^{\otimes -k} \]
 These terms will be matched with the $E_1$ page of the skeletal spectral sequence for the truncated poscheme of effective zero cycles that arises in approximate inclusion-exclusion, so that the main point is to prove this spectral sequence degenerates at $E_1$. We show this with a direct analysis on the $E_1$ page and a weight argument to treat the later differentials in order to prove a slight refinement of \cref{intro.stabilization}:

 \begin{theorem}\label{theorem.body-stabilization}
 For $d\gg_k 0$, the skeletal spectral sequence for the $k$-truncated poscheme of effective zero cycles for (a compactification of) $I_d/D_d$ induces a canonical isomorphism of bigraded vector spaces
\[ \Gr_W H_{\leq k}(U_d) \cong \zeta_X^{-1,\Coh}(\dim X+1)_{\deg \geq -k} \]
where we recall  $W$ denotes the weight filtration and $H_i(U_d)$ sits in degree $-i$.
\end{theorem}

The rest of this section is organized as follows: in \cref{ss.mot-stab} we recall the argument for Vakil--Wood's motivic stabilization via motivic approximate inclusion-exclusion as in \cref{theorem.approx-ie}-(1) --- this is the argument of \cite{vakil-wood:discriminants, vakil-wood-errata} (see also \cite{bilu-howe:mot-eul-mot-stat}), but the results of the previous sections render completely transparent the relation between the approximate motivic inclusion-exclusion formula and the inversion formula for the Kapranov zeta function. In \cref{ss.hom-stab-1} we carry out the first steps of the proof of \cref{intro.stabilization} by cohomological approximate inclusion-exclusion as in \cref{theorem.approx-ie}-(2). These first steps mirrors the motivic argument, but the final, critical step, which has no analog in the Grothendieck ring, is to show the degeneration at $E_1$ of the skeletal spectral sequence for the truncated symmetric power poscheme. We carry this out in \cref{ss.stab-degeneration} --- on $E_1$ we can analyze the differentials explicitly, while on higher pages we obtain vanishing by a weight argument using purity of $E_1$.

\subsection{Motivic stabilization}\label{ss.mot-stab}

The main geometric input is:

\begin{lemma}[{\cite[Lemma~3.2]{vakil-wood:discriminants}}]\label{lemma.conf-vector-bundle}
For any $M>0$, there is a $D$ sufficiently large such that for all $ d \geq D$ and $a_1 + a_2 + \dotsm + a_m =k \leq M$, $\Conf^{(a_1, \ldots, a_m)}_{V_d}(I)$ is a vector bundle of rank $r(k):=\dim V_d - k(\dim X+1)$ over $\Conf^{(a_1, \ldots, a_m)}(X)$.
\end{lemma}

As a first consequence, if we fix a $k$, then for $d\gg 0$ the hypothesis of \cref{theorem.approx-ie} is met for $k$. Indeed, the lemma implies that for any $j$ and $d\gg 0$, $\dim \Conf^j_{V_d}(I_d)= \dim V_d - j$. This implies that $V_{d,\geq j}$ has codimension at least $j$ in $V_d$. If this applies also for $j+1$ then we see that $V_{d,j}$ has codimension exactly $j$ --- indeed, then $\Conf^j_{V_{d,j}}(I_d)=\Conf^j_{V_{d,\geq j}}(I_d) \backslash \Conf^{j}_{V_{d, \geq j+1}}(I_d)$ has dimension $V_d - j$, thus so does $V_{d,j}$.  In particular, the codimension of $V_{d, >k}$ is $k+1$, while the codimension of $V_{d,\infty}=\bigcap_{j} V_{d, \geq j}$ can be made arbitrarily large by taking $d$ sufficiently large, so that the desired inequality
 \[ \dim V_{d,>k} > \dim V_{d,\infty} + k \dim_{/V_d} I_d = \dim V_{d,\infty} + k \dim X \]
 always holds for $d$ sufficiently large.

Fix a $k$ and assume $d$ is larger than the bound $D$ of \cref{lemma.conf-vector-bundle} for $M=k+1$ and also large enough for the hypothesis of \cref{theorem.approx-ie}  to hold. The latter gives
\begin{align*} [U_d] = [V_d]-[D_d] & = [V_d] + \sum_{\substack{ (a_1,a_2,\ldots, a_m) \\ a_i>0, \sum a_i \leq k} } (-1)^{m} [\Conf^{(a_1,a_2,\ldots, a_m)}_X(Z)] \\ & \mod \Fil^{k+1 - \dim V_d} K_0(\Var/\kappa)[\mbb{L}^{-1}], \end{align*}
and \cref{lemma.conf-vector-bundle} allows us to rewrite the terms as
\[ [\Conf^{(a_1, \ldots, a_m)}_{V_d}(I_d)]=[\Conf^{(a_1,\ldots,a_m)}(X)] \mbb{L}^{\dim V_d - k(\dim X+1)}. \]
Dividing everything by $\mbb{L}^{\dim V_d}$, we obtain
\begin{align*} \frac{[U_d]}{\mbb{L}^{\dim V_d}} & \equiv 1 + \sum_{\substack{ (a_1,a_2,\ldots, a_m) \\ a_i>0, \sum a_i \leq k} } (-1)^{m} [\Conf^{(a_1,a_2,\ldots, a_m)}_X(Z)]\mbb{L}^{- k(\dim X+1)} \\
& \equiv Z_X^{\Kap}(\mbb{L}^{-(\dim X +1)}) \mod \Fil^{k+1} K_0(\Var/\kappa)[\mbb{L}^{-1}]
\end{align*}
where the final line follows from the inversion formula \cref{remark.vakil-wood-inversion-formula}.
Taking $k$ larger and larger gives the result
\[ \lim_{d \rightarrow \infty}  \frac{[U_d]}{\mbb{L}^{\dim V_d}} = Z_X^{\Kap}(\mbb{L}^{-(\dim X +1)}) =: \zeta_X^{\Kap}(\dim X +1). \]

\subsection{Homological stabilization --- first steps}\label{ss.hom-stab-1}

To prove \cref{theorem.body-stabilization}, we must compute the weight graded of $H_i(U_d, \mbb{Q}_\ell)$. By Poincaré duality, this is equivalent to computing the weight graded of $H^{2\dim V_d - i}_c(U_d, \mbb{Q}_\ell)$.  By the long exact sequence of compactly supported cohomology for the decomposition of $V_d$ into the open $U_d$ and its closed complement $D_d$, we have
\begin{equation}\label{eq.Ud-from-Dd} H^i_c(U_d)=\begin{cases} H^{i-1}_c(D_d) & \textrm{for $1\leq i<2\dim V_d$} \\
H^{2\dim V_d}_c(V_d)=(\mbb{L}^\Coh)^{\dim V_d}& \textrm{for $i=2\dim V_d$} \\
0 & \textrm{otherwise}
\end{cases} \end{equation}
Thus it will suffice to compute the weight graded of $H^i_c(D_d)$.

It will be useful later on in our argument to compute this using an explicit compactification, so we introduce it into our setup from the beginning: let $\overline{V}_d= \mbb{P}(V_d(\kappa) \oplus \kappa)$ be the compactification of $V_d$ to a projective space over $\Spec \kappa$. The incidence variety $I_d$ evidently extends to
\[ \overline{I}_d \subseteq X \times \overline{V}_d  \]
(indeed, the points we have added correspond to $\mbb{P}(V_d)$, and evidently the property of being singular at a point depends only on the equation up to multiplication by a scalar), and $\overline{I}_d$ is the closure of $I_d$.
Let $\overline{D}_d$ denote the image of $\overline{I}_d$ in $\overline{V}_d$, which is also the closure of $D_d$ in $\overline{V}_d$ . Writing $j: D_d \hookrightarrow \overline{D}_d$ for the open immersion, we have $H^i_c(D_d, \mbb{Q}_\ell)=H^i(\overline{D}_d, j_! \mbb{Q}_\ell)$.

We will compute these cohomology groups using the $k$-truncated poscheme of effective zero cycles $\Gamma^{\leq k}_{\overline{V}_d}(\overline{I}_d)$. Thus, we need to invoke  \cref{theorem.approx-ie}-(2) to show that for a fixed $k$ this is a good approximation if $d \gg_k 0$. If we fix a $k$, then, arguing as in the motivic case above, we may assume $d \gg_k 0$ is large enough so that
\begin{enumerate}
\item $\Conf^p_{D_d}(I_d)$ is a vector bundle of rank $r(p)=\dim V_d - p(\dim X +1)$ over $\Conf^p(X)$ for all $1 \leq p \leq k$, and $\Conf^p_{\overline{D}_d}(\overline{I}_d)$ is the compactifying projective bundle obtained by taking the closure of $\Conf^{p}_{D_d}(I_d)$ inside of $\Conf^{p}(X) \times \overline{V}_d$.
\item $\dim \overline{V}_{d,>k} \leq \dim {\overline{V}_d} - (k+1)$
\item The dimension hypothesis of \cref{theorem.approx-ie} holds.
\end{enumerate}

We apply \cref{theorem.approx-ie} to $j_! \mbb{Q}_\ell$. Taking the skeletal spectral sequence and applying \cref{theorem.banerjee-qiso} to simplify the $E_1$ page, we thus obtain
\[E^{p,q}_1 = \begin{cases} H^q( (\overline{I}_d / \overline{D}_d)^{[p]}, j_!\mbb{Q}_\ell)[\sgn] & \text{if } 0\leq p \leq k-1, \\0 & \text{otherwise,}\end{cases}\]
and
\begin{align*} E_1^{p,q} & \Rightarrow H^{p+q}(\Gamma^{\leq k}_{\overline{D}_d}(\overline{I}_d), \epsilon^*j_!\mbb{Q}_\ell)\\
&=H^{p+q}(\overline{D}_d, j_!\mbb{Q}_\ell) & &\textrm{ if } p+q \geq 2\dim V_d - k+1 \\
&=H^{p+q}_c(D_d, \mbb{Q}_\ell). \end{align*}
By \cref{eq.Ud-from-Dd}, this is sufficient to compute $H^i_c(U_d, \mbb{Q}_\ell)$ for $i \geq 2 \dim V_d - k.$

We observe that we can rewrite, for $0 \leq p \leq k-1$,
\begin{align*} E^{p,\bullet}_1 &= H^\bullet( (\overline{I}_d/\overline{D}_d)^{[p]}, j_!\mbb{Q}_\ell)[\sgn]\\
&=  H^\bullet_c( (I_d / D_d)^{[p]}, \mbb{Q}_\ell)[\sgn] \\
&= H^\bullet_c( \Conf^{[p]}_{D_d}(I_d), \mbb{Q}_\ell)[\sgn] \\
&= H^\bullet_c( \Conf^{[p]}(X), \mbb{Q}_\ell)[\sgn] \otimes \mbb{L}_\Coh^{\dim V_d - (\dim X + 1)(p+1)}\\
&=H^\bullet(X^{[p]}, \mbb{Q}_\ell)[\sgn] \otimes \mbb{L}_\Coh^{\dim V_d - (\dim X +1)(p+1)}.
\end{align*}

The third and fifth equalities follow from \cref{prop.sgn-cohomology-support}, while the fourth follows from the projection formula. Renormalizing,
\[ E^{p,\bullet}_1 \otimes \mbb{L}_{\Coh}^{-\dim V_d} = H^\bullet(X^{[p]}, \mbb{Q}_\ell)[\sgn] \otimes \mbb{L}_\Coh^{(\dim X +1)(p+1)}. \]

The main point is then to establish:
\begin{proposition}\label{prop.E1-degen}
$E_\bullet$ degenerates at $E_1$.
\end{proposition}

Indeed, assume \cref{prop.E1-degen} holds. Then, the graded for the  filtration on $H^\bullet_c(D_d, \mbb{Q}_\ell) \otimes \mbb{L}_\Coh^{-\dim V_d}$ induced by the spectral sequence in the stable range $\bullet \geq 1 - k$ satisfies, for $0 \leq p \leq k-1$,
\[ \Gr^{p} H^{\bullet}_c(D_d, \mbb{Q}_\ell) \otimes \mbb{L}_\Coh^{-\dim V_d} = H^{\bullet-p}(X^{p+1}, \mbb{Q}_\ell)[\sgn] \otimes \mbb{L}_\Coh^{-(\dim X +1)(p+1)} \]
and is zero for $p\geq k$. Then, using \cref{eq.Ud-from-Dd}, the induced filtration on $H^\bullet_c(U_d, \mbb{Q}_\ell) \otimes \mbb{L}_\Coh^{-\dim V_d}$  in the stable range $\bullet \geq - k$ satisfies for $0 \leq p \leq k$
\[ \Gr^{p} H^\bullet_c(U_d, \mbb{Q}_\ell) \otimes \mbb{L}_\Coh^{-\dim V_d} = H^{\bullet-p}(X^{p}, \mbb{Q}_\ell)[\sgn] \otimes \mbb{L}_\Coh^{-(\dim X+1)(p)} \\
\]
and is zero for $p \geq k+1$. Then, we obtain \cref{theorem.body-stabilization} by comparison with \cref{eq.zeta-cohom-expression} --- the filtration induced by the spectral sequence agrees with the weight filtration by a spreading out argument similar to \cref{lemma.E2-degeneration} below.

Thus it remains to prove \cref{prop.E1-degen}. We do so in the next subsection.

\begin{remark}
The use of \cref{theorem.banerjee-qiso} to simplify the terms of the spectral sequence, as well as the weight arguments below, requires that we work with rational coefficients.
In fact the analog of \cref{theorem.body-stabilization} in singular homology with $\bbZ/2$ coefficients fails for $X = \bbC\bbP^2$, as shown in \cite[Proposition~8.12]{aumonier}.
\end{remark}

\subsection{Homological stabilization --- degeneration of the spectral sequence}\label{ss.stab-degeneration}

We first make a weight argument to show degeneration at $E_2$:

 \begin{lemma}\label{lemma.E2-degeneration}
$E_\bullet$ degenerates at $E_2$.
\end{lemma}
\begin{proof}
We first observe that we may spread $X$ out to a smooth projective $X_0/\Spec A$ for a finitely generated $\mbb{Z}$-algebra $A \subseteq \kappa$ (i.e.\ $X_{0} \times_{\Spec A} \Spec \kappa= X$) in which $\ell$ is invertible. We then may spread out $V_d$ to $V_{d,0}$, and $I_{d}$ to $I_{d,0}$ by the obvious definitions, and similarly to obtain $\overline{V}_{d,0}$ and $\overline{I}_{d,0}$. Then $\Gamma^{\leq k}_{\overline{V}_d}(\overline{I}_d)$ spreads out to $\Gamma^{\leq k}_{\overline{V}_{d,0}}(\overline{I}_{d,0})$. The $E_1$ page of the skeletal spectral sequence for cohomology relative to $\Spec A$ is again quasi-isomorphic to the relative Banerjee complex; in particular, the terms are locally constant sheaves on $\Spec A$. Thus so are the terms on the higher pages, so it suffices to show that the differential $d_r^{p,q}$ for $r \geq 2$ vanishes after specialization to geometric points over closed points of $\Spec A$. Since $A$ is of finite type over $\mbb{Z}$ and $\ell$ is invertible in $A$, any maximal ideal has residue field a finite field $\mbb{F}_a$, $\ell \nmid a$, thus by specializing above such a point we may assume that $\kappa=\overline{\mbb{F}}_a$ and $X_0/\mbb{F}_a$.

In this case, from Deligne's \cite{deligne:weil1} purity theorem  and the expression of the $E_1$ terms above, we find that geometric Frobenius acts on $E_1^{p,q}$ with eigenvalues $a$-Weil integers of weight $q$. The spectral sequence is Galois equivariant, and since $E_r^{p,q}$ is a subquotient of $E_1^{p,q}$, the eigenvalues of geometric Frobenius on $E_r^{p,q}$ are also of weight $q$. Because $d_r$ is Galois-equivariant it preserves generalized eigenspaces for geometric Frobenius, thus because $d_r$  is of degree $(r,1-r)$ it must be identically zero for $r \geq 2$, as claimed.
\end{proof}

It remains to show that the differentials vanish also on $E_1$. To that end, recall that the $E_1$ differential is obtained by restricting the map
\[H^\bullet_c( (I_d / D_d)^{[p-1]}, \mbb{Q}_\ell) \xrightarrow{ \sum_{i \in [p]} (-1)^i \alpha_i^* }H^\bullet_c( (I_d / D_d)^{[p]}, \mbb{Q}_\ell) \]
to the $\sgn$ component, where $\alpha_i$ forgets the $i$th coordinate. We would like to show this map is zero; we will do by computing on a simpler space. Note that we have natural maps of schemes over $X^{[p]}$
\[ (I_d/D_d)^{[p]} \hookrightarrow (\overline{I}_d/\overline{D}_d)^{[p]} \hookrightarrow X^{[p]} \times \overline{V}_d,  \]
and a corresponding set of maps on the configuration loci. If we consider the induced maps on cohomology, we obtain a commutative $\mfS_{[p]}$-equivariant diagram
\begin{equation}\label{eq.E1-diff-diagram}\begin{tikzcd}
	{H^\bullet_c( (I_d / D_d)^{[p]}, \mbb{Q}_\ell)} & {H^\bullet_c( \Conf^{[p]}_{D_d}(I_d), \mbb{Q}_\ell)} \\
	{H^\bullet((\overline{I}_d/\overline{D}_d)^{[p]}, \mbb{Q}_\ell) } & {H^\bullet_c(\Conf^{[p]}_{\overline{D}_d}(\overline{I}_d), \mbb{Q}_\ell)} \\
	{H^\bullet(X^{[p]}, \mbb{Q}_\ell)\otimes H^\bullet(\overline{V}_d, \mbb{Q}_\ell)} & {H^\bullet_c(\Conf^{[p]}(X), \mbb{Q}_\ell) \otimes H^\bullet(\overline{V}_d, \mbb{Q}_\ell)}
	\arrow[from=3-1, to=2-1]
	\arrow[from=3-2, to=2-2]
	\arrow[from=1-2, to=1-1]
	\arrow[from=2-2, to=2-1]
	\arrow[from=3-2, to=3-1]
	\arrow[from=1-1, to=2-1]
	\arrow[from=1-2, to=2-2]
\end{tikzcd}\end{equation}

The top right vertical map in \cref{eq.E1-diff-diagram} is induced by the open embedding of a vector bundle inside its compactifying projective bundle, and the bottom right vertical map is induced by the closed embedding of a projective bundle in an ambient trivial projective bundle. If we expand $H^\bullet(\overline{V}_d, \mbb{Q}_\ell)=\bigoplus_{k=0}^{2\dim V_d} \mbb{L}_\Coh^k$, we deduce that the right column of \cref{eq.E1-diff-diagram} is identified with the subquotient diagram
\[\begin{tikzcd}
	{H^\bullet_c(\Conf^{[p]}(X), \mbb{Q}_\ell) \otimes \mbb{L}_\Coh^{2\dim V_d- (p+1)(\dim X +1)}} \\
	{H^\bullet_c(\Conf^{[p]}(X), \mbb{Q}_\ell)\otimes\bigoplus_{k=0}^{2\dim V_d- (p+1)(\dim X +1)} \mbb{L}_\Coh^k} \\
	{H^\bullet_c(\Conf^{[p]}(X), \mbb{Q}_\ell) \otimes \bigoplus_{k=0}^{2\dim V_d} \mbb{L}_\Coh^k}
	\arrow[two heads, from=3-1, to=2-1]
	\arrow[hook, from=1-1, to=2-1]
\end{tikzcd}\]

The horizontal arrows in \cref{eq.E1-diff-diagram} become isomorphisms after passing to the $\sgn$ component for the action of $\mfS_{[p]}$, so, the $\sgn$ component of the left column of  \cref{eq.E1-diff-diagram}  is then identified with the subquotient diagram
\[\begin{tikzcd}
	{H^\bullet(X^{[p]}, \mbb{Q}_\ell)[\sgn] \otimes \mbb{L}_\Coh^{2\dim V_d- (p+1)(\dim X +1)}} \\
	{H^\bullet(X^{[p]}, \mbb{Q}_\ell)[\sgn]\otimes\bigoplus_{k=0}^{2\dim V_d- (p+1)(\dim X +1)} \mbb{L}_\Coh^k} \\
	{H^\bullet(X^{[p]}, \mbb{Q}_\ell)[\sgn] \otimes \bigoplus_{k=0}^{2\dim V_d} \mbb{L}_\Coh^k}
	\arrow[two heads, from=3-1, to=2-1]
	\arrow[hook, from=1-1, to=2-1]
\end{tikzcd}\]

Moreover, the analogous identifications for $p-1$ are compatible with the differential $\sum_{i \in [p]} (-1)^i \alpha_i^*$. The vanishing of the differential is then immediate, since in the bottom the differential preserves each summand corresponding to a power $\mbb{L}_\Coh^k$, but the summands contributing in degree $p$ and $p-1$ in the top are distinct.

\printbibliography

\end{document}

%% file: theoremnumbering.tex
\numberwithin{equation}{subsubsection}
\theoremstyle{plain}

\newtheorem{maintheorem}{Theorem}

\newtheorem*{theorem*}{Theorem}

\newtheorem{theorem}[subsubsection]{Theorem}
\newtheorem{corollary}[subsubsection]{Corollary}

\newtheorem{proposition}[subsubsection]{Proposition}
\newtheorem{lemma}[subsubsection]{Lemma}

\theoremstyle{definition}

\newtheorem{example}[subsubsection]{Example}
\newtheorem{definition}[subsubsection]{Definition}
\newtheorem{remark}[subsubsection]{Remark}

\theoremstyle{remark}

%% file: macros.tex
\newcommand{\Var}{\mathrm{Var}}

\newcommand{\calM}{\mathcal{M}}

\newcommand{\bbZ}{\mathbb{Z}}
\newcommand{\bbQ}{\mathbb{Q}}

\newcommand{\bbC}{\mathbb{C}}
\newcommand{\bbP}{\mathbb{P}}

\newcommand{\bbF}{\mathbb{F}}

\newcommand{\mfS}{\mathfrak{S}}

\newcommand{\Spec}{\mathrm{Spec}\,}
 
\newcommand{\mot}{\calM}

\NewDocumentCommand{\set}{somm}{%
   \IfNoValueTF{#2}
    {\IfBooleanTF{#1}{\left\{ #3 \mathrel{}\middle\vert\mathrel{} #4 \right\}}{\{#3 \mid #4\}}}
    {\mathopen{#2\{}#3 \mathrel{}#2\vert\mathrel{} #4\mathclose{#2\}}}%
  }